\documentclass[11pt,a4paper]{article}
\usepackage{amssymb, amsthm, amsmath, amsfonts, mathrsfs}
\usepackage{array, epsfig}
\usepackage{bbm}
\usepackage[hidelinks]{hyperref}
\usepackage[numbers,square]{natbib}
\usepackage{color}
\usepackage{enumitem}

  \evensidemargin
\oddsidemargin
\theoremstyle{plain}
\newtheorem{theorem}{Theorem}[section]
\newtheorem{lemma}[theorem]{Lemma}
\newtheorem{corollary}[theorem]{Corollary}
\newtheorem{proposition}[theorem]{Proposition}

\theoremstyle{definition}
\newtheorem{definition}[theorem]{Definition}
\newtheorem{example}[theorem]{Example}

\theoremstyle{remark}
\newtheorem{remark}[theorem]{Remark}

\DeclareMathOperator\supp{supp}

\newcommand{\dist}{\mathop{\mathrm{dist}}\nolimits}

\def\BB{\mathbb{B}}

\def\EE{\mathbb{E}}
\def\FF{\mathbb{F}}

\def\NN{\mathbb{N}}

\def\PP{\mathbb{P}}
\def\QQ{\mathbb{Q}}
\def\RR{\mathbb{R}}
\def\SS{\mathbb{S}}
\def\TT{\mathbb{T}}

\def\cB{\mathcal{B}}
\def\cC{\mathcal{C}}

\def\cF{\mathcal{F}}
\def\cG{\mathcal{G}}

\def\cL{\mathcal{L}}

\def\cN{\mathcal{N}}

\def\cP{\mathcal{P}}

\def\cT{\mathcal{T}}

\def\cV{\mathcal{V}}

\newcommand{\dd}{{\rm d}}
\newcommand{\conv}{\mathop{\mathrm{conv}}\nolimits}
\newcommand{\frI}{\mathop{\mathrm{I}}\nolimits}
\newcommand{\frD}{\mathop{\mathrm{D}}\nolimits}

\newcommand{\aff}{\mathop{\mathrm{aff}}\nolimits}
\newcommand{\lin}{\mathop{\mathrm{lin}}\nolimits}
\newcommand{\pow}{\mathop{\mathrm{pow}}\nolimits}

\newcommand{\proj}{\operatorname{proj}}
\newcommand{\inter}{\operatorname{int}}

\newcommand{\Vol}{\operatorname{Vol}}

\newcommand{\bd}{\operatorname{bd}}
\newcommand{\PPP}{\operatorname{PPP}}

\newcommand{\dint}{\textup{d}}

\newcommand{\apex}{\mathop{\mathrm{apex}}\nolimits}

\newcommand{\cl}{\mathop{\mathrm{cl}}\nolimits}

\DeclareMathOperator*{\argmin}{arg\,min}

\newcommand{\bsl}{\backslash}

\makeatletter
\let\@fnsymbol\@alph
\makeatother

\begin{document}

\title{\bfseries Poisson-Laguerre tessellations}

\author{Anna Gusakova\footnotemark[1],\; Mathias in Wolde-L\"ubke\footnotemark[2]}

\date{}
\renewcommand{\thefootnote}{\fnsymbol{footnote}}
\footnotetext[1]{M\"unster University, Germany. Email: gusakova@uni-muenster.de}
\footnotetext[2]{M\"unster University, Germany. Email: miwluebke@uni-muenster.de}

\maketitle

\begin{abstract}
    In this paper we introduce a family of Poisson-Laguerre tessellations in $\RR^d$ generated by a Poisson point process in $\RR^d\times \RR$, whose intensity measure has a density of the form $(v,h)\mapsto f(h)\dint h\dint v$, where $v\in\RR^d$ and $h\in\RR$, with respect to the Lebesgue measure. We study its sectional properties and show that the $\ell$-dimensional section of a Poisson-Laguerre tessellation corresponding to $f$ is an $\ell$-dimensional Poisson-Laguerre tessellation corresponding to $f_{\ell}$, which is up to a constant a fractional integral of $f$ of order $(d-\ell)/2$. Further we derive an explicit representation for the distribution of the volume weighted typical cell of the dual Poisson-Laguerre tessellation in terms of fractional integrals and derivatives of $f$. \\

\end{abstract}

\tableofcontents

\section{Introduction}

A random tessellation in $\RR^d$ is one of the classical and central models of stochastic geometry. A random tessellation is a locally finite covering of the space by convex polytopes with non-empty and disjoint interiors, which is typically generated by some Poisson process of simpler geometric objects (like points or hyperplanes). The study of random tessellations and the particular interest in these models are motivated both by their rich inner mathematical structures and by the number of applications in which such configurations naturally arise. For example, tessellations, and especially triangulations of a space, play a prominent role in finite element methods in numerical analysis and computer vision, data analysis, network modeling, astrophysics, and computational geometry (see \cite{AKL, BY98, CSKM, OBKC} and references therein). At the same time tessellations have proven to be a promising model for polycrystalline materials, plant cells, crack patterns or foam structures (see \cite{CSKM, OBKC} and references therein). Moreover, recently there were a few articles using random tessellations in machine learning \cite{ORT24, machinelearningtess}. 

One of the difficulties of using random tessellation in applications is related to the fact, that there are very few models for which rigorous results and exact formulas for basic mean characteristics are available. Two models of random tessellations, which appeared to be mathematically tractable and which have a number of applications, are the Poisson-Voronoi tessellation and its dual model, the Poisson-Delaunay tessellation. The construction of the Poisson-Voronoi tessellation may be described as follows. Let $\eta$ be a homogeneous Poisson point process in $\RR^d$ and for any point $v\in \eta$ we define its Voronoi cell as the set of all points $w\in \RR^d$, which are closer to $v$ then to any other point of the process $\eta$, namely
\[
    C(v,\eta):=\{w\in\RR^d\colon \|v-w\|\leq \|v'-w\| \text{ for all }v'\in\eta\}.
\]
Almost surely each Voronoi cell is a closed convex polytope and the collection of all Voronoi cells forms a Poisson-Voronoi tessellation. It is often useful to think about a Voronoi cell $C(v,\eta)$ as a crystal, which has started to grow at point $v$ and is growing with the same speed in all directions. The point $w$ belongs to the crystal, which has reached this point first. Together with the Poisson-Voronoi tessellation its dual model, the Poisson–Delaunay tessellation, is often considered. It can be obtained using the following procedure. We connect two distinct points $v_1,v_2\in\eta$ by an edge if and only if $C(v_1,\eta)\cap C(v_2,\eta)\neq \emptyset$. The resulting graph splits the space into convex polytopes, each of them is a simplex almost surely, giving rise to the tessellation called Poisson-Delaunay tessellation. There is a long list of literature devoted to the Poisson-Voronoi and Poisson-Delaunay tessellations, but despite the long history of study there are still some open questions related to the models above (see i.e. \cite{CSKM} and \cite[Chapter 10]{SW} for an overview). 

There are a few possibilities to generalize the construction above by considering the set $A$ of points $(v,h)\in \RR^d\times \RR$ and using weights $h$ to modify the definition of the Voronoi cell (see \cite[Section 3.1]{OBKC}). One of them is to use the so-called power function instead of the distance, namely
\[
    \pow\left(w,(v,h)\right) := \|w-v\|^2 + h,\qquad w\in\RR^d,
\]
leading to the definition of the Laguerre diagram. More precisely for any $(v,h)\in A$ we define its Laguerre cell as 
\[
    C((v,h), A) := \big\{w \in \RR^d: \pow(w,(v,h)) \le \pow(w,(v',h')) \text{ for all } (v',h') \in A\big\}, 
\]
which is again a convex and closed set. In terms of crystallization processes we may think of the set $C((v,h), A)$ as a crystal, which has started to grow at point $v\in\RR^d$ and at time $h\in\RR$, and is growing with the same speed in all directions, while the speed of growth decreases with time. As before the point $w\in\RR^d$ belongs to the crystal which has reached this point first. The crucial difference between the Voronoi and Laguerre cell is the fact that the Laguerre cell might be empty. This motivates the definition of the Laguerre diagram $\cL_d(A)$ as a collection of all non-empty Laguerre cells. The definition of the dual Laguerre diagram $\cL^*_d(A)$ is analogous to the definition of the Poisson-Delaunay tessellation as the dual of the Poisson-Voronoi tessellation (see Section \ref{sec:Construction} for the details).

Random Laguerre tessellations have been considered in the literature in a few special settings. In \cite{Ldoc,LZ08} the authors consider a Laguerre tessellation of an independent $\QQ$-marking $\xi$ of a homogeneous Poisson point process in $\RR^d$, where $\QQ$ is a probability measure on $(-\infty,0]$ satisfying some natural integrability assumptions (see Example \ref{ex:IndQMarking}). The number of explicit formulas including the description of the distribution for the typical cell of the dual model $\cL^*_d(\xi)$ can be found in \cite{Ldoc}. Another approach was used in \cite{GKT20, GKT21}, where as a set of generating points $A$ of a Laguerre tessellation the authors considered a Poisson point process $\eta$ on the product space $\RR^d\times E$, where $E$ is some possibly unbounded interval, with intensity measure $\Lambda$ of the form
\begin{equation}\label{eq:IntensityMeasure}
    \Lambda(\cdot)=\gamma\int_{\RR^d}\int_{E}f(h){\bf 1}((v,h)\in\cdot)\,\dint h\,\dint v,\qquad \gamma>0.
\end{equation}
More precisely in \cite{GKT20} the construction is based on the functions $f(h)={\rm const}\cdot h^{\beta}{\bf 1}(h\ge 0)$, $\beta>-1$, and $f(h)={\rm const}\cdot (-h)^{-\beta}{\bf 1}(h<0)$, $\beta>{d\over 2}+1$, which lead to the so-called $\beta$-Voronoi and $\beta'$-Voronoi tessellations, respectively, while in \cite{GKT21} the function $f(h)=e^{\lambda h}$, $\lambda>0$, has been considered leading to the definition of the Gaussian-Voronoi tessellation. These three families of random tessellations appeared to be well-tractable due to the connection of the distribution of the typical cell of the corresponding dual models with $\beta$-, $\beta'$- and Gaussian random simplices, which are well-studied models of random polytopes \cite{Miles_IRS, KZT20}. Further in \cite{sectional} intersections of $\beta$-, $\beta'$- and Gaussian-Voronoi tessellations with an affine subspace of dimension $1\leq \ell\leq d-1$ have been studied. In particular, it was shown that the intersection of a $d$-dimensional Poisson-Voronoi tessellation with an $\ell$-dimensional affine subspace has the same distribution as an $\ell$-dimensional $\beta$-Voronoi tessellation with $\beta=\frac{d-\ell}{2}-1$. The latter in combination with the properties of $\beta$-Voronoi tessellations lead to the new formulas for the expected intrinsic volumes of the typical Poisson-Voronoi cell.

In this article we aim to study a general model for random (Poisson)-Laguerre tessellation, focusing on the Laguerre tessellations $\cL_{d,\gamma}(f):=\cL_d(\eta)$ induced by the Poisson point process $\eta$ with intensity measure $\Lambda$ given by \eqref{eq:IntensityMeasure}, where $f$ is a general locally integrable and non-negative function. It should be noted that in the case when the probability measure $\QQ$ is absolutely continuous with respect to the Lebesgue measure and has density $q$, the point process $\xi$ is itself a Poisson point process on $\RR^d\times (-\infty,0]$ with intensity measure of the form \eqref{eq:IntensityMeasure}, where $f=q$. Consequently, our setting partially includes those studied in \cite{Ldoc, LZ08}. We will show that if $f$ is locally integrable on $\RR$, as in the cases of the $\beta$-Voronoi and Gaussian-Voronoi tessellations, the resulting construction $\cL_{d,\gamma}(f)$ is a random tessellation, provided that $f$ satisfies a mild and natural integrability assumption. Conversely, when $f$ is not locally integrable on $\RR$ but only on an open half-line $(-\infty, b)$, $b\in\RR$, like for example in the $\beta'$-Voronoi case, additional conditions on $f$ are required to ensure that the resulting construction is indeed a locally finite covering of $\RR^d$ by convex polytopes almost surely. It should be mentioned that while the proof in the first case follows a similar approach as in previous works, the second case is more intricate and requires new techniques. Next we will study the sectional properties of $\cL_{d,\gamma}(f)$ and show that its intersection with an $\ell$-dimensional affine subspace has the same distribution as $\cL_{\ell,\gamma}(f_{\ell})$, where $f_{\ell}$ is (up to a constant) a fractional integral of $f$ of order $(d-\ell)/2$ (Theorem \ref{thm:sectional}). This result generalizes \cite[Theorem 4.1]{sectional}. While our proof follows similar ideas, it is formulated in a different framework. A key challenge (which did not appear in \cite[Theorem 4.1]{sectional}) is to ensure that the sectional tessellation $\cL_{\ell,\gamma}(f_{\ell})$ is again a random tessellation. As before, handling the case, when $f$ is not locally integrable on $\RR$ requires 
the most of efforts. Additionally, for sufficiently good functions $f$, we derive an explicit representation for the distribution of the typical cell of the dual tessellation $\cL^*_{d,\gamma}(f)$ in terms of the function $f$ and its fractional integrals and derivatives (Theorem \ref{thm:typcell}). This representation extends \cite[Theorem 4.5]{GKT20} and \cite[Theorem 5.1]{GKT21}, and partially generalizes \cite[Theorem 3.3.1]{Ldoc}. The application of Theorem \ref{thm:typcell} requires verifying complicated integrability conditions on $f$, which for certain families of function are established in Proposition \ref{prop:FiniteNormalizationConst}. 
Finally, we conclude with a discussion of the canonical decomposition of the typical cells of the $\beta$-, $\beta'$- and Gaussian-Delaunay tessellations, which might be seen as a characterization of these three families. 

The rest of the paper is structured as follows. In Section 2 we collected some frequently used notations as well as basic facts about random tessellations and fractional calculus. Section 3 is devoted to the construction of Poisson-Laguerre tessellations and their dual tessellations, as well as to the study of their basic properties. In Section 4 we consider sectional properties of Poisson-Laguerre tessellations. Finally, in Section 5 we deal with the dual Poisson-Laguerre tessellation and study the distribution of its typical cell. 

\section{Preliminaries}

\subsection{Frequently used notation}

Given a set $A\subseteq\RR^d$ we denote by $\inter A$, $\cl A$ and $\bd A$ the interior, closure and boundary of $A$, respectively. In the case of a countable set $A$ we denote by $\# A$ its cardinality. Moreover for any $y\in\RR^d$ we write $A+y:=\{x+y\colon x\in A\}$ and for any $s\in \RR$ we set $sA:=\{sx\colon x\in A\}$. Further we write $\aff(A)$, $\lin(A)$ and $\conv(A)$ to denote the affine, linear and convex hull of $A$, respectively. Given $y_1,\ldots,y_{d+1}\in\RR^d$ we write $\Delta_{d}(y_1,\dots,y_{d+1})$ for the $d$-dimensional volume of the simplex $\conv(y_1,\ldots,y_{d+1})$ and $\nabla_{d}(y_1,\dots, y_d)$ for the $d$-dimensional volume of the polytope spanned by the vectors $y_1,\ldots,y_d$.

A closed Euclidean ball in $\RR^d$ with radius $r>0$ centered at $0$ is denoted by $\BB^d(r)$ and we set $\BB^d:=\BB^d(1)$. By $\sigma_{d-1}$ we denote the spherical Lebesgue measure on the $(d-1)$-dimensional unit sphere $\SS^{d-1}$, normalized in such a way that
$$
    \omega_d:=\sigma_{d-1}(\SS^{d-1})={2\pi^{d\over 2}\over \Gamma\left({d\over 2}\right)}.
$$

We denote by $\overline{\RR}=\RR\cup\{-\infty,\infty\}$ the extended system of real numbers, and by $\RR_+$ the set of non-negative real numbers. We will use the usual measure theoretical conventions: $0\cdot (\pm\infty):=0$, $c\cdot(\pm\infty)=(\pm\infty)\pm c:=\pm\infty$ for $c\in (0,\infty)$ and $e^{-\infty}:=0$. 

In what follows we shall represent points $x\in\RR^{d+1}$ in the form $x=(v,h)$ with $v\in\RR^{d}$ (called \textit{spatial} coordinate) and $h\in\RR$ (called \textit{height}, \textit{weight} or \textit{time} coordinate). Denote by ${\rm Ref}\colon\RR^{d+1}\to\RR^{d+1}$ a reflection map with respect to space hyperplane, namely ${\rm Ref}(v,h)=(v,-h)$, and for given $c\in\RR$ by $\tau_c\colon\RR\to\RR$ we denote the shift map, namely $\tau_c(x)=x+c$.

Let $\Pi$ (respectively, $\Pi^+$) be the standard downward  (respectively, upward) paraboloid, defined as
\begin{align*}
    \Pi&:=\{(v,h)\in\RR^{d+1}\colon h=-\|v\|^2\},\qquad \Pi^+:={\rm Ref}(\Pi)=\{(v,h)\in\RR^{d+1}\colon h=\|v\|^2\}.
\end{align*}
and let $\Pi_{(w,t)}$ be the translation of $\Pi$ by a vector $(w,t)\in\RR^{d+1}$, that is,
\[
    \Pi_{(w,t)}:=\{(v,h)\in\RR^{d}\times\RR\colon h=-\|v-w\|^2+t\}.
\]
The point $(w,t)$ is called the apex of the paraboloid $\Pi_{(w,t)}$ and is denoted by $\apex\Pi_{(w,t)}$. Given points $x_i=(v_i,h_i)$ for $i =1,\dots, d+1$ with affinely independent spatial coordinates $v_1,\ldots,v_{d+1}$ we denote by $\Pi(x_1,\dots,x_{d+1})$ the unique translation of the downward paraboloid $\Pi$ containing $x_1,\dots,x_{d+1}$. Given a set $A\subseteq \RR^{d+1}$ we define the hypograph and epigraph of $A$ as
\begin{align*}
    A^{\downarrow}:&=\{(v,h')\in\RR^{d}\times\RR\colon (v,h) \in A \text{ for some } h\ge h'\},\\
    A^{\uparrow}:&=\{(v,h')\in\RR^{d}\times\RR\colon (v,h) \in A \text{ for some } h\leq h'\}.
\end{align*}

For $k \in \{0,\ldots,d\}$ let $G(d,k)$ be the set of all $k$-dimensional linear subspaces of $\RR^d$. The unique rotation invariant probability measure on $G(d,k)$ is denoted by $\nu_k$. Given $L\in G(d,k)$ we denote by $L^{\perp}$ its orthogonal complement. Analogously, $A(d,k)$ denotes the $k$-dimensional affine subspaces of $\RR^d$ and it admits the unique (up to normalization) rigid motion invariant measure $\mu_k$, defined as
\[
    \mu_k(\cdot)=\int_{G(d,k)}\int_{L^\bot}{\bf 1}(L+x\in \cdot)\lambda_{L^\bot}(\dd x) \nu_k(\dd L).
\]

\subsection{Random tessellation}\label{sec:Rtessel}

In this subsection we recall the concept of a random tessellation and include a brief overview of basic properties. For a more detailed discussion we refer the reader to \cite[Chapter 10]{SW}. According to \cite[Definition 10.1.1]{SW} and \cite[Lemma~10.1.1]{SW} a tessellation (or a mosaic) $T$ in $\RR^d$ is a locally finite system of convex polytopes that cover the whole space and have non-empty disjoint interiors. The elements of $T$ are called {\textit{cells}}. 

Given a polytope $P$ we denote by $\cF_k(P)$, $k\in\{0,1,\ldots,d\}$, the set of its $k$-dimensional faces, where $\cF_d(P)=\{P\}$, and let $\cF(P):=\bigcup_{k=0}^{d}\cF_{k}(P)$. A tessellation $T$ is called {\textit{face-to-face}} if for all $P_1,P_2\in T$ we have
\[
    P_1\cap P_2\in(\cF(P_1)\cap\cF(P_2))\cup\{\emptyset\}.
\]
In the case of a face-to-face tessellation we also set $\cF_k(T):=\bigcup_{t\in T}\cF_k(t)$ for $k\in \{0,1,\dots,d\}$.

A face-to-face tessellation in $\RR^d$ is called {\textit{normal}} if each $F\in\cF_k(T)$ is contained in precisely $d+1-k$ cells for all $k\in\{0,1,\ldots,d-1\}$. We denote by $\TT$ the set of all face-to-face tessellations in $\RR^d$. By a {\textit{random tessellation}} in $\RR^d$ we understand a particle process $X$ in $\RR^d$ (see \cite[Section~4.1]{SW}) satisfying $X\in\TT$ almost surely.

\subsection{Fractional integrals and derivatives}\label{sec:FractionalIntegrals}
In what follows it will often be convenient to use common notations from fractional calculus. Here we will only collect the basic definitions and facts used in this article, while for more detailed discussions we refer the reader to \cite{samko1993fractional}, \cite{miller1993introduction} and \cite{TAFDE}.

Let $-\infty \le a < b \le \infty$ and $\alpha >0$. For any measurable function $f\colon (a,b)\to \RR_+$ the integral
\[
    \big(\frI_{a+}^{\alpha}f\big)(x):=\frac{1}{\Gamma\left(\alpha\right)} \int_{a}^{x} f(t) (x-t)^{\alpha-1}\dd t,\qquad x\in (a,b)
\]
is called \textit{fractional integral of order $\alpha$}. Note that for any non-negative $f$ and $\alpha\ge 1$ the function $\frI_{a+}^{\alpha}f$ is monotonically increasing and for any $c\in \RR$ we have
\begin{equation}\label{eq:FractionalIntegralShift}
    \big(\frI_{a+}^{\alpha}f\big)(x)
    =\big(\frI_{(a-c)+}^{\alpha}f\circ \tau_c\big)(x-c).
\end{equation}
Thus, we can restrict ourselves to the cases $a=0$ and $a=-\infty$. For $a=0$ the expression $\frI_{0+}^{\alpha}f$ is known as the \textit{Riemann-Liouville fractional integral}, while in the case $a=-\infty$ 
\begin{equation}\label{eq:Frac_Int}
    \big(\frI^{\alpha}f\big)(x):=\big(\frI^{\alpha}_{-\infty}f\big)(x)=\frac{1}{\Gamma\left(\alpha\right)} \int_{-\infty}^{x} f(t) (x-t)^{\alpha-1}\dd t,
\end{equation}
is often referred to as the \textit{Liouville version}. In the case $a=-\infty$ and $b<\infty$ we again use \eqref{eq:FractionalIntegralShift} with $c=b$ and, thus, without loss of generality may assume that $b=0$ in this case. Further for any $\alpha>0$ we have the following relation between fractional Riemann-Liouville $\frI_{0+}^{\alpha}$ and Liouville $\frI^{\alpha}$ integrals. For any $x<0$ by applying the change of variables $t=-u^{-1}$ we have
    \begin{equation}\label{eq:frac_f_g}
        \begin{aligned}
            \big(\frI^{\alpha}f\big)(x)&=\frac{1}{\Gamma(\alpha)}\int_{-\infty}^{x}f(t)(x-t)^{\alpha-1} \dd t\\
            &=(-x)^{\alpha-1}\frac{1}{\Gamma(\alpha)}\int_{0}^{-\frac{1}{x}}u^{-\alpha-1}f\left(-1/u\right)\left(-1/x-u\right)^{\alpha-1}\dd u\\
            &=(-x)^{\alpha-1}\big(\frI_{0+}^\alpha f_\alpha\big)\left(-1/x\right),    
        \end{aligned}
    \end{equation} 
    where $f_{\alpha}(u)=u^{-\alpha-1}f(-1/u)$, which is a function defined on $(0,\infty)$. 

\begin{remark}\label{rem:convention}
    It will often be convenient to assume that the function $f$ is defined on whole $\RR$ by setting $f(x)=0$ if $x\not\in (a,b)$. The value of the fractional integral $\frI_{a+}^{\alpha}f$ for $x>a$ will not change in this case and for $x\leq a$ we will have $ \big(\frI_{a+}^{\alpha}f\big)(x)=0$. Thus, if the dependence of $a$ is not essential we will use the Liouville version of the fractional integral $\frI^{\alpha}f$ in order to simplify the notation. 
    \end{remark}
    
    Using the above convention, we may define the fractional integral in terms of the  convolution of functions $f$ and $p(t)=t^{\alpha-1}{\bf 1}_{[0,\infty)}(t)$, namely
\begin{align*}
    \big(\frI^{\alpha}f\big)(x)={1\over \Gamma(\alpha)}(f*p)(x)={1\over \Gamma(\alpha)}\int_{-\infty}^{\infty}f(t)p(x-t)\dint t.
\end{align*}
In what follows we will be interested in the following class 
\[
    L_{{\rm loc}}^{1,+}(E):=\Big\{f\colon \inter E\to \RR_+ \colon \int_Kf(x)\dd x<\infty \, \forall K\subset E, \, K\text{ compact}\Big\},\qquad E\in\cB(\RR),
\] 
of non-negative functions which are locally integrable on $E$. In particular using the usual equivalence relation we will write $f=g$ for $f,g\in L_{{\rm loc}}^{1,+}(E)$ if $f(x)=g(x)$ in almost every point $x\in E$. It was shown in \cite[Lemma 2.1.]{MARTINEZ1992111} that for any $\alpha>0$ and $f\in L_{{\rm loc}}^{1,+}([a,\infty))$, where $a>-\infty$, we have $(\frI_{a+}^{\alpha}f)(x)<\infty$ at almost every $x>a$ and that $\frI_{a+}^{\alpha}f$ is locally integrable on $[a,\infty)$. Due to relation \eqref{eq:frac_f_g} we also conclude that $(\frI^{\alpha}f)(x)<\infty$ for almost every $x<0$ and any $\alpha>0$ if $f\colon(-\infty,0)\to[0,\infty)$ is such that $f_{\alpha}\in L_{{\rm loc}}^{1,+}([0,\infty))$. Also note that if $f_{\alpha}\in L_{{\rm loc}}^{1,+}([0,\infty))$, then $f_{\beta}\in L_{{\rm loc}}^{1,+}([0,\infty))$ for any $\beta<\alpha$.

In what follows we will often use the following fact. 

\begin{lemma}\label{lm:technical}
    If $f\in L_{\rm loc}^{1,+}((-\infty,b))$, where $-\infty<b\le \infty$, and $(\frI^{\alpha}f)(p)<\infty$ for some $p<b$ and $\alpha>1$, then $(\frI^{\beta}f)(p)<\infty$ for any $\beta\in [1,\alpha]$. 
\end{lemma}

\begin{proof}
    Let $p<b$. Since $f$ is non-negative and locally integrable on $(-\infty,b)$ we have
    \[
        \Gamma(\beta)(\frI^{\beta}f)(p)=\int_{-\infty}^{p}f(t)(p-t)^{\beta-1}\dd t\leq \int_{-\infty}^{p-1}f(t)(p-t)^{\alpha-1}\dd t+\int_{p-1}^pf(t)\dd t<\infty.
    \]
\end{proof}

Further it shall be noted that for a non-negative measurable function $f$ and $\alpha, \beta > 0$ the semigroup property
\begin{equation}\label{eq:SemigroupProp}
    \frI_{a+}^\alpha \frI_{a+}^\beta f=\frI_{a+}^\beta \frI_{a+}^\alpha f = \frI_{a+}^{\alpha+\beta} f
\end{equation}
holds by Tonelli's theorem.

 As a counterpart to the fractional integral, the \textit{fractional derivative} of order $\alpha > 0$ is defined as
\[
    \frD_{a+}^\alpha f(x) := \frac{\dd^n}{\dd x^n} \frI_{a+}^{n-\alpha}f(x)=\frac{1}{\Gamma(n-\alpha)}\frac{\dd^n}{\dd x^n} \int_a^x f(t) (x-t)^{n-\alpha-1} \dd t,
\]
where $n=\lceil\alpha\rceil$ and $\lceil\alpha\rceil$ is the smallest integer greater than or equal to $\alpha$. Additionally we set $\frI_{a+}^0f=\frD_{a+}^0f=f$. Using the semigroup property for any $x>a$ and $\alpha>0$ we have
\[
    \frD_{a+}^\alpha \frI_{a+}^\alpha f (x)={\frac{\dd^n}{\dd x^n}}(\frI_{a+}^nf)(x)={\frac{\dd^{n-1}}{\dd x^{n-1}}}\Big(\frac{\dd}{\dd x}\int_{a}^{x}(\frI_{a+}^{n-1}f)(t)\dd t\Big),
\]
where $n=\lceil\alpha\rceil$. Further note that if $(\frI^1_{a+}f)(p)=\int_{a}^pf(t)\dd t<\infty$ for any $p<b$, then $\frI^1_{a+}f$ is continuous on $(a,b)$ and for almost all $x\in(a,b)$ 
\[
    \frac{\dd}{\dd x}(\frI^1_{a+}f)(x)=f(x)
\]
(see \cite[Proposition 6.3, Theorem 6.5]{Analysis}). Assuming $(\frI_{a+}^nf)(p)<\infty$ for any $p<b$ and $f\in L_{\rm loc}^{1,+}((-\infty,b))$ by Lemma \ref{lm:technical} we have $(\frI_{a+}^mf)(p)<\infty$ for any $1\leq m\leq n$ and $p<b$.
Thus, for almost all $x\in (a,b)$ we get
\[
    {\frac{\dd^n}{\dd x^n}}(\frI_{a+}^nf)(x)=\frac{\dd^{n-1}}{\dd x^{n-1}}(\frI_{a+}^{n-1}f)(x)=\ldots=\frac{\dd}{\dd x}(\frI^1_{a+}f)(x)=f(x),
\]
which leads to 
\begin{equation}\label{eq:DiffIntInverse}
     \frD_{a+}^\alpha \frI_{a+}^\alpha f (x) =  f(x).
\end{equation}
In particular the equality holds if $f$ is continuous in $x$. Hence the fractional differentiation is an operation inverse to the fractional integration from the left.

\section{Construction of Poisson-Laguerre tessellation}\label{sec:Construction}

In this section we will introduce the random Laguerre tessellation, which is a generalized (weighted) version of the well-known \textit{Poisson-Voronoi tessellation}. More precisely we will be interested in tessellations in $\RR^d$, whose construction is based on a set of points of the form $(v,h)$, $v\in\RR^d$ (spatial coordinate), $h\in \RR$ (weight), which are given by an inhomogeneous Poisson point process in $\RR^d\times\RR$. 

\subsection{Definition of (deterministic) Laguerre tessellation}\label{sec:Ltessel}

We start by introducing the deterministic construction of a Laguerre tessellation using two different approaches. The first approach arises as a generalization of the definition of a Voronoi tessellation for a set of points with weights, while the second is used in computational geometry and relies on a vertical projection of a $(d+1)$-dimensional polyhedral set to $\RR^d$ as in \cite{BY98}. 

\subsubsection{Definition via power function} \label{sec:ConstructionStandard}

Given two points $v, w \in \RR^d$ and $h \in \RR$ we define the \textit{power} of $w$ with respect to the pair $(v,h)$ as
\begin{align}
    \pow\left(w,(v,h)\right) := \|w-v\|^2 + h.\label{eq:PowerFunction}
\end{align} 
As before $h$ is referred to as the weight of $v$. Let $A$ be a countable set of points of the form $(v,h) \in \RR^d\times \RR$. We define the \textit{Laguerre cell} of 
$(v,h) \in A$ as
\begin{align*}
    C\left((v,h), A\right) := \big\{w \in \RR^d: \pow(w,(v,h)) \le \pow(w,(v',h')) \text{ for all } (v',h') \in A\big\}. 
\end{align*} 
The collection of all Laguerre cells of $A$, which have non-vanishing topological interior, is called the \textit{Laguerre diagram}
\[
    \cL_d(A):=\{C((v,h),A)\colon (v,h)\in A,\, \inter{C((v,h),A)}\neq\emptyset \}.
\]
If all points $(v,h)\in A$ have the same weight $h\equiv h_0$, the diagram $\cL_d(A)$ coincides with the \textit{Voronoi diagram} of the set $\{v\colon (v,h)\in A\}$. Intuitively, we can think of $\cL_d(A)$ as a result of a crystallization process. More precisely, for a point $(v,h)\in A$ let $v\in \RR^d$ be the location and $h\in \RR$ be the time at which a crystal starts to grow. The evolution of the given crystal $(v,h)$ is described as follows, by any time $t>h$ the crystal covers the ball of radius $\sqrt{t}$ around $v$. For a generic point $w\in \RR^d$ the value $\pow(w,(v,h))$ is the time 
 which the crystallization process needs to reach the point $w$. Hence, $C((v,h),A)$ is the collection of points, which are reached by the crystal starting at $v$ and at time $h$ before they are reached by any other crystal.

 \subsubsection{Definition via vertical projection of \texorpdfstring{$(d+1)$}{}-dimensional polyhedral set}\label{sec:polyhedral}

Let us now introduce an alternative approach to the construction of the Laguerre diagram $\cL_d(A)$ based on methods from computational geometry \cite{BY98}. We will use this construction in the proof of Theorem \ref{thm:sectional}. 

As before let $A$ be a countable set of points $(v,h) \in \RR^d \times \RR$. We will identify a point $(v,h) \in A$ with the $d$-dimensional sphere $S(v,\sqrt{-h})$ with center $v$ and (potentially imaginary) radius $\sqrt{-h}$. We call two spheres $S(v_1,r_1)$, $S(v_2,r_2)$ orthogonal if $\pow(v_2,(v_1,-r_1^2))=r_2^2$ or equivalently $\pow(v_1,(v_2,-r_2^2))=r_1^2$. Consider the transformation $\phi \colon S(v,r) \mapsto (v, \|v\|^2 - r^2)$ which maps $(d-1)$-dimensional spheres to points in $\RR^d \times \RR$. Interpreting points in $\RR^d$ as spheres of radius $0$ we get that the image of $\RR^d$ under $\phi$ is the standard upward paraboloid $\Pi^+$ and in this case we often write $\phi(w):=\phi(S(w,0))$. It should be noted that for arbitrary $v\in\RR^d$ the preimage under $\phi$ of the vertical line in $\RR^{d+1}$ passing through $(v,0)$ is the set of spheres centered at $v$. Spheres with real radius are mapped to points from $(\Pi^+)^{\downarrow}$ (lying below $\Pi^+$), while spheres with imaginary radius correspond to points from $(\Pi^+)^{\uparrow}$ (lying above $\Pi^+$). For $x \in \RR^{d+1}$ we define the polar hyperplane of $x$ with respect to the quadric $\Pi^+$ as
\begin{align*}
    x^\circ_{\Pi^+} := \left\{y \in \RR^{d+1}: (x_1,\dots , x_{d+1}, 1) \Delta_{\Pi^+} (y_1, \dots, y_{d+1}, 1)^T = 0\right\}, 
\end{align*}
where
\begin{align*}
    \Delta_{\Pi^+} := 
        \begin{pmatrix}
        I_d & 0 & 0\\
        0 & 0 & -1/2\\
        0 & -1/2 & 0
        \end{pmatrix}.
\end{align*}
We will often write simply $x^{\circ}$ if it is clear from the context with respect to which quadric polarity is considered. 
Then \cite[Lemma~17.2.1]{BY98} implies that the set of spheres orthogonal to a given sphere $S(v,r)$ is mapped by $\phi$ to $\phi(S(v,r))^o$. Furthermore, \cite[Lemma~17.2.3]{BY98} shows that $\pow(w,(v,-r^2))$, where $w\in \RR^d$, is the signed vertical distance between the point $\phi(w)$ and the hyperplane $\phi(S(v,r))^\circ$, i.e.\ the difference of the $(d+1)$-th coordinates of $\phi(w)$ and the unique point $y\in\phi(S(v,r))^o$ with $(y_1,\dots,y_d)=w$. Indeed,
\begin{align*}
    \phi\left(S\left(v,r\right)\right)^\circ=\left\{(y,y_{d+1}) \in \RR^d \times \RR: y_{d+1}=\| y \|^2 -\| v-y \|^2 +r^2\right\}
\end{align*}
yields that the signed vertical distance between $\phi(w)=(w,\|w\|^2)$ and $\phi(S(v,r))^\circ$ equals
\begin{align*}
    \|w\|^2 - (\| w \|^2 -\| v-w \|^2 +r^2) = \| v-w \|^2 -r^2 = \pow(w,(v,-r^2)).
\end{align*}

We now define an unbounded convex closed set $P(A)$ as the intersection of the epigraphs of all polar hyperplanes 
$\phi(S(v,\sqrt{-h}))^{\circ}$, where $(v,h) \in A$, i.e.
\begin{align}\label{def:polyhedralset}
    P(A):=\bigcap_{(v,h)\in A}\big[\phi(S(v,\sqrt{-h}))^{o}\big]^{\uparrow}.
\end{align}
Note that in the case when $A$ is locally finite $P(A)$ is a polyhedral set. For a facet (i.e. $d$-dimensional face) $F$ of $P(A)$ (see \cite[Section 1.4 and Section 2.1]{Schneider} for a definition) we denote by $\proj_{\RR^d}(F)$ the orthogonal projection of $F$ to $\RR^d$. We show that the collection of non-empty Laguerre cells coincides with
\[
    \mathcal{P}_d(A):=\{\proj_{\RR^d}F:F\text{ is a facet of }P(A)\}.
\]
By definition and \cite[Theorem 2.1.2]{Schneider} we have that $\mathcal{P}_d(A)$ is a collection of convex sets with non-vanishing and non-intersecting topological interiors. Let $(v,h)\in A$ and without loss of generality let $F(v,h):=\phi(S(v,\sqrt{-h}))^\circ \cap \bd P(A) \neq \emptyset$ (otherwise the corresponding Laguerre cell $C((v,h),X)$ is empty and does not influence the Laguerre diagram). Note in this case $F(v,h)$ is a support set and, hence, is a face of $P(A)$.  Let $\proj_{\RR}:\RR^{d+1} \to \RR, (w,r)\mapsto r$, be the projection onto the $(d+1)$-th coordinate. By the definition of $P(A)$ we have 
\begin{align*}
    F(v,h)&=\{(w,t)\in \RR^d \times \RR: (w,t) \in \phi(S(v,\sqrt{-h}))^\circ, \\
    &\hspace{5cm} t \ge \proj_{\RR}(\phi(S(v',\sqrt{-h'}))^\circ) \quad \forall (v',h') \in A\}\\
    &=\{(w,t)\in \RR^d \times \RR: t=\|w\|^2-\|v-w\|^2-h,\\
    &\hspace{5cm} t \ge \|w\|^2-\|v'-w\|^2-h' \quad \forall (v',h') \in A\}\\
    &=\{(w,t)\in \RR^d \times \RR:  t=\|w\|^2-\|v-w\|^2-h,\\
    &\hspace{5cm}\pow(w,(v,h)) \le \pow(w,(v',h')) \quad \forall (v',h') \in A\}.
\end{align*}
The latter implies
\[
    \proj_{\RR^d} F(v,h) = C((v,h),A).
\]
On the other hand we have $\inter \proj_{\RR^d} F(v,h)\neq \emptyset$ if and only if $F(v,h)$ has dimension $d$ and, hence, 
\[
    \cL_d(A)=\cP_d(A).
\]
For a more detailed overview of the above interpretation we refer the reader to \cite[Chapters 17,18]{BY98}.

\subsubsection{Properties of Laguerre diagrams}\label{sec:LaguerreDiagramProperties}

It shall be emphasized that $\cL_d(A)$ is not necessarily a tessellation in the sense of the definition given in Section \ref{sec:Rtessel}. From the definition of the Laguerre diagram we already get that the cells are convex, closed and have non-empty disjoint interiors (see \cite[Section~2]{Sch93} and \cite[Section~2.2]{Ldoc} for more details). In \cite[Proposition 1, Step 5]{Sch93} it has also been shown that the cells are in face-to-face position. Hence it remains to check whether the collection is a locally finite covering of $\RR^d$ and that the cells are bounded. Putting some additional restrictions on the set $A$ we may ensure $\cL_d(A)\in\TT$. 

As in \cite[Definition 2.2.1]{Ldoc} we introduce the following \textit{regularity conditions} on $A$:
\begin{enumerate}
    \item[(P1)] $\conv(v\colon (v,h)\in A)=\RR^{d}$.\label{item:P1}
    \item[(P2)] For every $w\in\RR^{d}$ and every $t\in \RR$ there are only finitely many $(v,h)\in A$ satisfying $\pow (w,(v,h))\leq t$.\label{item:P2}
\end{enumerate}
Further, we say the points of $A$ are in \textit{general position} if:
\begin{enumerate} 
    \item[(P3)] No $d+2$ points $(v_0,h_0),\dots, (v_{d+1},h_{d+1})$ 
    of $A$ lie on the same downward paraboloid of the form
    \[
        \{(v,h)\in \RR^{d}\times\RR:  \|v - w\|^2 + h = t\}
    \]
    with  $(w,t)\in\RR^{d}\times \RR$.\label{item:P3}
    \item[(P4)] No $k+1$ nuclei are contained in a $(k-1)$-dimensional affine subspace of $\RR^d$ for $k=2,\ldots,d$.\label{item:P4}
\end{enumerate} 
In \cite[Proposition 2.2.4]{Ldoc} and \cite[Proposition~1]{Sch93} it was shown that if $A$ satisfies the regularity conditions \hyperref[item:P1]{(P1)} and \hyperref[item:P2]{(P2)}, then $\cL_d(A)$ is a tessellation. Further in \cite[Corollary 2.2.7]{Ldoc} it has been shown that under condition \hyperref[item:P3]{(P3)} and \hyperref[item:P4]{(P4)} the tessellation $\cL_d(A)$ is normal.

It should be noted that while condition \hyperref[item:P1]{(P1)} is necessary for $\cL_d(A)$ being a tessellation, condition \hyperref[item:P2]{(P2)} is not. Following \cite{Sch93} we may slightly relax the regularity conditions and say that $A$ is \textit{admissible} if the Laguerre diagram $\cL(A)$ is a locally finite covering of $\RR^d$ by convex polytopes, which are in face-to-face positions and $A$ satisfies \hyperref[item:P1]{(P1)}. Note that if $A$ is regular then $A$ is admissible and if $A$ is admissible, then $\tilde A:=\{(v,h)\in A\colon \inter C((v,h),A)\neq\emptyset\}$ is regular (see \cite[Proposition 2]{Sch93}).

\subsubsection{Dual model}\label{sec:DualModel}

Given a countable set of points $A$ in $\RR^d \times \RR$ which is admissible and satisfies \hyperref[item:P3]{(P3)} and \hyperref[item:P4]{(P4)}, we will associate with a tessellation $\cL_d(A)$ its \textit{dual tessellation} in the same fashion as the definition of the \textit{Delaunay triangulation} as dual model of the Voronoi tessellation. More precisely, for a given vertex $z\in\cF_0(\cL_d(A))$ we construct the \textit{Delaunay cell} $D(z,A)$ as 
\[
    D(z,A) := \conv(v:(v,h)\in A,z\in C((v,h),A)).
\]
Since $\cL_d(A)$ is a normal tessellation in $\RR^d$ we have that for a given vertex $z\in \cF_0(\cL_d(A))$ there exist exactly $d+1$ points $(v_1,h_1),\dots,(v_{d+1},h_{d+1})$ of $A$ such that $z\in C((v_i,h_i),A)$ for all $i=1,\dots,d+1$. Therefore the Delaunay cell $D(z,A)$ is a simplex. We define the \textit{dual Laguerre tessellation} $\cL_d^*(A)$ as the collection of all these Delaunay cells, i.e.
\[
    \cL_d^*(A):=\{D(z,A):z\in \cF_0(\cL_d(A))\},
\]  
which is a simplicial tessellation as follows from \cite[Proposition 2]{Sch93}. 

It should be noted, that $\cL_d^*(A)$ is a Laguerre tessellation itself. Indeed, since $\cL_d(A)$ is normal we have that for any $z\in\cF_0(\cL_d(A))$ there is a unique number $K_z\in \RR$ such that there exist exactly $d+1$ points $(v_1,h_1),\dots,(v_{d+1},h_{d+1})$ of $A$ with
\[
    \pow(z,(v_1,h_1))=\ldots=\pow(z,(v_{d+1},h_{d+1}))=K_z
\]
and there is no other point $(v,h)$ of $A$ with $\pow(z,(v,h)) < K_z$. Recalling the definition of $\pow(z,(v,h))=\|z-v\|^2+h$ and $\inter \Pi_{(z,K_z)}^{\downarrow}=\{(v,h)\in \RR^d \times \RR: h < - \|z-v\|^2 + K_z \}$ we see that the latter condition is equivalent to saying that the points $(v_1,h_1),\dots,(v_{d+1},h_{d+1})$ belong to the downward paraboloid $\Pi_{(z,K_z)}$ with apex $(z,K_z)$ and $A\cap \inter \Pi_{(z,K_z)}^{\downarrow}=\emptyset$. Let
\begin{equation}\label{eq:Astar}
    A^*:=\{(z,-K_z)\colon z\in \cF_0(\cL_d(A))\},
\end{equation}
then $A^*$ is regular and $\cL_d^*(A)=\cL_d(A^*)$ (see \cite[Proposition 2]{Sch93}).

\subsection{Definition and some properties of Poisson-Laguerre tessellations} \label{sec:randomL}
In this paper we are interested in the situation when $A$ is given by the support of some (inhomogeneous) Poisson point processes in $\RR^d\times\RR$. More precisely, let $E \subseteq \RR$ be an interval of one of the following types: 
\begin{itemize}
    \item[(i)] $E=[a,\infty)$ for some $a\in\RR$;
    \item[(ii)] $E=(-\infty,b)$ for some $b\in\RR$; 
    \item[(iii)] $E=\RR$, 
\end{itemize}
and let $f\in L_{{\rm loc}}^{1,+}(E)$. In this case $f$ defines a locally finite diffuse measure $\Lambda_{f,\gamma}$ on $\RR^d\times E$ as 
\begin{align*}
    \Lambda_{f,\gamma}(B\times A):=\gamma\int_B\int_A f(h) \dd h \,\dd v,\qquad \gamma\in(0,\infty),
\end{align*}
for any Borel sets $B \subseteq \RR^d$ and $A \subseteq E$. Let $\eta_{f,\gamma}$ be a Poisson point process in $\RR^d \times E$ with intensity measure $\Lambda_{f,\gamma}$. 

\begin{remark}\label{rmk:intervals}
    Note that we do not require that the support of $f$ is $E$ and in general $f$ may be supported on a smaller set. On the other hand the condition $f\in L_{\rm loc}^{1,+}(E)$ puts some restrictions on the type of the interval $E$ for a given function, but not necessarily determines it uniquely. Thus, for example any function $f\colon (a,\infty)\to \RR_+$, which is locally integrable on $[a,\infty)$ may also be defined on $\RR$ by setting $f(x)=0$ for $x\leq a$ and will still be locally integrable on $\RR$. We prefer to distinguish the cases (i) and (iii) since in some situations stronger statements can be shown in the case of interval (i), but not for the interval (iii). On the other hand the difference between the cases (ii) and (iii) is more essential since there are functions $f\colon (-\infty,b)\to\RR_+$, which are locally integrable on $(-\infty,b)$ but not locally integrable on $(-\infty,b]$. Considering the functions defined on the interval of type (ii) we will often assume, that they are not locally integrable on $\RR$, and these functions have to be treated more carefully (as in Theorem \ref{thm:FICondition}). 
\end{remark}

The fractional integral of $f$ appears naturally when we consider the properties of the Poisson point process $\eta_{f,\gamma}$. Namely it describes the expected number of points of $\eta_{f,\gamma}$ lying below the standard downward paraboloid with given apex as stated in the next lemma.

\begin{lemma}\label{lm:IntensityParabola}
    Let $E$ be the interval of type (i), (ii) or (iii) and $f\in L_{\rm loc}^{1,+}(E)$. Then for any $\gamma\in (0,\infty)$ and $(w,t)\in \RR^d\times \RR$ we have
    \[
    \EE(\eta_{f,\gamma}(\inter \Pi^{\downarrow}_{(w,t)}))=\EE(\eta_{f,\gamma}(\Pi^{\downarrow}_{(w,t)}))=\gamma\pi^{\frac{d}{2}}(\frI^{{d\over 2}+1}f)(t),
    \]
    where $\frI^{{d\over 2}+1}f$ is defined by \eqref{eq:Frac_Int}.
\end{lemma}
\begin{proof}
    By Campbell's theorem \cite[Theorem 3.1.2]{SW} we have
    \begin{align*}
        \EE(\eta_{f,\gamma}(\inter \Pi^{\downarrow}_{(w,t)}))=\EE(\eta_{f,\gamma}(\Pi^{\downarrow}_{(w,t)})) &= \gamma\int_{\RR^d} \int_{E} {\bf 1}\left(h\le-\|v-w\|^2+t\right) f(h) \dd h \dd v \\
        &= \gamma\int_{\RR}{\bf 1}(h\le t) f(h) \int_{\RR^d} {\bf 1}\left(\lVert v \rVert \le (t - h)^\frac{1}{2}\right)\dd v\dd h\\ 
        &= \frac{\pi^{\frac{d}{2}}\gamma}{\Gamma({\frac{d}{2}}+1)}\int_{-\infty}^t f(h)(t - h)^\frac{d}{2}\dd h \\
        &= \gamma\pi^{\frac{d}{2}} (\frI^{\frac{d}{2}+1}f)(t),
    \end{align*}
    which finishes the proof.
\end{proof}

Next we specify the conditions for $f$ such that $\eta_{f,\gamma}$ is almost surely admissible and the points of $\eta_{f,\gamma}$ are in general position. In this case we have that $\cL_d(\eta_{f,\gamma})$ is almost surely a normal random tessellation (see Section \ref{sec:LaguerreDiagramProperties}).

\begin{theorem}\label{thm:FICondition}
    Let $E$ be the interval of type (i), (ii) or (iii) and consider $f\in L_{\rm loc}^{1,+}(E)$ satisfying
    \begin{equation}\label{eq:FICondition}
        (\frI^{\frac{d}{2}+1}f)(t) <\infty
    \end{equation}
    for all $t\in E$, and in the case of an interval of type (ii) we additionally assume that there exists $\varepsilon > 0$ and $n_0\in \NN$ such that for all $n\ge n_0$ we have
    \begin{align}\label{eq:frI_conv}    
        (\frI^{{d\over 2}+1}f)(b-1/n) \ge n^\varepsilon,
    \end{align}
    where $\frI^{{d\over 2}+1}f$ is defined by \eqref{eq:Frac_Int}. Then $\cL_d(\eta_{f,\gamma})$ is almost surely a normal random tessellation for any $\gamma\in(0,\infty)$.
\end{theorem}

We abuse the notation in Section \ref{sec:Ltessel} slightly and will call functions $f\in L_{\rm loc}^{1,+}(E)$ satisfying the assumptions of Theorem \ref{thm:FICondition} \textit{admissible}.
\begin{definition}\label{def:admissible}
    Let $E$ be the interval of type (i), (ii) or (iii). We call a function $f\in L_{\rm loc}^{1,+}(E)$ \textit{admissible} if it satisfies the following conditions:
    \begin{enumerate}
        \item[(F1)] $(\frI^{\frac{d}{2}+1}f)(t) <\infty$ for all $t\in E$, \label{item:F1}
        \item[(F2)] in the case of an interval of type (ii) there exists $\varepsilon > 0$ and $n_0\in \NN$ such that for all $n\ge n_0$ it holds \[(\frI^{{d\over 2}+1}f)(b-1/n) \ge n^\varepsilon,\]\label{item:F2}
    \end{enumerate}
    where $\frI^{{d\over 2}+1}f$ is defined by \eqref{eq:Frac_Int}.    
\end{definition}

\begin{remark}\label{rem:NotLocInt}
    In Remark \ref{rmk:intervals} we argued that in the case of an interval E of type (ii) we assume that the function $f:(-\infty,b)\to \RR_+$ is not locally integrable on $(-\infty,b]$. Note that \hyperref[item:F1]{(F1)} and \hyperref[item:F2]{(F2)} imply $f\not\in L_{\rm loc}^{1,+}((-\infty,b])$. Indeed, from \hyperref[item:F2]{(F2)} it follows that
    \[
        \int_{b-1}^bf(t)(b-t)^{d\over 2}\dint t+\int_{-\infty}^{b-1}f(t)(b-t)^{d\over 2}\dint t=\int_{-\infty}^bf(t)(b-t)^{d\over 2}\dint t=\infty.
    \]
    Using ${b-t\over b-1-t}\leq 2$ which holds for any $t\leq b-2$ we note
    \begin{align*}
        \int_{-\infty}^{b-1}f(t)(b-t)^{d\over 2}\dint t&= \int_{b-2}^{b-1}f(t)(b-t)^{d\over 2}\dint t+\int_{-\infty}^{b-2}f(t)(b-t)^{d\over 2}\dint t\\
        &\leq 2^{d\over 2}\int_{b-2}^{b-1}f(t)\dint t+2^{d\over 2}\int_{-\infty}^{b-2}f(t)(b-1-t)^{d\over 2}\dint t\\
        & \leq 2^{d\over 2}\int_{b-2}^{b-1}f(t)\dint t +2^{d\over 2}\Gamma\Big({d\over 2}+1\Big)(\frI^{{d\over 2}+1}f)(b-1)<\infty,
    \end{align*}
    due to \hyperref[item:F1]{(F1)} and since $f\in L_{\rm loc}^{1,+}((-\infty,b))$. Therefore we have
    \[
        \int_{b-1}^bf(t)\dint t\ge \int_{b-1}^bf(t)(b-t)^{d\over 2}\dint t = \infty
    \]
    and $f$ is not locally integrable on $(-\infty,b]$.
\end{remark}

From now on let $\cL_{d,\gamma}(f)$ denote the normal Laguerre tessellation constructed for the Poisson point process $\eta_{f,\gamma}\sim \PPP(\gamma f(h)\dd v\dd h)$ on $\RR^d \times E$ with $f$ being admissible and let $\cL^*_{d,\gamma}(f)$ be its dual. Let us point out that condition \hyperref[item:F1]{(F1)} is not very restrictive and holds for big classes of functions. In the next proposition we collected some of them.

\begin{proposition}\label{prop:FICondition2}
    Let $d\ge 1$ and let $f\colon \inter E \to \RR_+$ be a non-negative measurable function, satisfying one of the following conditions:
    \begin{enumerate}[label=(\alph*)]
        \item[(a)] $E=[a,\infty)$, $a\in\RR$ and $f\in L^{1,+}_{\rm loc}(E)$;
        \item[(b)] $E=(-\infty,b)$, $b\in\RR$ and $u^{-{d\over 2}-2}(f\circ \tau_b)\big(-1/u\big)\in L_{\rm loc}^{1,+}([0,\infty))$;
    \end{enumerate}
    Then $f$ satisfies \hyperref[item:F1]{(F1)}.
\end{proposition}

\begin{remark}
    In particular (a) implies that when $f$ is a density of some locally finite measure on the interval of type (i) with respect to the Lebesgue measure, it satisfies \hyperref[item:F1]{(F1)} and, hence, is admissible.
\end{remark}

\begin{proof}
    (a) is a direct consequence of \cite[Lemma 2.1]{MARTINEZ1992111} in combination with \eqref{eq:FractionalIntegralShift}, while (b) follows from \eqref{eq:frac_f_g} and (a).
\end{proof}

It should be noted that random tessellations $\cL_{d,\gamma}(f)$ and their dual models $\cL^*_{d,\gamma}(f)$ have been studied before for some special types of functions. Below we list some examples of admissible functions including the cases considered previously in the literature.

\begin{example}[Independently marked homogeneous Poisson point process]\label{ex:IndQMarking}

In \cite{LZ08,Ldoc} the authors considered the Laguerre tessellation $\cL_d(\xi)$, where $\xi$ is an independent $\QQ$-marking of a homogeneous Poisson point process in $\RR^d$ with intensity $\gamma$ and $\QQ$ is a probability measure on $(-\infty,0]$. Let us point out that in \cite{LZ08,Ldoc} another definition of a power function, namely $\pow(w,(v,r))=\|v-w\|^2-r^2$, is used. It coincides with our definition by setting $h=-r^2$, which should be taken into account while comparing the corresponding results. In \cite[Theorem 4.1]{LZ08} it was shown that $\cL_d(\xi)$ is a random tessellation if and only if
\begin{equation}\label{eq:ConditionQ}
    \EE[(-R)^{d\over 2}]<\infty,
\end{equation}
where $R$ is a random variable with distribution $\QQ$.

Further note that an independent $\QQ$-marking $\xi$ of a Poisson point process is itself a Poisson point process on the product space $\RR^d\times (-\infty,0]$ (see \cite[Theorem 3.5.7]{SW} or \cite[Proposition 6.16]{LP}). Then if $\QQ$ is absolutely continuous with respect to the Lebesgue measure it has a density $q:\RR\to\RR_+$ satisfying $q(x)=0$ for all $x>0$ and $q\in L^1(\RR)$. In this case we have $\xi\overset{d}{=}\eta_{q,\gamma}$ and $\cL_d(\xi)\overset {d}{=}\cL_{d,\gamma}(q)$. According to Definition \ref{def:admissible} $q$ is admissible if 
\begin{equation}\label{eq:ConditionQ2}
    (\frI^{{d\over 2}+1}q)(p)=\Gamma\Big({d\over 2}+1\Big)\int_{-\infty}^pq(h)(p-h)^{d\over 2}\dint h<\infty,
\end{equation}
for all $p\in \RR$. Note that since the fractional integral is monotone in $p$ for any $p\leq 0$ we have
\[
(\frI^{{d\over 2}+1}q)(p)\leq \Gamma\Big({d\over 2}+1\Big)\int_{-\infty}^0q(h)(-h)^{d\over 2}\dint h=\Gamma\Big({d\over 2}+1\Big)\EE[(-R)^{d\over 2}].
\]
On the other hand for any $p>0$, $h<0$ and $d\ge 2$ using Jensen's inequality we get 
\[
    (p-h)^{\frac{d}{2}}=2^{\frac{d}{2}}\Big(\frac{1}{2}p+\frac{1}{2}(-h)\Big)^{\frac{d}{2}}\le 2^{\frac{d}{2}-1}(p^{\frac{d}{2}}+(-h)^{\frac{d}{2}})<2^{d}(p^{\frac{d}{2}}+(-h)^{\frac{d}{2}}).
\]
At the same time for $d=1$ we use $\sqrt{p+(-h)}\leq \sqrt{p}+\sqrt{-h}$. Hence,
\[
    \int_{-\infty}^p(p-h)^{d\over 2}q(h)\dint h\leq 2^{d} \left(p^{d\over 2}\int_{\RR}q(h)\dint h+\int_{-\infty}^0(-h)^{\frac{d}{2}}q(h)\dint h\right)=2^{d}p^{d\over 2}+2^{d}\EE[(-R)^{d\over 2}],
\]
since $q$ is a density of some probability measure. Thus \eqref{eq:ConditionQ2} is equivalent to \eqref{eq:ConditionQ} in this case.

The advantage of our approach is that we could also treat the situation, when the mark distribution is supported on whole $\RR$, but on the other hand we can only consider mark distributions, which are absolutely continuous with respect to the Lebesgue measure.
\end{example}

\begin{example}[$\beta$-, $\beta'$- and Gaussian-Voronoi tessellations]\label{ex:BetaModels}
In \cite{GKT20, GKT21} functions which are not integrable have been considered, in particular, the following three families of functions:
\begin{align}
    &f_{d,\beta} \colon (0,\infty)\to \RR_+, &&f_{d,\beta}(h)=c_{d+1,\beta}h^{\beta},&&\beta>-1,\, c_{d+1,\beta}:=\frac{\Gamma\left(\frac{d+1}{2}+\beta+1\right)}{\pi^{\frac{d+1}{2}}\Gamma(\beta+1)},\label{eq:BetaModel}\\
    &f'_{d,\beta}\colon(-\infty,0)\to \RR_+ , &&f'_{d,\beta}(h)=c'_{d+1,\beta}(-h)^{-\beta},&&\beta >\frac{d}2 + 1,\, c'_{d+1,\beta}:=\frac{\Gamma\left(\beta\right)}{\pi^{\frac{d+1}{2}}\Gamma(\beta-{\frac{d+1}{2}})},\label{eq:BetaPrimeModel}\\
    &\widetilde f_{\lambda}\colon \RR \to \RR_+, &&\widetilde f_{\lambda}(h)=e^{\lambda h},&&\lambda>0.\label{eq:GaussianModel}
\end{align}
In what follows we will call these three settings $\beta$-model, $\beta'$-model and Gaussian-model, respectively. It was shown in \cite[Lemma 3]{GKT20} and in \cite[Section 3.3]{GKT21} that the corresponding Laguerre diagrams $\cV_{d,\beta,\gamma}:=\cL_{d,\gamma}(f_{d,\beta})$, $\cV'_{d,\beta,\gamma}:=\cL_{d,\gamma}(f'_{d,\beta})$ and $\cG_{d,\lambda,\gamma}:=\cL_{d,\gamma}(\widetilde f_{\lambda})$ are stationary normal random tessellations in $\RR^d$, which are called \textit{$\beta$-Voronoi}, \textit{$\beta'$-Voronoi} and \textit{Gaussian-Voronoi tessellations}, respectively. It should be pointed out though that \cite[Lemma 1]{GKT20}, which was used in the proof of \cite[Lemma 3]{GKT20}, contains a gap, such that the proof in the case of the $\beta'$-model is incomplete. We close this gap in Theorem \ref{thm:FICondition}. Let us ensure that functions defined as in \eqref{eq:BetaModel}, \eqref{eq:BetaPrimeModel} and \eqref{eq:GaussianModel} are admissible and Theorem \ref{thm:FICondition} applies in this three cases.

\begin{figure}[t]
\centering
\includegraphics[width=0.3\columnwidth]{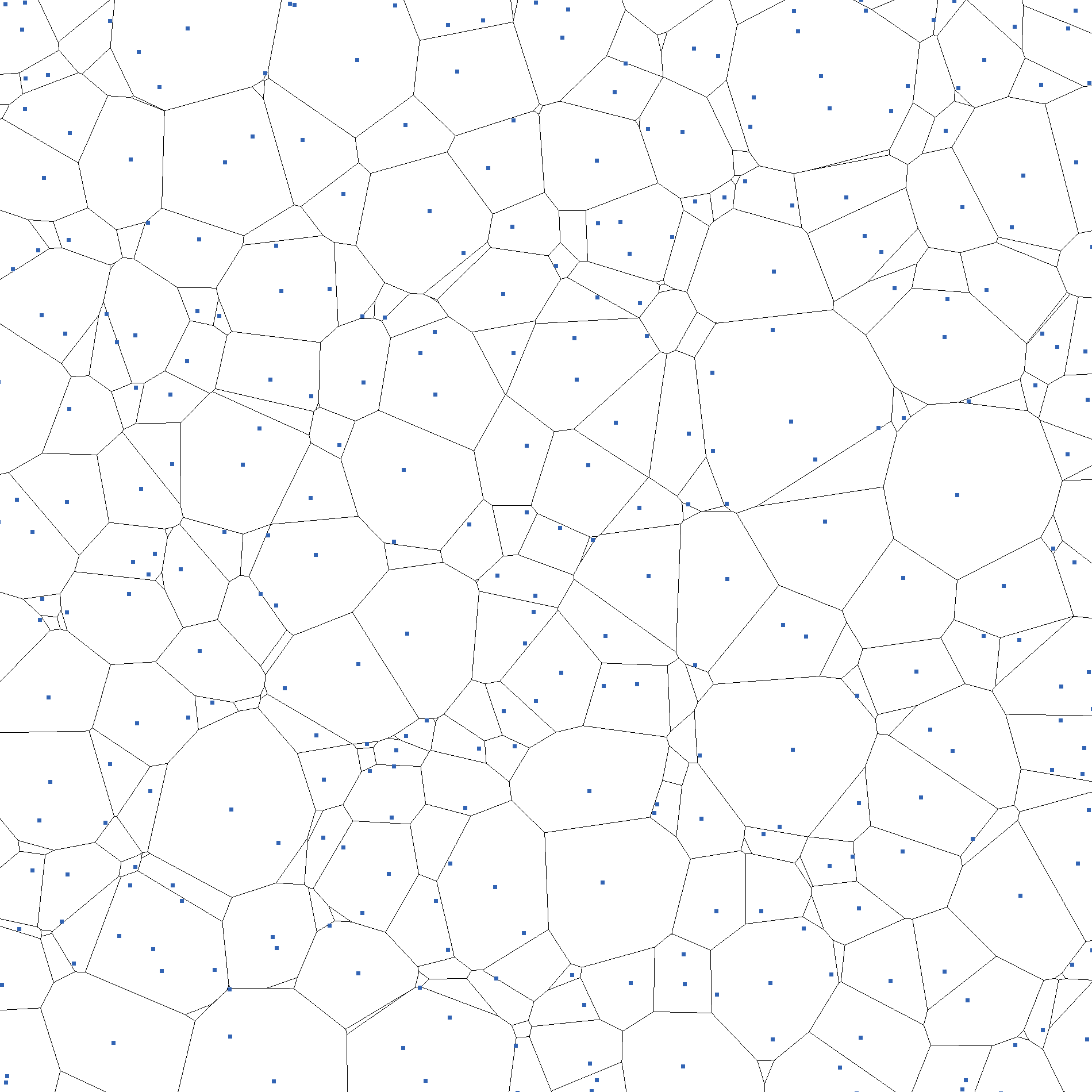}
\quad
\includegraphics[width=0.3\columnwidth]{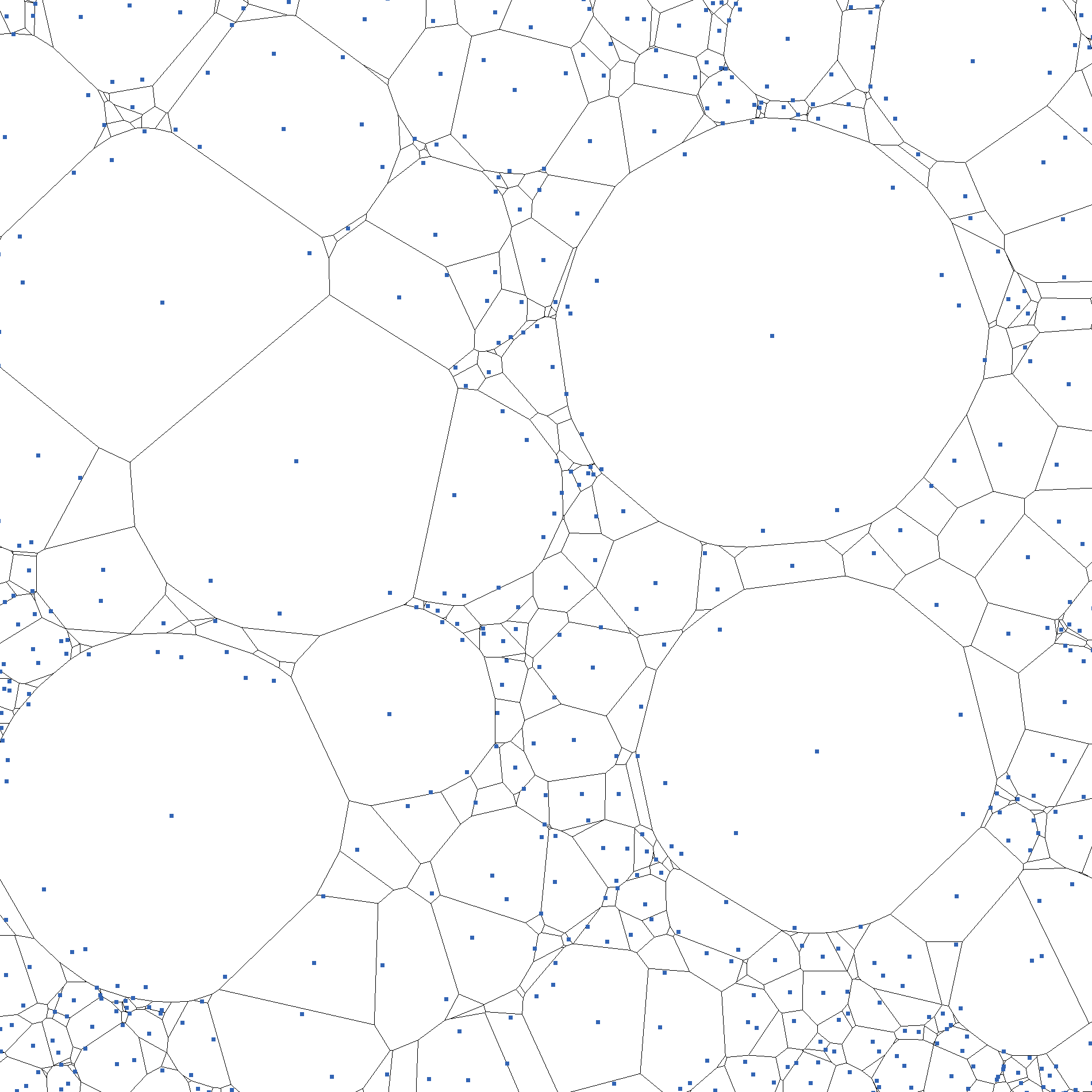}
\quad
\includegraphics[width=0.3\columnwidth]{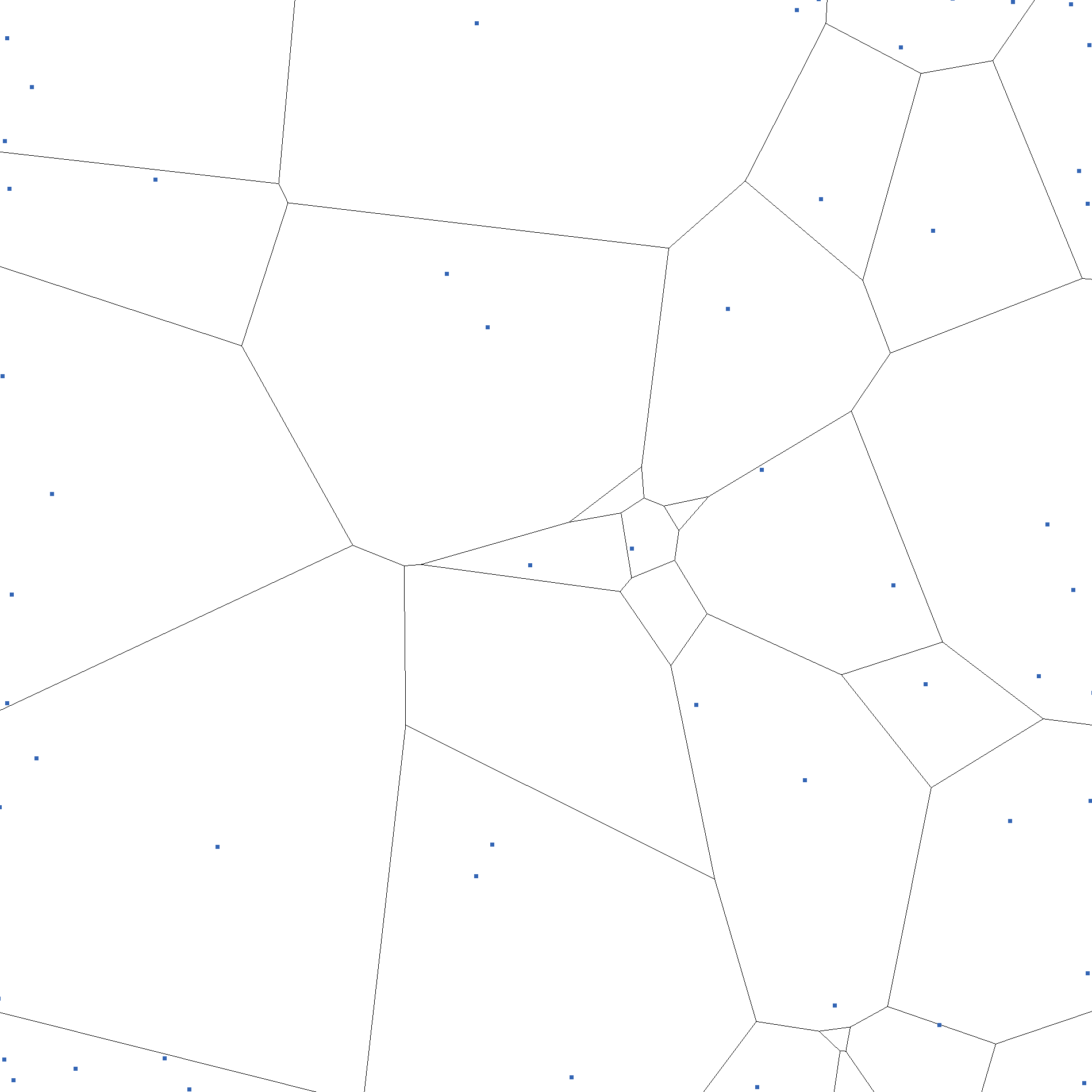}
\caption{Left: $\beta$-Voronoi tessellation in $\RR^2$ with $\beta=5$ (interval $E$ of type (i)) \cite{sectional}. Middle: $\beta'$-Voronoi tessellation in $\RR^2$ with $\beta=2.5$ (interval $E$ of type (ii)) \cite{sectional}. Right: Gaussian-Voronoi tessellation in $\RR^2$ (interval $E$ of type (iii)) \cite{sectional}.} \label{fig:BetaModels}
\end{figure}

\textbf{$\beta$-model:} In the case of the $\beta$-model it is easy to ensure that $f_{d,\beta} \in L_{\rm loc}^{1,+}([0,\infty))$ for $\beta>-1$ and, hence, by Proposition \ref{prop:FICondition2} the condition \hyperref[item:F1]{(F1)} of Definition \ref{def:admissible} holds and $f_{d,\beta}$ is admissible. It should be noted that one may additionally include the case $\beta=-1$ in the family of $\beta$-Voronoi tessellations by defining $\cV_{d,-1,\gamma}$ to be a classical Poisson-Voronoi tessellation and interpret the latter as a Laguerre tessellation $\cL_{d,\gamma}(f)$ with $f(h)=\pi^{d+1\over 2}\Gamma({d+1\over 2})^{-1}\delta_0(h)$, where $\delta_0$ is the Dirac delta function at $0$. Formally we do not allow such functions in the definition of Poisson-Laguerre tessellations, but there are strong evidences, that the Poisson-Voronoi tessellation indeed appears as a limit of $\cV_{d,\beta,\gamma}$ as $\beta \to -1$ (see \cite[Remark 6] {GKT20}). 

\medskip

\textbf{$\beta'$-model:} First note that
\[
    u^{-{d\over 2}-2}f'_{d,\beta}\big(-{1\over u}\big)=c'_{d+1,\beta}u^{\beta-{d\over 2}-2}\in L_{\rm loc}^{1,+}([0,\infty)),
\]
for $\beta>{d\over 2}+1$. Hence, by Proposition \ref{prop:FICondition2} the function $f'_{d,\beta}$ satisfies \hyperref[item:F1]{(F1)}. Further using \cite[Property 2.5 (b)]{TAFDE} we have that there exists $\varepsilon>0$ and $n_0\in \NN$ such that
\[
    (\frI^{\frac{d}{2}+1}f_{d,\beta}')(-1/n) = \frac{c_{d+1,\beta}'\Gamma(\beta-\frac{d}{2}-1)}{\Gamma(\beta)}n^{\beta-\frac{d}{2}-1}\ge n^\varepsilon
\]
for all $n\ge n_0$ since $\beta > \frac{d}{2}+1$ and therefore $f'_{d,\beta}$ satisfies \hyperref[item:F2]{(F2)} and is admissible. 

\medskip

\textbf{Gaussian-model:}  By direct computations one can show that $\widetilde f_{\lambda}$ satisfies \hyperref[item:F1]{(F1)} as well and, hence, is admissbile by Definition \ref{def:admissible}.
\end{example}

\medskip

The above examples demonstrate that Theorem \ref{thm:FICondition} provides a sufficient condition for $\cL_{d,\gamma}(f)$ being a normal random tessellation, which is often easy to verify and it allows us to treat the models considered in \cite{Ldoc,GKT20, GKT21} in a unified way. 

\begin{example}[Further examples of admissible functions]\label{ex:FirstExample}

Below we present a few more examples of admissible functions $f$ leading to a normal random tessellation $\cL_{d,\gamma}(f)$.

\begin{enumerate}
    \item $f\colon [1,\infty)\to \RR_+$, $f(h)=h^\beta (\log(h))^{\alpha}$, where $\beta>-1$, $\alpha\in \RR$.

    \medskip

    Let $\varepsilon ={\beta+1\over 2}>0$. For $\alpha\ge 0$ we note that $f(h)\leq \big({\alpha e\over \varepsilon}\big)^{\alpha}h^{\beta+\varepsilon}$. Since  $h^{\beta+\varepsilon}$ is locally integrable we have $f\in L^{1,+}_{\rm loc}([1,\infty))$ and by Proposition \ref{prop:FICondition2} satisfies \hyperref[item:F1]{(F1)}. In the case $\alpha<0$ we note that $f(h)\leq \big(-{\alpha e\over \varepsilon}\big)^{\alpha}h^{\beta-\varepsilon}$ and, hence, $f\in L^{1,+}_{\rm loc}([1,\infty))$ as well.
    
    \item $f\colon (0,\infty)\to \RR_+$, $f(h)=\alpha_1h^{\beta_1}+\ldots +\alpha_k h^{\beta_k}$, $k\in\NN$, where $\beta_1, \ldots, \beta_k >-1$, $\alpha_1,\ldots,\alpha_k>0$.

    \medskip

    This function is obviously in $L^{1,+}_{\rm loc}([0,\infty))$ as a linear combination of locally integrable and non-negative functions on $(0,\infty)$ and, hence, by Proposition \ref{prop:FICondition2} it satisfies \hyperref[item:F1]{(F1)}.
    
    \item $f\colon (0,\infty)\to \RR_+$, $f(h)=h^\beta \log(h)+(e\beta)^{-1}$, where $\beta>0$ and we set $f(0):=\lim_{h\to 0^+}f(h)=0$. 
\end{enumerate}
\end{example}

\begin{remark}
    An interesting question would be to study the influence of the choice of the interval $E$ on the resulting tessellation $\cL_{d,\gamma}(f)$. Looking at the simulations (see Figure \ref{fig:BetaModels}) one may conjecture that in the case of an interval of type (i) the resulting tessellation behaves similar to the classical Poisson-Voronoi tessellation, while the interval of type (ii) leads to a more "irregular" structure.
\end{remark}

The next proposition provides information about the behavior of the tessellation $\cL_{d,\gamma}(f)$ after scaling and shifting the Poisson point process $\eta_{f,\gamma}$ along the height coordinate $h$. Given a tessellation $T$ and $s>0$ we denote by $sT:=\{st:t\in T\}$.

\begin{proposition}\label{prop:compwlin}
    Let $\gamma\in (0,\infty)$ and let $f$ be admissible. For a function  
    $\varphi\colon \overline{\RR}\to \overline{\RR}$, $\varphi(x)=\lambda x+c$, where $\lambda>0$ and $c\in \RR$ we have that $f \circ \varphi\in L_{\rm loc}^{1,+}(\varphi^{-1}(E))$ is admissible and 
    \[
        \cL_{d,\gamma}(f)\overset{d}{=}\lambda^{\frac{1}{2}}\cL_{d,\widetilde{\gamma}}(f\circ \varphi),
    \]
    where $\widetilde{\gamma}:=\lambda^{\frac{d}{2}+1}\gamma\in(0,\infty)$. 
\end{proposition}

\begin{remark}
    Proposition \ref{prop:compwlin} in particular allows us to restrict our consideration to the following cases $E=[0,\infty)$, $E=(-\infty, 0)$ and $E=\RR$. 
\end{remark}

We finalize this subsection with the following proposition, which is a technical tool we will use in Section \ref{sec:sectional}.

\begin{proposition}\label{prop:trunc}
    Let $E$ be an interval of type (ii). Let $\gamma\in (0,\infty)$ and let $f$ be admissible. We define a generalized function $g\colon\RR\to\overline{\RR}_+$ by 
    \[
        g(p):=\begin{cases}
            f(p),\qquad&\text{ if }p<b,\\
            \infty,\qquad&\text{ otherwise}.
        \end{cases}
    \]
    Let $\widetilde{\eta}$ be a Poisson point process in $\RR^d \times \RR$ whose intensity measure is given by $\gamma g(h) \dd v\dd h$. Then $\cL_{d}(\widetilde{\eta})$ is almost surely a normal random tessellation and
    \[
        \cL_{d}(\widetilde{\eta}) \overset{d}{=} \cL_{d}(\eta_{f,\gamma}).
    \]
\end{proposition}

\begin{remark}
    The above proposition has the following intuitional meaning. If we turn back to the interpretation of a Laguerre tessellation $\cL_d(\eta_{f,\gamma})$ as a result of crystallization process (see Section \ref{sec:ConstructionStandard}) Proposition \ref{prop:trunc} reads as follows. By time $b$ with probability $1$ the crystallization process has covered the whole $\RR^d$ and any crystal which start to grow earliest at time $b$ does not cover any points of $\RR^d$. Hence the cells corresponding to points $(v,h)\in \eta$ with $h>b$ have empty interior and these points do not contribute to the Laguerre diagram and can be discarded.
\end{remark}

\subsection{Proofs}

\begin{proof}[Proof of Theorem \ref{thm:FICondition}]
    As described in Section \ref{sec:LaguerreDiagramProperties} the cells of $\cL_d(\eta_{f,\gamma})$ are almost surely convex, closed, have non-empty disjoint interior and are in face-to-face position. To show that the Laguerre cells are compact almost surely it is sufficient to show that $\eta_{f,\gamma}$ satisfies condition \hyperref[item:P1]{(P1)} almost surely as in the proof of \cite[Proposition~2.2.2]{Ldoc}. In the case when $\gamma\int_E f(h)dh=\infty$ condition \hyperref[item:P1]{(P1)} holds almost surely since the projection of the Poisson point process $\eta_{f,\gamma}$ to the space component $\RR^d$ is with probability one an everywhere dense subset of $\RR^d$. If otherwise $\gamma\int_E f(h)dh=c\in\RR$, then the projection of the Poisson point process $\eta_{f,\gamma}$ has the same distribution as a homogeneous Poisson point process with intensity $c$ according to the mapping theorem for Poisson point processes, see \cite[Theorem 5.1]{LP}. In this case $\conv(v:(v,h)\in \eta_{f,\gamma})$ is a stationary random closed set and hence \hyperref[item:P1]{(P1)} holds as well almost surely by \cite[Theorem 2.4.4]{SW}. By \cite[Lemma~10.1.1]{SW} we get that the Laguerre cells are convex polytopes almost surely.
    
    Next we want to ensure that $\cL_d(\eta_{f,\gamma})$ is almost surely a locally finite covering of $\RR^d$. For this we consider two cases. Let first $E$ be the interval of type (i) or (iii) and we show that condition \hyperref[item:P2]{(P2)} holds almost surely.
    We consider the complement of this event and show that 
    \[
        \PP\big(\exists (w,t) \in \RR^d \times \RR: \eta_{f,\gamma}(\Pi_{(w,t)}^\downarrow) = \infty\big) = 0. 
    \]
    Let $(w,t)\in \RR^d \times \RR$ and let $(x,p)\in \Pi_{(w,t)}^\downarrow$. Then 
    \[
        p\le -\|x-w\|^2+t=-\|x\|^2-\|w\|^2+2 \langle x,w \rangle +t.
    \]
    Note that there exists $c_1(w)\in \RR$ such that $2\langle x,w \rangle \le \frac{1}{2}\|x\|^2+c_1(w)$ for all $x\in \RR^d$. Indeed, taking $c_1(w):=2\|w\|^2$ we have 
    \[
        {1\over 2}\|x\|^2+2\|w\|^2-2\|x\|\|w\|={1\over 2}(\|x\|-2\|w\|)^2\ge 0
    \]
    for any $w,x\in\RR^d$ and since $2\langle x,w\rangle\leq 2\|x\|\|w\|$ the claim follows. Hence 
    \begin{equation}\label{eq:27.05.24_1}
        p\le -\|x\|^2-\|w\|^2+\frac{1}{2}\|x\|^2+c_1(w)+t=-\frac{1}{2}\|x\|^2+c_2(w),
    \end{equation}
    where $c_2(w):=c_1(w)+t-\|w\|^2$ and $(x,p)\in \{(v,h)\in \RR^d \times \RR: h\le -\frac{1}{2}\|v\|^2+c_2(w)\}$. For $n\in \NN$ let $\Pi_n^\downarrow:=\{(v,h)\in \RR^d \times \RR: h\le -\frac{1}{2}\|v\|^2+n\}$. Then by \eqref{eq:27.05.24_1} we have that $(x,p)\in \Pi_{\lceil c_2(w)\rceil}^\downarrow$ and, hence,
    \begin{align*}
        \PP\big(\exists (w,t) \in \RR^d \times \RR: \eta_{f,\gamma}(\Pi_{(w,t)}^\downarrow) = \infty\big)
        &\le \PP\big(\exists n\in \NN: \eta_{f,\gamma}(\Pi_n^\downarrow) = \infty\big) \\
        &\le \sum_{n\in \NN}\PP\big(\eta_{f,\gamma}(\Pi_n^\downarrow) = \infty\big).
    \end{align*}
    Therefore it is sufficient to show that $\PP\big(\eta_{f,\gamma}(\Pi_n^\downarrow) = \infty\big)=0$ for all $n\in \NN$. 
    Since $\eta_{f,\gamma}(\Pi_{n}^\downarrow)$ is Poisson distributed with mean $\EE(\eta_{f,\gamma}(\Pi_{n}^\downarrow))$ the latter is equivalent to $\EE(\eta_{f,\gamma}(\Pi_{n}^\downarrow)) <\infty$ for all $n \in \NN$. Note that $\Pi_{n}^\downarrow$ is the image of the set $\Pi_{(0,n)}^\downarrow$ under the mapping $(v,h)\mapsto (\sqrt{2}v,h)$. Thus, by the mapping theorem for Poisson point processes \cite[Theorem 5.1]{LP} and by Lemma \ref{lm:IntensityParabola} we have
    \[
        \EE(\eta_{f,\gamma}(\Pi_{n}^\downarrow))=2^{d\over 2}\EE(\eta_{f,\gamma}(\Pi_{(0,n)}^\downarrow))=\gamma(2\pi)^{\frac{d}{2}} (\frI^{\frac{d}{2}+1}f)(n)<\infty,
    \]
    for any $n\in \NN$ as follows from \eqref{eq:FICondition}. Therefore for $E$ being the interval of type (i) or (iii) we have showed that the point process $\eta_{f,\gamma}$ is almost surely regular and hence $\cL(\eta_{f,\gamma})$ is a random tessellation.
    
    Let now $E$ be an interval of type (ii). We first show, that $\cL_d(\eta_{f,\gamma})$ covers $\RR^d$ almost surely. Let $w\in \RR^d$, then $w\in C((v,h),\eta_{f,\gamma})$ if $(v,h)=\argmin_{(v',h')\in \eta_{f,\gamma}} \pow(w,(v',h'))$. By Lemma \ref{lm:IntensityParabola} we have 
    \begin{align*}
        \EE\big(\eta_{f,\gamma}(\inter \Pi_{(w,b)}^\downarrow)\big) &=  \gamma \pi^{\frac{d}{2}}(\frI^{\frac{d}{2}+1}f)(b),
    \end{align*}
    and, hence,
    \[
        \PP\big(\eta_{f,\gamma}(\inter \Pi_{(w,b)}^\downarrow) = \emptyset\big) = \exp\big(-\EE\big(\eta_{f,\gamma}(\inter \Pi_{(w,b)}^\downarrow)\big)\big) = \exp\big(-\gamma\pi^{\frac{d}{2}} (\frI^{\frac{d}{2}+1}f)(b)\big) = 0,
    \]
    since by \eqref{eq:frI_conv} we have $(\frI^{\frac{d}{2}+1}f)(b) = \infty$. This implies that almost surely there exists a point $(v_0,h_0)\in \eta_{f,\gamma}$ such that $h_0 < -\|v-w\|^2+b$ or equivalently $K:=\pow(w,(v_0,h_0)<b$. Then
    \[
        \argmin_{(v',h')\in \eta_{f,\gamma}} \pow(w,(v',h'))=\argmin_{(v',h')\in \eta_{f,\gamma}\cap \Pi_{(w,K)}^{\downarrow}} \pow(w,(v',h')).
    \]
    By Lemma \ref{lm:IntensityParabola} and \eqref{eq:FICondition} we have that $\eta_{f,\gamma}(\Pi_{(w,K)}^{\downarrow})<\infty$ almost surely and the existence of 
    \[
        (v,h)=\argmin_{(v',h')\in \eta_{f,\gamma}} \pow(w,(v',h'))
    \] 
    follows. Next we show that $\cL_d(\eta_{f,\gamma})$ is a locally finite system almost surely. Let $B\subset \RR^d$ be a compact subset of $\RR^d$. Note that $B$ can be covered by finitely many closed balls of radius $1$ and if there are infinitely many cells $C\in \cL_d(\eta_{f,\gamma})$ such that $B\cap C\neq \emptyset$ then there are infinitely many cells $C\in \cL_d(\eta_{f,\gamma})$ intersecting one of these balls. Noting that $\cL_d(\eta_{f,\gamma})$ is stationary with respect to the spatial coordinate we assume without loss of generality that $B=\BB^d$. Consider the event 
    \[
        \mathcal{E}:=\{\exists q<b:\min_{(v',h')\in \eta_{f,\gamma}} \pow(z,(v',h'))\le q \,\forall z\in \BB^d\}.
    \]
    Intuitively, thinking of $\cL_d(\eta_{f,\gamma})$ in terms of crystallization processes (see Section \ref{sec:ConstructionStandard}), the event $\mathcal{E}$ means that by time $q$ the set $\BB^d\subset \RR^d$ has already been completely covered by the crystals. Note
    \begin{equation}\label{eq:28.05.24_6}
        \PP(\#\{C\in \cL_d(\eta_{f,\gamma}): C\cap \BB^d \neq \emptyset\}=\infty)
        \le \PP(\#\{C\in \cL_d(\eta_{f,\gamma}): C\cap \BB^d \neq \emptyset\}=\infty\mid \mathcal{E})+\PP(\mathcal{E}^c).
    \end{equation}
    We will first show that 
    \begin{equation}\label{eq:28.05.24_1}
        \PP(\#\{C\in \cL_d(\eta_{f,\gamma}): C\cap \BB^d \neq \emptyset\}=\infty\mid \mathcal{E})=0.
    \end{equation}
    Let $(v,h)\in \eta_{f,\gamma}$ and let $z\in C((v,h),\eta_{f,\gamma})\cap \BB^d$. Then given the event $\mathcal{E}$ we have 
    \[
        \pow(z,(v,h))=\min_{(v',h')\in \eta_{f,\gamma}} \pow(z,(v',h'))\le q<b,
    \]
    and denoting by $W:=\{(v,h)\in \RR^d\times \RR\colon h\leq q, \|v\|\leq 1+\sqrt{q-h}\}$ we obtain
    \[
        \#\{C\in \cL_d(\eta_{f,\gamma}): C\cap \BB^d \neq \emptyset\}=\eta_{f,\gamma}(W).
    \]
    Hence it is sufficient to show that 
    \[
        \PP\big(\eta_{f,\gamma}(W)=\infty\big)=0.
    \] 
    Let us introduce the following construction. Let $\varepsilon$ be small enough such that $q+\varepsilon < b$. We cover the ball $\BB^d(1+\sqrt{\varepsilon/d})$ with finitely many $d$-dimensional boxes $A_1,\ldots,A_m$ having side length $\sqrt{\varepsilon/d}$ and diameter $\sqrt{\varepsilon}$ for all $i=1,\ldots,m$. We denote the vertices of $A_i$ by $w_j^{(i)}$, where $j=1,\ldots,2^d$. As next step we will show that 
    \begin{equation}\label{eq:28.05.24_2}
        W\subset \bigcup_{i=1}^m\bigcup_{j=1}^{2^d}\Pi^{\downarrow}_{(w_j^{(i)},q+\varepsilon)}.
    \end{equation}
    Let $(v,h)\in W$. First assume that $v\in A_i$ for some $i=1,\ldots,m$. Then
    \[
        \|v-w_j^{(i)}\| \le \sqrt{\varepsilon} \le \sqrt{q-h+\varepsilon}
    \]
    and, hence, $h\le -\|v-w_j^{(i)}\|^2+q+\varepsilon$ meaning that $(v,h)\in \Pi_{(w_j^{(i)},q+\varepsilon)}^\downarrow$. Let now $v\notin A_i$ for all $i=1,\ldots,m$ and, hence, $v\notin \BB^d(1+\sqrt{\varepsilon/d})$. Let $z\in \BB^d$ be such that 
    \begin{equation}\label{eq:28.05.24_3}
        \|z-v\|=\min_{z'\in\BB^d}\|z'-v\|=\|v\|-1,
    \end{equation}
    and let $y\in \bd\big(\bigcup_{i=1}^m A_i\big)$ be a point satisfying 
    \begin{equation}\label{eq:28.05.24_4}
        \|v-z\|=\|v-y\|+\|y-z\|.
    \end{equation} 
    In particular $y\in \bd A_i$ for some $1\leq i\leq m$. By definition we have $\|y-z\|\ge \sqrt{\varepsilon/d}$ and there exists a vertex $w_j^{(i)}$ of $A_i$, such that
    \begin{equation}\label{eq:28.05.24_5}
        \|y-w_j^{(i)}\| \le \sqrt{\frac{\varepsilon}{2d}} \le \|y-z\|.
    \end{equation}
    Hence, combining \eqref{eq:28.05.24_4} and \eqref{eq:28.05.24_5} we get
    \[
        \|w_j^{(i)}-v\| \le \|w_j^{(i)}-y\|+\|y-v\| \le \|z-y\|+\|y-v\| = \|z-v\|.
    \]
    Further since $(v,h)\in W$ and due to \eqref{eq:28.05.24_3} and the above inequality we conclude
    \[
    h\leq -(\|v\|-1)^2+q=-\|v-z\|^2+q\le -\|w_j^{(i)}-v\|^2+q
    \]
    and $(v,h)\in \Pi_{(w_j^{(i)},q+\varepsilon)}^\downarrow$. Thus, \eqref{eq:28.05.24_2} follows and we have 
    \begin{equation}\label{eq:2024_05_22}
            \PP\big(\eta_{f,\gamma}(W)=\infty\big)\leq \sum_{i=1}^m \sum_{j=1}^{2^d} \PP\big(\eta_{f,\gamma}(\Pi_{(w_j^{(i)},q+\varepsilon)}^\downarrow)=\infty\big).
    \end{equation}
    Finally by Lemma \ref{lm:IntensityParabola} and \eqref{eq:FICondition} we have
    \[
        \EE\big(\eta_{f,\gamma}(\Pi_{(w,q+\varepsilon)}^\downarrow)\big) = \gamma\pi^{d\over 2}(\frI^{{d\over 2}+1}f)(q+\varepsilon)<\infty,
    \]
    implying $\PP\big(\eta_{f,\gamma}(\Pi_{(w,q+\varepsilon)}^\downarrow)=\infty\big)=0$ for any $w\in \RR^d$. Combining this with \eqref{eq:2024_05_22} proves \eqref{eq:28.05.24_1}.
    
    It remains to show that $\PP(\mathcal{E}^c)=0$. Note that
    \[
        \mathcal{E}^c=\{\forall q < b \,\exists z\in \BB^d: \min_{(v,h)\in \eta_{f,\gamma}} \pow(z,(v,h))>q\}.
    \]
    Let $b_n:=b-\frac{1}{2n}$ and $A_n:=\{\exists z\in \BB^d: \min_{(v,h)\in \eta_{f,\gamma}}\pow(z,(v,h))>b_n \}$, then $\mathcal{E}^c\subseteq \limsup_{n\to \infty} A_n$. We will use the Borel-Cantelli lemma and show that
    \[
    \sum_{n=1}^{\infty}\PP(A_n)<\infty.
    \]
    Since we can cover $\BB^d$ with $4^dn^{d/2}$ balls of radius $\sqrt{1/8n}$ and due to stationarity of $\eta_{f,\gamma}$ with respect to the spatial coordinate it is enough to consider the event
    \[
        A_n':= \big\{\exists z\in \BB^d(\sqrt{1/8n}): \min_{(v,h)\in \eta_{f,\gamma}}\pow(z,(v,h))>b_n \big\},
    \]
    since 
    \begin{equation}\label{eq:28.05.24_7}
        \PP(A_n)\leq 4^dn^{d/2}\PP(A_n').
    \end{equation}
    Note, that if for some $(v,h)\in\eta_{f,\gamma}$ we have that $\pow(z,(v,h))\leq b_n$ for all $z\in \BB^d(\sqrt{1/8n})$, then $A_n'$ does not hold. Hence, $A_n'$ implies that for all $(v,h)\in\eta_{f,\gamma}$ we have that $\pow(z,(v,h))>b_n$ for at least one $z\in \BB^d(\sqrt{1/8n})$ and by the definition of the power function \eqref{eq:PowerFunction} we get
    \[
        \PP(A_n')\le \PP\big(\min_{(v,h)\in \eta_{f,\gamma}} \sup_{z\in \BB^d(\sqrt{1/8n})} \|v-z\|^2+h > b_n\big).
    \]
   Further note that for $h\leq b_n-1/8n$ and $v\in \BB^d(\sqrt{b_n-h}-\sqrt{1/8n})$ we have
    \begin{align*}   
        \sup_{z\in \BB^d(\sqrt{1/8n})} \|v-z\|^2+h &\le \sup_{z\in \BB^d(\sqrt{1/8n})}(\|v\|+ \|z\|)^2 + h \\
        &\le (\sqrt{b_n-h}-\sqrt{1/8n} + \sqrt{1/8n})^2 + h = b_n.
    \end{align*}
    Thus, defining 
    \[
        W_{n}:=\{(v,h)\in \RR^d \times \RR:  h\le b_n-1/8n,\, v\in \BB^d(\sqrt{b_n-h}-\sqrt{1/8n})\},
    \]
    we get
    \begin{align}    
        \PP(A_n')&\le \PP\big(\eta_{f,\gamma}(W_n) = \emptyset\big)=\exp\big(-\EE\big(\eta_{f,\gamma}(W_{n})\big)\big).\label{eq:28.05.24_8}
    \end{align}
    By Campbell's theorem \cite[Theorem 3.1.2]{SW} we obtain
    \begin{align*}
        \EE\big(\eta_{f,\gamma}(W_{n})\big) &= \gamma \int_{-\infty}^{b_n-1/8n} \int_{\BB^d(\sqrt{b_n-h}-\sqrt{1/8n})} f(h) \dd v \dd h\\
        &= \frac{\pi^{\frac{d}{2}}\gamma}{\Gamma({\frac{d}{2}}+1)}\int_{-\infty}^{b_n-1/8n} f(h) (\sqrt{b_n-h}-\sqrt{1/8n})^d \dd h\\
        &\ge \frac{\pi^{\frac{d}{2}}\gamma}{\Gamma({\frac{d}{2}}+1)} \int_{-\infty}^{b_n-1/2n} f(h) (\sqrt{b_n-h}-\sqrt{1/8n})^d \dd h.
    \end{align*}
    Notice that for $h\leq b_n-1/2n$ we have $\sqrt{b_n-h}\ge\sqrt{1/2n}$, which in turn implies $\sqrt{1/8n} \le \frac{1}{2} \sqrt{b_n-h}$. Hence, by \eqref{eq:Frac_Int} we conclude
    \begin{align*}
        \EE\big(\eta_{f,\gamma}(W_{n})\big) &\ge \frac{\pi^{\frac{d}{2}}\gamma}{2^d\Gamma({\frac{d}{2}}+1)}\int_{-\infty}^{b_n-1/2n} f(h) (b_n-h)^{d/2} \dd h\\
        &\ge \frac{\pi^{\frac{d}{2}}\gamma}{2^d\Gamma({\frac{d}{2}}+1)} \int_{-\infty}^{b_n-1/2n} f(h) (b_n-1/2n-h)^{d/2} \dd h\\
        &= \gamma 2^{-d}\pi^{d/2} (\frI^{{d\over 2}+1}f)(b-1/n).
    \end{align*}
    By assumption \eqref{eq:frI_conv} we have $(\frI^{{d\over 2}+1}f)(b-1/n)\ge n^\varepsilon$ for some $\varepsilon >0$ and all $n \ge n_0$ and, hence, by \eqref{eq:28.05.24_7} and \eqref{eq:28.05.24_8} we get
    \begin{align*}
        \sum_{n=1}^\infty \PP(A_n) \le \sum_{n=1}^\infty 4^d n^{d/2} \PP(A_n') \le 4^d \sum_{n=1}^{n_0} n^{d/2} \PP(A_n')+4^d\sum_{n=n_0}^\infty n^{d/2} \exp\big( - \gamma 2^{-d}\pi^{d/2} n^\varepsilon\big) < \infty.
    \end{align*}
    By the Borel-Cantelli lemma we therefore have $\PP(\mathcal{E}^c)=0$, which together with \eqref{eq:28.05.24_6} and \eqref{eq:28.05.24_1} implies that $\cL_d(\eta_{f,\gamma})$ is locally finite almost surely. Overall we showed for $E$ being an interval of type (ii) and $f$ satisfying \eqref{eq:FICondition} and \eqref{eq:frI_conv} that $\eta_{f,\gamma}$ is almost surely admissible and hence also in this case $\cL_d(\eta_{f,\gamma})$ is a random tessellation.
    
    Further the Laguerre tessellation is almost surely normal since $f$ is absolutely continuous with respect to the Lebesgue measure which implies \hyperref[item:P3]{(P3)} and \hyperref[item:P4]{(P4)}, see the proof of \cite[Lemma 3]{GKT20} for more details.
\end{proof}

\begin{proof}[Proof of Proposition \ref{prop:compwlin}]    
    First of all we note that $f \circ \varphi$ is clearly non-negative and measurable and since $K\subseteq E$ is compact if and only if $\varphi^{-1}(K)\subseteq \varphi^{-1}(E)$ is compact we get $f \circ \varphi\in L_{\rm loc}^{1,+}(\varphi^{-1}(E))$. By convention we have $f\circ \varphi(x)=0$ for all $x\not \in \varphi^{-1}(E)$, see Remark \ref{rem:convention}. In order to show \hyperref[item:F1]{(F1)} (see Definition \ref{def:admissible}) for $f\circ \varphi$ we let $t\in \varphi^{-1}(E)$, then
    \begin{align*}
        \big(\frI^{\frac{d}{2}+1}\left(f\circ \varphi\right)\big)(t)&=\frac{1}{\Gamma\left(\frac{d}{2}+1\right)}\int_{\varphi^{-1}(a)}^{t}f(\varphi(x))(t-x)^{\frac{d}{2}}\dd x\\
        &=\frac{1}{\Gamma\left(\frac{d}{2}+1\right)}\lambda^{-1}\int_{a}^{\varphi(t)}f(y)(t-\varphi^{-1}(y))^{\frac{d}{2}}\dd y\\
        &=\frac{1}{\Gamma\left(\frac{d}{2}+1\right)}\lambda^{-\frac{d}{2}-1}\int_{a}^{\varphi(t)}f(y)(\lambda t+c-y)^{\frac{d}{2}}\dd y\\
        &=\lambda^{-\frac{d}{2}-1}(\frI^{\frac{d}{2}+1}f)(\varphi(t))<\infty,
    \end{align*}
    since $\varphi(t)\in E$ and $\lambda >0$. For $E$ being an interval of type (ii) we have $\varphi^{-1}((-\infty,b))=(-\infty, \frac{b-c}{\lambda})$, where $b\in \RR$. Since $f$ satisfies \hyperref[item:F2]{(F2)} for some $n_0$ and $\varepsilon>0$ by Definition \ref{def:admissible} and according to the previous computations we have
    \begin{align*}
        \big(\frI^{\frac{d}{2}+1}\left(f\circ \varphi\right)\big)\Big((b-c)/\lambda-1/n\Big)=\lambda^{-\frac{d}{2}-1}(\frI^{\frac{d}{2}+1}f)(b-\lambda/n) \ge \lambda^{-\frac{d}{2}-1}\Big(\left\lfloor {n\over \lambda}\right\rfloor\Big)^{\varepsilon}\ge n^{\varepsilon/2}
    \end{align*}
    for sufficiently big $n$ due to monotonicity of the fractional integral. This means that $f\circ \varphi$ satisfies \hyperref[item:F2]{(F2)} and, hence, is admissible. Therefore, according to Theorem \ref{thm:FICondition} the diagram $\cL_{d,\widetilde{\gamma}}(f\circ \varphi)$ is indeed a random tessellation in $\RR^d$ and normal almost surely. 
    
    Next, we show that shifting the height coordinate $h$ does not influence the tessellation. For $c\in \RR$ let $\eta_{f\circ \tau_c,\gamma}$ be a Poisson point process on $\RR^d \times \tau_c^{-1}(E)$ with intensity measure $\Lambda_{f\circ \tau_c,\gamma}$ having density $\gamma (f\circ \tau_c)$. Define the transformation $g_c\colon\RR^d\times E\to\RR^d\times\tau_c^{-1}(E)$, $(v,h)\mapsto(v,h-c)$. By the mapping theorem for Poisson point processes \cite[Theorem 5.1]{LP} we have that $g_c(\eta_{f,\gamma})$ defines a Possion point process on $\RR^d\times\tau_c^{-1}(E)$ with intensity measure $g_c(\Lambda_{f,\gamma})=\Lambda_{f\circ \tau_c,\gamma}$ and, hence, $\eta_{f\circ \tau_c,\gamma}\overset{d}{ =} g_c(\eta_{f,\gamma})$ implying $\cL_d(\eta_{f\circ \tau_c,\gamma})\overset{d}{=}\cL_d(g_c(\eta_{f,\gamma}))$.  Further we show that $\cL_d(\eta_{f,\gamma})=\cL_d(g_c(\eta_{f,\gamma}))$ almost surely. Let $(v,h)\in \eta_{f,\gamma}$, $c\in\RR$ then
    \begin{align*}
        C((v,h),\eta_{f,\gamma})&=\{w\in\RR^d:\|w-v\|^2+h\le\|w-v'\|^2+h'\text{ for all } (v',h')\in\eta_{f,\gamma}\}\\
        &=\{w\in\RR^d:\|w-v\|^2+h-c\le\|w-v'\|^2+h'-c\text{ for all } (v',h')\in\eta_{f,\gamma}\}\\
        &=\{w\in\RR^d:\|w-\widetilde{v}\|^2+\widetilde{h}\le\|w-\widetilde{v}'\|^2+\widetilde{h}'\text{ for all } (\widetilde{v}',\widetilde{h}')\in g_c(\eta_{f,\gamma})\}\\
        &=C(g(v,h),g_c(\eta_{f,\gamma})).
    \end{align*}
    Since $g_c$ is a bijection we also have $C((\widetilde{v},\widetilde{h}),g_c(\eta_{f,\gamma}))=C(g_c^{-1}(\widetilde{v},\widetilde{h}),\eta_{f,\gamma})$ for all $(\widetilde{v},\widetilde{h})\in g_c(\eta_{f,\gamma})$. Therefore,
    \[
    \cL_d(\eta_{f,\gamma})=\cL_d(g_c(\eta_{f,\gamma}))\overset{d}{=}\cL_d(\eta_{f\circ \tau_c,\gamma}).
    \]

    Now, we show the impact of scaling the height parameter $h$ by some $\lambda>0$. For $\lambda>0$ we consider the transformation $T_{\lambda}:\RR^d \times E \to\RR^d\times (\lambda^{-1}E)$, $T_{\lambda}(v,h):=(\lambda^{-\frac{1}{2}}v,\lambda^{-1}h)$. By the mapping theorem for Poisson point processes \cite[Theorem 5.1]{LP} we have that $T_{\lambda}(\eta_{f,\gamma})$ is a Poisson point process on $\RR^d\times(\lambda^{-1}E)$ with intensity measure $\Lambda$, given by
    \begin{align*}
      \Lambda(B\times (-\infty,s)) &= \gamma \int_{\RR^d} \int_{\RR} f(h) {\bf 1}(T_{\lambda}(v,h)\in B\times (-\infty,s)){\bf 1}(h\in E) \dd h \dd v\\
      &=\lambda\gamma \int_{\RR^d} \int_{\RR} f(\lambda x) {\bf 1}(\lambda^{-\frac{1}{2}}v\in B){\bf 1}(x< s){\bf 1}(x\in \lambda^{-1}E) \dd x \dd v\\
      &=\lambda^{\frac{d}{2}+1}\gamma \int_{\RR^d} \int_{\RR} f(\lambda x) {\bf 1}(w\in B){\bf 1}(x< s){\bf 1}(x\in \lambda^{-1}E) \dd x \dd w,
    \end{align*}
    for any Borel set $B\subset\RR^d$ and $s\in\RR$. Therefore $T_{\lambda}(\eta_{f,\gamma})$ coincides in distribution with a Poisson point process $\widetilde{\eta}$ on $\RR^d\times (\lambda^{-1}E)$ with intensity measure having density $\widetilde{\gamma}f(\lambda h)$, where $\widetilde{\gamma}=\lambda^{\frac{d}{2}+1}\gamma$. Now let $w\in \RR^d$ and $(v,h)\in\eta_{f,\gamma}$. Then $w\in C((v,h),\eta_{f,\gamma})$ if and only if $w\in \lambda^\frac{1}{2}C(T_{\lambda}(v,h),T_{\lambda}(\eta_{f,\gamma}))$. Indeed,
    \begin{align*}
        &C((v,h),\eta_{f,\gamma}) =\{w\in \RR^d: \|w-v\|^2+h \le \|w-v'\|^2+h'\,\, \text{for all } (v',h')\in \eta_{f,\gamma}\}\\
        &=\{w\in \RR^d: \|\lambda^{-\frac{1}{2}}w-\lambda^{-\frac{1}{2}}v\|^2+\lambda^{-1}h \le \|\lambda^{-\frac{1}{2}}w-\lambda^{-\frac{1}{2}}v'\|^2+\lambda^{-1}h'\,\, \text{for all } (v',h')\in \eta_{f,\gamma}\}\\
        &=\{\lambda^{1\over 2}\widetilde w\in \RR^d: \|\widetilde w-\widetilde{v}\|^2+\widetilde{h} \le \|\widetilde w-\widetilde{v}'\|^2+\widetilde{h}'\,\, \text{for all } (\widetilde{v}',\widetilde{h}')\in T_{\lambda}(\eta_{f,\gamma})\}\\
        &=\lambda^{\frac{1}{2}}C(T_{\lambda}(v,h),T_{\lambda}(\eta_{f,\gamma})).
    \end{align*}
    Finally, since the Laguerre tessellation is defined as the collection of all non-empty Laguerre cells, it follows that $\cL_d(\eta_{f,\gamma})=\lambda^{\frac{1}{2}}\cL_d(T_{\lambda}(\eta_{f,\gamma}))$ and hence $\cL_d(\eta_{f,\gamma})$ and $\lambda^{\frac{1}{2}}\cL_d(\widetilde{\eta})$ coincide in distribution. 
    
    Combining these two steps we get that the random tessellations $\cL_{d,\gamma}(f)$ and $\lambda^{\frac{1}{2}}\cL_{d,\widetilde{\gamma}}(f\circ \varphi)$, where $\widetilde{\gamma}:=\lambda^{\frac{d}{2}+1}\gamma$, coincide in distribution.
\end{proof}

Before we prove Proposition \ref{prop:trunc} we formulate the following lemma, which will also be useful in Section \ref{sec:typcell}.

\begin{lemma}\label{lm:jacobian}
    Consider the transformation $\Psi$ defined as
    \begin{align*}
        \Psi\colon  &\RR^{d} \times \RR \times (\RR^{d})^{d+1} \to &&(\RR^{d} \times \RR)^{d+1}\\
        &(w,p,y_1,\dots,y_{d+1}) \mapsto &&(w+y_1,p-\|y_1\|^2,\dots,w+y_{d+1},p-\|y_{d+1}\|^2),
    \end{align*}
    and let $J(\Psi)$ be its Jacobian matrix.  Then $|\det J(\Psi)| =2^{d+1}d!\Delta_{d}(y_1,\dots,y_{d+1})$.
\end{lemma}

\begin{proof}
    For the proof of this lemma see the proof of Theorem 5.1 in \cite{GKT21}.
\end{proof}

\begin{proof}[Proof of Proposition \ref{prop:trunc}]
We start by noting that by restriction properties of Poisson point process \cite[Theorem 5.2]{LP} we get that $\widetilde\eta_{E}=\{(v,h)\in\widetilde \eta\colon h\in E\}$ has the same distribution as $\eta_{f,\gamma}$. Thus, it is sufficient to show that $\PP(\cL_{d}(\widetilde\eta_{E})=\cL_{d}(\widetilde\eta))=1$, where $\cL_{d}(\widetilde\eta_{E})$ is a normal random tessellation according to Theorem \ref{thm:FICondition}. 
We consider the complement of the above event, i.e. $\cL_{d}(\widetilde\eta_{E})\neq\cL_{d}(\widetilde\eta)$, and note that since $\widetilde{\eta}_{E}\subseteq \widetilde{\eta}$ we have for any $(v,h)\in\widetilde{\eta}_{E}$ that
\begin{equation}\label{eq:subset}
    \begin{aligned}
        C((v,h),\widetilde{\eta})&=\left\{w \in \RR^d: \pow(w,(v,h)) \le \pow(w,(v',h')) \text{ for all } (v',h') \in \widetilde{\eta}\right\}\\
        &\subseteq \left\{w \in \RR^d: \pow(w,(v,h)) \le \pow(w,(v',h')) \text{ for all } (v',h') \in \widetilde{\eta}_{E}\right\} \\
        &= C((v,h),\widetilde{\eta}_{E}).
    \end{aligned}
\end{equation}
Keeping in mind that the cells of $\cL_{d}(\widetilde\eta_{E})$ cover $\RR^d$ and using \eqref{eq:subset} we get that $\cL_{d}(\widetilde\eta_{E})\neq\cL_{d}(\widetilde\eta)$ implies that there exists a point $(v,h)\in \widetilde{\eta}_{E}$ such that $C((v,h),\widetilde{\eta}) \subsetneq C((v,h),\widetilde{\eta}_{E})$ and, hence,
\[
    \PP(\cL_{d}(\widetilde\eta_{E})\neq\cL_{d}(\widetilde\eta))\le \PP(\exists (v,h)\in \widetilde{\eta}_{E}: C((v,h),\widetilde{\eta})\subsetneq C((v,h),\widetilde{\eta}_{E})).
\]
Let $(v,h)\in \widetilde{\eta}_{E}$ be such that $C((v,h),\widetilde{\eta})\subsetneq C((v,h),\widetilde{\eta}_{E})$. Since the cells of $\cL_{d}(\widetilde\eta_{E})$ are convex polytopes there exists a vertex $w\in\cF_0(C((v,h),\widetilde{\eta}_{E}))$ such that $w\in C((v,h),\widetilde{\eta}_{E})\setminus C((v,h),\widetilde{\eta})$. The latter implies that
\begin{align}\label{eq:pow}
    \|v-w\|^2+h>\|v'-w\|^2+h'
\end{align} 
for some $(v',h')$ with $h'\ge b$.
As argued in Section \ref{sec:DualModel} there exist exactly $d+1$ distinct points $(v_1,h_1),\ldots, (v_{d+1},h_{d+1})$ of $\widetilde{\eta}_{E}$ such that the points $(v_1,h_1),\ldots, (v_{d+1},h_{d+1})$ belong to the downward paraboloid $\Pi_{(w,p)}$ with apex $(w,p)$, $p\in\RR$ and $\widetilde{\eta}_{E} \cap \inter \Pi_{(w,p)}^{\downarrow}=\emptyset$, where one of these points is $(v,h)$, say $(v_1,h_1)=(v,h)$. Since $(v,h)\in \Pi_{(w,p)}$ and using \eqref{eq:pow} we have 
\[
    p=\|v-w\|^2+h>\|v'-w\|^2+h'\ge b.
\]
For $d+1$ distinct points $x_1=(v_1,h_1),\ldots, x_{d+1}=(v_{d+1},h_{d+1})\in\widetilde\eta_{E}$ recall that $\Pi(x_1,\dots,x_{d+1})$ denotes the almost surely unique downward paraboloid containing $x_1,\dots,x_{d+1}$ on its boundary. Hence, recalling that $\widetilde\eta_{E}\overset{d}{=}\eta_{f,\gamma}$ we have
\begin{align*}
    \PP(\cL_{d}(\widetilde\eta_{E})\neq\cL_{d}(\widetilde\eta))&\le \PP\Big(\exists (x_1,\ldots,x_{d+1})\in (\eta_{f,\gamma})_{\neq}^{d+1}:\inter\Pi^\downarrow(x_1,\ldots,x_{d+1})\cap \eta_{f,\gamma}=\emptyset,\\
    &\hspace{5cm}\apex\Pi(x_1,\ldots,x_{d+1})=(w,p), p\ge b\Big).
\end{align*}
Here $(\eta_{f,\gamma})_{\neq}^{d+1}$ denotes the collection of all tuples of the form $(x_1,\dots,x_{d+1})$ consisting of pairwise distinct points $x_1,\dots,x_{d+1}$  of the Poisson point process $\eta_{f,\gamma}$. Further applying the multivariate Mecke's formula \cite[Corollary 3.2.3]{SW} and get
\begin{align*}
    &\PP\Big(\exists (x_1,\dots,x_{d+1})\in (\eta_{f,\gamma})_{\neq}^{d+1}:\inter\Pi^\downarrow(x_1,\ldots,x_{d+1})\cap \eta_{f,\gamma}=\emptyset,\\
    &\hspace{5cm}\apex\Pi(x_1,\dots,x_{d+1})=(w,p), p\ge b\Big)\\
    &\le \EE\sum_{(x_1,\dots,x_{d+1})\in (\eta_{f,\gamma})_{\neq}^{d+1}} {\bf 1}(\inter\Pi^\downarrow_{(w,p)}\cap\eta_{f,\gamma}=\emptyset, \apex\Pi^\downarrow(x_1,\dots,x_{d+1})=(w,p), p\ge b)\\
    &= \gamma^{d+1}\int_{(\RR^d)^{d+1}}\int_{\RR^{d+1}} \PP(\eta_{f,\gamma}(\inter\Pi^\downarrow_{(w,p)})=0){\bf 1}(\apex\Pi^\downarrow((v_1,h_1),\dots,(v_{d+1},h_{d+1}))=(w,p))\\
    &\hspace{3cm}\times{\bf 1}(p\ge b)\prod_{i=1}^{d+1}f(h_i)\dd h_1,\dots,\dd h_{d+1}\dd v_1,\dots,\dd v_{d+1}.
\end{align*}
 Consider the transformation $\Psi\colon\RR \times (\RR^d)^{d+1} \times \RR^d \to (\RR^d \times \RR^{d+1})$, 
 \[
    \Psi(p,y_1,\dots,y_{d+1}, w):=(w+y_1,p-\|y_1\|^2,\dots,w+y_{d+1},p-\|y_{d+1}\|^2).
 \] 
 Applying $\Psi$ and by Lemma \ref{lm:jacobian} we get
\begin{align*}
    \PP(\cL_{d}(\widetilde\eta_{E})\neq\cL_{d}(\widetilde\eta))&\le(2\gamma)^{d+1}d!\int_{\RR^d}\int_{(\RR^d)^{d+1}}\int_{b}^{\infty} \PP(\eta_{f,\gamma}(\inter\Pi^\downarrow_{(w,p)})=0)\Delta_{d}(y_1,\ldots,y_{d+1})\\
    &\hspace{5cm}\times\prod_{i=1}^{d+1}f(p-\|y_i\|^2)\dd p \dd y_1,\dots,\dd y_{d+1}\dd w.
\end{align*}
For $w\in \RR^d$ and $p\ge b$ we have by Lemma \ref{lm:IntensityParabola} that
\[
    \PP(\eta_{f,\gamma}(\inter\Pi^\downarrow_{(w,p)})=0)=\exp\big(-\gamma\pi^{d\over 2}(\frI^{{d\over 2}+1}f)(p)\big)=0,
\]
since from \hyperref[item:F2]{(F2)} (see Definition \ref{def:admissible}) and monotonicity of the fractional integral in $p$ it follows that $(\frI^{{d\over 2}+1}f)(p)\ge(\frI^{{d\over 2}+1}f)(b)= \infty$. This finishes the proof.
\end{proof}

\section{Affine sections of Poisson-Laguerre tessellations}\label{sec:sectional}

In this section we study sectional properties of the Poisson-Laguerre tessellation $\cL_{d,\gamma}(f)$. Let $L\in A(d,\ell)$ be an affine subspace of dimension $\ell$. Given a tessellation $T$ in $\RR^d$ define 
$$
T\cap L:=\{t\cap L\colon t\in T\},
$$
which is again a tessellation in $L$. Since $L$ is isometric to $\RR^{\ell}$ we will always assume without loss of generality that $T\cap L$ is a tessellation in $\RR^{\ell}$. The aim of this section is to determine the distribution of $\cL_{d,\gamma}(f)\cap L$. Our main motivation to study this problem is the fact that for a stationary and isotropic random tessellation $\cT$ the knowledge of the properties of its sections $\cT\cap L$  gives access to mean value characteristics of the original tessellation $\cT$ (see i.e. \cite[p. 466-467]{SW} for more details).

Sections of random tessellations $\cV_{d,\beta,\gamma}$, $\cV'_{d,\beta,\gamma}$ and $\cG_{d,\lambda}$ from Example \ref{ex:BetaModels} have been studied in \cite{sectional} and in particular in \cite[Theorem~4.1]{sectional} the following results have been shown: for any $L\in A(d,\ell)$ and any $\gamma>0$ it holds that
\begin{itemize}
	\item[(1)] for any $\beta\ge -1$ we have $(\cV_{d,\beta,\gamma}\cap 
 L)\overset{d}{=}\cV_{\ell,\beta+\frac{d-\ell}{2},\gamma}$  (up to isometry);
	\item[(2)]  for any  $\beta>\frac{d}2 + 1$ we have  $(\cV'_{d,\beta,\gamma}\cap L)\overset{d}{=}\cV'_{\ell,\beta-\frac{d-\ell}{2},\gamma}$  (up to isometry);
	\item[(3)]  for any  {$\lambda>0$,} we have  $(\cG_{d,\lambda,\gamma}\cap L)\overset{d}{=}\cG_{\ell,\lambda,\gamma}$  (up to isometry).
\end{itemize}

We generalize the above result to random Poisson-Laguerre tessellations $\cL_{d,\gamma}(f)$ with admissible $f$ and prove the following theorem, which is the main result of this section.

\begin{theorem} \label{thm:sectional}
    Let $L\subset \RR^d$ be an $\ell$-dimensional affine subspace, where $\ell \in \{1,\dots,d-1\}$ and let $f$ be admissible. Then for any $\gamma>0$ we have
    \[
        (\cL_{d,\gamma}(f) \cap L)\overset{d}{=}\cL_{\ell,\gamma}(f_{\ell})\qquad \text{(up to isometry)},
    \]
    where
    \begin{align*}
        f_{\ell}(p):= 
        \begin{cases}
            \pi^{\frac{d-\ell}{2}}(\frI^{\frac{d-\ell}{2}}f)(p)\qquad &\text{if }p\in E,\\ 
            0 &\text{otherwise},
        \end{cases}
    \end{align*}
    and $\frI^{\frac{d-\ell}{2}}f$ is defined in \eqref{eq:Frac_Int}. Moreover, $\cL_{\ell,\gamma}(f_{\ell})$ is almost surely a normal random tessellation, namely $f_{\ell}$ is admissible.
     
\end{theorem}

\begin{proof}
    We start by recalling the construction of Laguerre tessellations introduced in Section \ref{sec:polyhedral}. According to this construction $\cL_{d,\gamma}(f)$ appears as a vertical projection of the boundary of convex closed set $P(\eta_{f,\gamma})$ defined by \eqref{def:polyhedralset}, where 
    \[
        \varphi\colon  \RR^d\times\RR \to \RR^d\times\RR, \quad (v,h) \mapsto (v, \|v\|^2+h).
    \]
    Note that under $\varphi$ points $(v,0)$ are mapped to the standard upward paraboloid $\Pi^+$, while for $h>0$ and $h'<0$ points $(v,h)$ and $(v,h')$ are mapped to $(\Pi^{+})^\uparrow$ and $(\Pi^{+})^\downarrow$, respectively. By the mapping theorem for Poisson point processes \cite[Theorem 5.1]{LP} the process $\varphi(\eta_{f,\gamma})$ is a Poisson point process on $\RR^d \times \RR$ with intensity measure $\Lambda$, given by
    \begin{align*}
        \Lambda(B\times (-\infty,s)) &= \gamma\int_{\RR^d} \int_{E} f(h) {\bf 1}(\varphi(v,h) \in B\times (-\infty,s)) \dd h \dd v\\
        &= \gamma\int_{\RR^d} \int_{\RR} f(h) {\bf 1}(v \in B) {\bf 1}(h+\|v\|^2 < s) \dd h \dd v\\
        &= \gamma\int_{\RR^d} \int_{\RR} f(h-\|v\|^2){\bf 1}((v,h) \in B\times (-\infty,s))\dd h \dd v,
    \end{align*}
    for any Borel set $B\subset\RR^d$ and $s\in\RR$. As described in Section \ref{sec:polyhedral} the cells of $\cL_d(\eta_{f,\gamma})$ are obtained as orthogonal projections of the facets of $P(\eta_{f,\gamma})$ onto $\RR^d$. According to \eqref{def:polyhedralset} the set $P(\eta_{f,\gamma})$ can also be defined as the intersection of the epigraphs of all polar hyperplanes $x^\circ$, where $x \in \varphi(\eta_{f,\gamma})$, namely
    \[
        P(\eta_{f,\gamma}) = \bigcap_{x \in \varphi(\eta_{f,\gamma})} (x^{\circ})^\uparrow,
    \]
    where we recall, that
    \begin{align*}
        x^\circ &=\{(y,y_{d+1})\in \RR^d \times \RR: (x_1,\dots , x_{d+1}, 1) \Delta_{\Pi^+} (y_1, \dots, y_{d+1}, 1)^T = 0\}\\
        &=\big\{(y,y_{d+1})\in \RR^d \times \RR: y_{d+1} = 2 \sum_{i=1}^{d} x_i y_i - x_{d+1} \big\}.        
    \end{align*} 
    
    We note that it is enough to consider the case $\ell=d-1$ since the statement for general $\ell$ will follow by induction. Let $L$ be a hyperplane in $\RR^{d}$. By stationarity and isotropy of the random tessellation $\cL_{d,\gamma}(f)$ we will assume without loss of generality that $L\cong\RR^{d-1}$ is the linear subspace of $\RR^d$ spanned by the first $d-1$ standard orthonormal vectors of $\RR^d$, i.e $L = \{(y_1,\dots,y_{d-1},y_d) \in \RR^d: y_d=0\}$. We extend the hyperplane $L$ by adding the height coordinate and define $L' := \{(y_1,\dots,y_{d-1},y_d,h) \in \RR^d\times\RR: y_d=0\}$. The intersection of the Laguerre tessellation $\cL_{d,\gamma}(f)$ with the hyperplane $L$ can be obtained by applying vertical projection to the facets ($(d-1)$-dimensional faces) of the $d$-dimensional polyhedral set
    \begin{equation}\label{eq:28.05.24_10}
        P(\eta_{f,\gamma})\cap{L'}:=\bigcap_{x\in\varphi(\eta_{f,\gamma})}(x^\circ \cap L')^\uparrow
    \end{equation} 
    to $\RR^{d-1}$, namely,
    \begin{equation}\label{eq:28.05.24_11}
        \cL_{d,\gamma}(f)\cap L=\{\proj_{\RR^{d-1}} F\colon F\text{ is facet of }P(\eta_{f,\gamma})\cap{L'}\}.
    \end{equation}
    Consider the linear transformation 
    \[
        \proj_{L'}\colon \RR^d\times \RR \to L', (y_1,\dots,y_{d-1},y_d,h)\mapsto (y_1,\dots,y_{d-1},0,h).
    \]
    We first show, that for all $x\in\RR^d$ the intersection of the $d$-dimensional hyperplane $x^\circ$ with $L'$ coincides with the $(d-1)$-dimensional hyperplane, which is a polar hyperplane of the point $\proj_{L'}(x)\in L'$ with respect to the quadric $(\Pi^+)':=\Pi^+ \cap L'$ in $L'\cong \RR^d$. Indeed, since for any $x\in\RR^{d+1}$ it holds that
    \[
        [\proj_{L'}(x)]^\circ_{(\Pi^+)'}=\big\{(y_1,\dots,y_{d-1},0,y_{d+1})\in\RR^{d+1}:y_{d+1}=\sum_{i=1}^{d-1}x_iy_i-x_{d+1}\big\},
    \]
    we obtain
    \begin{align*}        
        x^\circ_{\Pi^+} \cap L' 
        &=\big\{(y_1,\dots,y_{d-1},y_d,y_{d+1})\in \RR^d\times\RR:y_{d+1}=2\sum_{i=1}^{d}x_i y_i - x_{d+1},\, y_d=0\big\}\\
        &=\big\{(y_1,\dots,y_{d-1},0,y_{d+1})\in \RR^d\times\RR:y_{d+1}=2\sum_{i=1}^{d-1}x_i y_i - x_{d+1}\big\}=[\proj_{L'}(x)]^\circ_{(\Pi^+)'}.
    \end{align*}
    Hence, combining this with \eqref{eq:28.05.24_10} we get
    \begin{equation}\label{eq:28.05.24_12}
        P(\eta_{f,\gamma})\cap{L'}=\bigcap_{x\in\varphi(\eta_{f,\gamma})}([\proj_{L'}(x)]^{\circ}_{(\Pi^+)'})^{\uparrow}= \bigcap_{y\in\proj_{L'}(\varphi(\eta_{f,\gamma}))}(y^{\circ}_{(\Pi^+)'})^{\uparrow}.
    \end{equation}
    By the mapping theorem for Poisson point processes \cite[Theorem 5.1]{LP} we have that $\proj_{L'}(\varphi(\eta_{f,\gamma}))$ is a Poisson point process in $L'$ with intensity measure $\Lambda'$ defined as follows, for any Borel set $B \subseteq L$ and $s \in \RR$ we have
    \begin{align*}
        \Lambda'(B\times (-\infty, s))&= \gamma\int_{\RR^d} \int_{\RR} f(h-\|v\|^2) {\bf 1}(\proj_{L'}(v,h) \in B \times (-\infty,s)) \dd h \dd v\\
        &= \gamma\int_{\RR} \int_{\RR^{d-1}} {\bf 1}(v' \in B) {\bf 1}(h< s) \int_{\RR} f(h-\|v'\|^2-v_d^2)\dd v_d \dd v' \dd h,
    \end{align*}
    where we denote by $v':=(v_1,\dots,v_{d-1}) \in \RR^{d-1}\cong L$. Further we note that 
    \begin{align*}
        \int_{\RR} f(h-\|v'\|^2-v_d^2)\dd v_d&= \int_{-\infty}^{h-\|v'\|^2} f(y)(h-\|v'\|^2-y)^{-\frac{1}{2}}\dd y= \sqrt{\pi}(\frI^{1\over 2}f)(h-\|v'\|^2),
    \end{align*}
    and, hence, 
    \begin{equation}\label{eq:28.05.24_13}
     \Lambda'(B\times (-\infty, s))= \gamma\sqrt{\pi}\int_{\RR} \int_{\RR^{d-1}} {\bf 1}(v' \in B) {\bf 1}(h< s) (\frI^{1\over 2}f)(h-\|v'\|^2) \dd v' \dd h.
    \end{equation}
    
    Now define 
    \begin{align*}
        f_{d-1}(p):= 
        \begin{cases}
            \sqrt{\pi} (\frI^{1\over 2}f)(p), &\text{if } (\frI^{\frac{1}{2}}f)(p)<\infty\\
            0, &\text{otherwise}.
        \end{cases}
    \end{align*}
    
    We directly note that $f_{d-1}$ is clearly a non-negative function. As a next step we need to ensure that if $f$ is admissible (with dimension $d$) then $f_{d-1}$ is admissible (with dimension $d-1$). We will consider the cases for the interval $E$ of type (i), (ii) and (iii) separately.

    \medskip

    \textbf{Interval of type (i):} In this case we note that since $f\in L_{{\rm loc}}^{1,+}([a,\infty))$ by \cite[Lemma 2.1]{MARTINEZ1992111} in combination with \eqref{eq:FractionalIntegralShift} we have that $(\frI^{1/2}f)(x)<\infty$ for almost every $x\in [a,\infty)$ and,  moreover, $(\frI^{1/2}f)(p)=0$ for any $p\in (-\infty,a]$, meaning that $f_{d-1}=\sqrt{\pi} (\frI^{1/2}f)$ and the support of $f_{d-1}$ is a subset of $[a,\infty)$. By \cite[Lemma 2.1]{MARTINEZ1992111} we also have $f_{d-1}\in L_{{\rm loc}}^{1,+}([a,\infty))$ and by Proposition \ref{prop:FICondition2} we have that $f_{d-1}$ satisfies \hyperref[item:F1]{(F1)} and, hence, is admissible by Definition \ref{def:admissible}.

    \medskip

    \textbf{Interval of type (iii):} Let $f\in L_{\rm loc}^{1,+}(\RR)$ be admissible. Then by the semigroup property \eqref{eq:SemigroupProp}, condition \hyperref[item:F1]{(F1)} and Lemma \ref{lm:technical} we have
    \begin{align*}
        \int_{-\infty}^q(\frI^{1\over 2}f)(p)\dint p=(\frI^{3\over 2}f)(q)<\infty
    \end{align*}
    for any $q\in \RR$. This in particular implies that $(\frI^{1/2}f)(x)<\infty$ for almost every $x\in\RR$ and $f_{d-1}=\sqrt{\pi}\frI^{1/2}f\in L^{1,+}_{\rm loc}(\RR)$. Further by the semigroup property \eqref{eq:SemigroupProp} and condition \hyperref[item:F1]{(F1)} for $f$ we get
    \[
        (\frI^{{d-1\over 2}+1}f_{d-1})(t)=\sqrt{\pi}(\frI^{{d+1\over 2}}\frI^{{1\over 2}}f)(t)=\sqrt{\pi}(\frI^{{d\over 2}+1}f)(t)<\infty
    \]
    for any $t\in\RR$ and $f_{d-1}$ is admissible with dimension $d-1$.

    \medskip

    \textbf{Interval of type (ii):} Finally we consider the case of an admissible function $f\in L_{\rm loc}^{1,+}((-\infty,b))$, where $b\in\RR$. As in the previous case for any $q<b$ we obtain
    $$
        \int_{-\infty}^q(\frI^{1\over 2}f)(p)\dint p=(\frI^{3\over 2}f)(q)<\infty,
    $$
    meaning that $\frI^{1/2}f$ is locally integrable on $(-\infty,b)$ and, in particular, $(\frI^{1/2}f)(x)<\infty$ for almost every $x\in (-\infty,b)$. On the other hand we note, that for any $p\ge b$ we have
    $$
        \Gamma\Big({1\over 2}\Big)(\frI^{1\over 2}f)(p)=\int_{-\infty}^pf(t)(p-t)^{-{1\over 2}}\dint t\ge \int_{b-1}^bf(t)(p-t)^{-{1\over 2}}\dint t \ge (p-b+1)^{-{1\over 2}}\int_{b-1}^bf(t)\dint t=\infty,
    $$
    since $f$ is not locally integrable on $(-\infty,b]$ (see Remark \ref{rem:NotLocInt}). The latter implies $f_{d-1}(p)=\sqrt{\pi}(\frI^{1/2}f)(p){\bf 1}(p<b)$, which is locally integrable on $(-\infty, b)$. By the semigroup property \eqref{eq:SemigroupProp} we have that 
    \[
        (\frI^{{d-1\over 2}+1}f_{d-1})(t)=\sqrt{\pi}(\frI^{{d\over 2}+1}f)(t)
    \]
    for any $t\in(-\infty,b)$, which implies that $f_{d-1}$ is admissible with dimension $d-1$ since $f$ is admissible with dimension $d$ (see Definition \ref{def:admissible}).

    \medskip

    By Theorem \ref{thm:FICondition} we conclude that in all three cases $\cL_{d-1,\gamma}(f_{d-1})$ is almost surely a normal random tessellation. By combining \eqref{eq:28.05.24_10}, \eqref{eq:28.05.24_11}, \eqref{eq:28.05.24_12} and \eqref{eq:28.05.24_13} we also have
    $$
        \cL_{d,\gamma}(f)\cap L\overset{d}{=}\cL_{d-1,\gamma}(\sqrt{\pi} (\frI^{1\over 2}f))\overset{d}{=}\cL_{d-1,\gamma}(f_{d-1}),
    $$
    in the case of an interval $E$ of type (i) and (iii). For the interval (ii) we still need to ensure that $\cL_{d-1,\gamma}(\sqrt{\pi} (\frI^{1/2}f))\overset{d}{=}\cL_{d-1,\gamma}(f_{d-1})$ since $\sqrt{\pi} (\frI^{1/2}f)\neq f_{d-1}$ on $\RR$. On the other hand $\sqrt{\pi} (\frI^{1/2}f)\colon \RR\to\overline{\RR}_+$ is a generalized function obtained from $f_{d-1}$ by setting 
    $$
        \sqrt{\pi} (\frI^{1/2}f)(p)=\begin{cases}
            f_{d-1}(p),\qquad &\text{ if } p<b,\\
            \infty,\qquad &\text{ otherwise}.
        \end{cases}
    $$
    Thus, by Proposition \ref{prop:trunc} we get $\cL_{d-1,\gamma}(\sqrt{\pi} (\frI^{1/2}f))\overset{d}{=}\cL_{d-1,\gamma}(f_{d-1})$, which concludes the proof for $\ell=d-1$.

    The case for general $\ell\in \{1,\dots,d-1\}$ follows now by induction taking into account that according to the semigroup property \eqref{eq:SemigroupProp} for any $1\leq \ell\leq d-2$ we have
    $$
        f_{\ell}=\sqrt{\pi}(\frI^{1\over 2}f_{\ell+1})=\sqrt{\pi}\big(\frI^{1\over 2}\big[\pi^{d-\ell-1\over 2}(\frI^{d-\ell-1\over 2}f)\big]\big)=\pi^{d-\ell\over 2}(\frI^{d-\ell\over 2}f).
    $$

\end{proof}

Let us now ensure that \cite[Theorem~4.1]{sectional} is indeed a corollary of Theorem \ref{thm:sectional}. Consider the tessellations $\cV_{d,\beta,\gamma}=\cL_{d,\gamma}(f_{d,\beta})$, $\cV'_{d,\beta,\gamma}=\cL_{d,\gamma}(f'_{d,\beta})$ and $\cG_{d,\lambda,\gamma}=\cL_{d,\gamma}(\widetilde f_{\lambda})$, defined in Example \ref{ex:BetaModels}.

\medskip

\textbf{$\beta$-model:} Note that for any $\alpha>0$, $\beta>-1$ and $t>0$ by \cite[Property 2.5(a)]{TAFDE} we have 
$$
    (\frI^{\alpha} h^{\beta})(t)={\Gamma(\beta+1)\over \Gamma(\beta+1+\alpha)}t^{\beta+\alpha}.
$$
Then for any $1\leq \ell\leq d-1$ and any $p>0$ we get
\begin{align*}
    \pi^{d-\ell\over 2}(\frI^{d-\ell\over 2}f_{d,\beta})(p) &= c_{d+1,\beta} \frac{\pi^{\frac{d-\ell}{2}}\Gamma\left(\beta + 1\right)}{\Gamma\left(\beta+1+\frac{d-\ell}{2}\right)}\, p^{\beta+\frac{d-\ell}{2}}= c_{\ell+1,\beta+\frac{d-\ell}{2}}\, p^{\beta+\frac{d-\ell}{2}}=f_{\ell,\beta+\frac{d-\ell}{2}}(p).
\end{align*}
Hence, by Theorem \ref{thm:sectional} for any $L\in A(d,\ell)$, $\beta>-1$ and $\gamma>0$ we conclude
\[
    (\cV_{d,\beta,\gamma}\cap L)\overset{d}{=}\cL_{\ell,\gamma}\big(f_{\ell,\beta+\frac{d-\ell}{2}}\big)=\cV_{\ell,\beta+\frac{d-\ell}{2},\gamma}.
\]

\medskip

\textbf{$\beta'$-model:} This case corresponds to the interval of type (ii) with $b=0$. In Example \ref{ex:BetaModels} we have shown that for any $\beta>{d\over 2}+1$ the function $f'_{d,\beta}$ is admissible. Further for any $p\ge 0$ and $1\leq \ell\leq d-1$ we have
\[
    (\frI^{d-\ell\over 2}f'_{d,\beta})(p)={c'_{d+1,\beta}\over \Gamma\big({d-\ell\over 2}\big)}\int_{-\infty}^p (-t)^{-\beta} (p-t)^{\frac{d-\ell}{2}-1} \dd t = \infty.
\]
On the other hand for $p<0$ by \cite[Property 2.5(b)]{TAFDE} we have 
\begin{align*}
    \pi^{d-\ell\over 2}(\frI^{d-\ell\over 2}f'_{d,\beta})(p)
    &= \pi^{\frac{d-\ell}{2}} c_{d+1,\beta}'\frac{\Gamma\left(\beta - \frac{d-\ell}{2}\right)}{\Gamma\left(\beta\right)}(-p)^{-(\beta-\frac{d-\ell}{2})}\\
    &= c_{\ell+1,\beta-\frac{d-\ell}{2}}' (-p)^{-(\beta-\frac{d-\ell}{2})}=f'_{\ell,\beta-{d-\ell\over 2}}(p),
\end{align*}
and by Theorem \ref{thm:sectional} for any $L\in A(d,\ell)$, $\beta>{d\over 2}+1$ and $\gamma>0$ we conclude
\[
    (\cV'_{d,\beta,\gamma}\cap L)\overset{d}{=}\cL_{\ell,\gamma}\big(f'_{\ell,\beta-\frac{d-\ell}{2}}\big)=\cV'_{\ell,\beta-\frac{d-\ell}{2},\gamma}.
\]

\medskip

\textbf{Gaussian-model:} For any $\lambda>0$ and $p \in \RR$ by \cite[Property 2.11(a)]{TAFDE} we have
\begin{align*}
      \pi^{d-\ell\over 2}(\frI^{d-\ell\over 2}\widetilde f_{\lambda})(p)= (\pi/\lambda)^{\frac{d-\ell}{2}} e^{\lambda p}=e^{\lambda p+{d-\ell\over 2}\log(\pi/\lambda)}.
\end{align*}
Further note, that $\pi^{d-\ell\over 2}(\frI^{d-\ell\over 2}\widetilde f_{\lambda})=\widetilde f_{\lambda}\circ \varphi$, where $\varphi(x)=x+{d-\ell\over 2\lambda}\log(\pi/\lambda)$ and by Proposition \ref{prop:compwlin} and Theorem \ref{thm:sectional} we get for $L\in A(d,\ell)$, $\lambda>0$ and $\gamma>0$ that
$$
    (\cG_{d,\lambda,\gamma}\cap L)\overset{d}{=}\cL_{\ell,\gamma}(\widetilde f_{\lambda}\circ \varphi)\overset{d}{=}\cL_{\ell,\gamma}(\widetilde f_{\lambda})=\cG_{\ell,\lambda,\gamma}.
$$

\begin{remark}
    As shown above 
    in the case of the Gaussian-Voronoi tessellation $\cG_{d,\gamma,\lambda}=\cL_{d,\gamma}(\widetilde f_{\lambda})$ the function $\widetilde f_{\lambda}$ is a solution of the following integral equation
    \[
       (\frI^{\frac{1}{2}}f)(p)={1\over \sqrt{\pi}}\int_{-\infty}^p(p-t)^{-{1\over 2}}f(t)\dint t=\lambda^{-{1\over 2}}f(p),
    \]
    which is Volterra integral equation of the second kind.
\end{remark}

\section{Typical cell of dual Poisson-Laguerre tessellation}\label{sec:typcell}

In this section we consider the dual model $\cL^*_{d,\gamma}(f)=\cL^*_d(\eta_{f,\gamma})$ of the random Laguerre tessellation $\cL_{d,\gamma}(f)$ with $f$ being admissible. In this case according to Theorem \ref{thm:FICondition} $\cL_{d,\gamma}(f)$ is an almost surely normal random tessellation and, hence, its dual $\cL^*_{d,\gamma}(f)$ is an almost surely simplicial random tessellation as described in Section \ref{sec:DualModel}. Recall (see Section \ref{sec:DualModel}, Equation \eqref{eq:Astar}) that in this setting the tessellation $\cL^*_{d,\gamma}(f)$ is itself a random Laguerre tessellation constructed using the point process defined as
\[
    \eta_{f,\gamma}^*(\omega):=\begin{cases}
        \sum_{z\in\cF_0(\cL_d(\eta_{f,\gamma}(\omega)))}\delta_{(z,-K_z)},\qquad &\text{if }\cL_d(\eta_{f,\gamma}(\omega))\text{ is normal},\\
        \delta_0,\qquad&\text{otherwise},
    \end{cases}
\]
for any $\omega\in\Omega$, where given a vertex $z\in\cF_0(\cL_d(\eta_{f,\gamma}(\omega)))$ the value $K_z$ is the unique number such that there are exactly $d+1$ distinct points $x_1,\ldots,x_{d+1}\in (\eta_{f,\gamma}\cap \Pi_{(z,K_z)})$ and $\eta_{f,\gamma}\cap \inter \Pi_{(z,K_z)}^{\downarrow}=\emptyset$. Further note that since $\eta_{f,\gamma}$ satisfies \hyperref[item:P4]{(P4)} almost surely, we have that almost surely for any distinct points $x_1,\ldots,x_{d+1}$ of $\eta_{f,\gamma}$ there is a unique downward paraboloid $\Pi(x_1,\ldots,x_{d+1})$ containing them. This leads to the following alternative representation
\begin{equation}\label{eq:DualProcess}
\eta_{f,\gamma}^*={1\over (d+1)!}\sum_{(x_1,\ldots,x_{d+1})\in (\eta_{f,\gamma})_{\neq}^{d+1}}\delta_{{\rm Ref}(\apex \Pi(x_1,\ldots,x_{d+1}))}{\bf 1}(\eta_{f,\gamma}\cap \inter\Pi^{\downarrow}(x_1,\ldots,x_{d+1})=\emptyset).
\end{equation}

\subsection{Definition of the (volume weighted) typical cell}

Since the point process $\eta_{f,\gamma}$ is stationary in the spatial component and the height coordinate $h$ of a point $(v,h)\in \eta_{f,\gamma}$ does not depend on the location $v$, the tessellation $\cL^*_{d,\gamma}(f)$ is a stationary random tessellation in $\RR^d$. We denote by $\cC'$ the space of non-empty compact subsets of $\RR^d$ equipped with the Fell topology and the corresponding Borel $\sigma$-algebra $\cB(\cC')$. Under the additional assumption that
\begin{equation}\label{eq:IntensityFiniteness}
    \alpha(f,\gamma):=\EE \sum_{(v,h)\in \eta^*_{f,\gamma}}{{\bf 1}}(v\in[0,1]^d)\in (0,\infty),
\end{equation}
the process $\eta_{f,\gamma}^*$ is a stationary marked point process in $\RR^d$ (in the sense of \cite[Definition 3.5.1]{SW}) with mark space $\RR$. The stationarity in this case means that 
$$
    \theta_x(\eta_{f,\gamma}^*):=\sum_{(v,h)\in \eta_{f,\gamma}^*}\delta_{(v-x,h)}\overset{d}{=}\eta_{f,\gamma}^*
$$
for all $x\in\RR^d$, where $\theta_x:(v,h)\mapsto (v-x,h)$. According to \cite[Theorem 3.5.2]{SW} the Palm measure $\PP^0$ of $\eta_{f,\gamma}^*$ is given by
$$
    \PP^0(A\times B)={1\over \alpha(f,\gamma)}\EE\sum_{(v,h)\in \eta^*_{f,\gamma}}{{\bf 1}}(h\in A){\bf 1}(v\in[0,1]^d){\bf 1} (\theta_v(\eta_{f,\gamma}^*)\in B),
$$
where $A\in \cB(\RR),\, B\in \cN(\RR^d\times \RR)$. Let $\eta^{*,0}_{f,\gamma}$ be a point process with distribution $\PP^0(\RR\times \cdot)$. In particular almost surely
$\eta^{*,0}_{f,\gamma}$ contains an atom of the form $x_0=(0,h')$ for some $h'\in \RR$. Then we define \textit{the typical cell} $Z_{d,\gamma}(f)$ of $\cL^*_{d,\gamma}(f)$ as the Laguerre cell of $x_0$ with respect to the process $\eta^{*,0}_{f,\gamma}$. In other words the typical cell $Z_{d,\gamma}(f)$ is a random polytope with distribution
\[
  \PP_{f,\gamma}(A):=\frac{1}{\alpha(f,\gamma)} \EE \sum_{(v,h)\in \eta^*_{f,\gamma}} {\bf 1}(v\in[0,1]^d){\bf 1}(C((v,h),\eta^*_{f,\gamma})-v\in A),  \qquad A\in \cB(\cC').
\] 
It should be noted that under \eqref{eq:IntensityFiniteness} the random tessellation $\cL^*_{d,\gamma}(f)$ can be identified with a stationary particle process $X_f:=\sum_{t\in \cL^*_{d,\gamma}(f)}\delta_t$. In this case the typical cell $Z_{d,\gamma}(f)$ has the same distribution (up to translation) as the typical grain of $X_f$ (see \cite[Section 4.1 - 4.2]{SW}).

Motivated by the definition of the typical cell we introduce a slightly more general concept, namely \textit{the $\nu$-weighted typical cell}. For a given $\nu\in\RR$ we define a probability measure $\PP_{f,\gamma,\nu}$ on $\cC'$ as 
\begin{equation*}
    \PP_{f,\gamma,\nu}(A) := {1\over \alpha(f,\gamma,\nu)}\EE\sum_{(v,h)\in \eta^*_{f,\gamma}}{\bf 1}(v\in[0,1]^d){\bf 1}(C((v,h),\eta^*_{f,\gamma})-v\in A)\Vol(C((v,h),\eta^*_{f,\gamma}))^\nu
\end{equation*}
for $A\in  \cB(\cC')$, where  $\alpha(f,\gamma,\nu)$ is the normalizing constant given by
\begin{equation*}
    \alpha(f,\gamma,\nu):= \EE \sum_{(v,h)\in \eta^*_{f,\gamma}}{{\bf 1}}(v\in[0,1]^d)\Vol(C((v,h),\eta^*_{f,\gamma}))^\nu=\alpha(f,\gamma)\EE[\Vol(Z_{d,\gamma}(f))^\nu],
\end{equation*}
with additional assumption that $\alpha(f,\gamma,\nu)\in (0,\infty)$. Let $Z_{d,\gamma,\nu}(f)$ be the random polytope with distribution $\PP_{f,\gamma,\nu}$. We note that $\PP_{f,\gamma}=\PP_{f,\gamma,0}$ and, thus $Z_{d,\gamma,0}(f)$ has the same distribution as the typical cell of $\cL^*_{d,\gamma}(f)$. At the same time it is known \cite[Theorem 10.4.1 and Equation (10.4)]{SW} that the volume-weighted version $Z_{d,\gamma,1}(f)$ of the typical cell has the same distribution as the zero-cell of $\cL^*_{d,\gamma}(f)$, which is almost surely the unique cell of $\cL^*_{d,\gamma}(f)$ containing the origin. 

\subsection{Main results: stochastic representation for the typical cell}

The aim of this section is to find a good description of the distribution $\PP_{f,\gamma,\nu}$ in terms of the function $f$. The case when $f\colon\RR\to \RR_+$ satisfying $f(x)=0$ for all $x>0$, and $f\in L^1(\RR)$ (see Example \ref{ex:IndQMarking}) has been considered in \cite{LZ08, Ldoc}, where the exact description of the distribution $\PP_{f,\gamma}$ has been obtained. In \cite[Theorem 4.5]{GKT20} and \cite[Theorem 5.1]{GKT21} an explicit representation of the distribution of $Z_{d,\gamma,\nu}(f)$, for the cases $f=f_{d,\beta}$, $f=f'_{d,\beta}$ and $f=\widetilde f_{\lambda}$, defined by \eqref{eq:BetaModel}, \eqref{eq:BetaPrimeModel} and \eqref{eq:GaussianModel}, respectively, has been obtained for $\nu\ge -1$ and additionally $\nu<2\beta-d$ in the $\beta'$-case. More precisely, for these three cases the following relation holds
\begin{equation}\label{eq:TypicalCellRepresentation}
    Z_{d,\gamma,\nu}(f)\overset{d}{=}\conv(RX_1,\ldots,RX_{d+1}),
\end{equation}
where
\begin{enumerate}
\item[(a)] $R$ is a non-negative random variable whose density is proportional to 
    \begin{align*}
        \beta\text{-model}:\quad r^{(d+1)^2+\nu d+2(d+1)\beta}e^{-\gamma\,c_{d+1,\beta}(\pi c_{d+2,\beta})^{-1}r^{d+2+2\beta}},\\
        \beta'\text{-model}:\quad r^{(d+1)^2+\nu d-2(d+1)\beta}e^{-\gamma\,c'_{d+1,\beta}(\pi c'_{d+2,\beta})^{-1}r^{d+2-2\beta}} ,
    \end{align*}
    and $R\equiv 1$ for the Gaussian model;
\item[(b)] $(X_1,\ldots,X_{d+1})$ are random points in $\RR^d$ whose joint density is proportional to
    \begin{equation}\label{eq:JointDensityX}
        \Delta_{d}(x_1,\ldots,x_{d+1})^{\nu+1}   \prod\limits_{i=1}^{d+1}u(x_i),
    \end{equation}
    with $u(x)=(1-\|x\|^2)^{\beta}{\bf 1}(\|x\|\leq 1)$ in the case of the $\beta$-model, $u(x)=(1+\|x\|^2)^{-\beta}$ for the $\beta'$-model and $u(x)=e^{-\lambda\|x\|^2}$ for the Gaussian model;
\item[(c)] $R$ is independent of $(X_1,\ldots,X_{d+1})$.
\end{enumerate}
It should be noted that in \cite[Theorem 5.1]{GKT21} only the case $\lambda=1/2$ for the Gaussian model $\widetilde f_{\lambda}$ has been considered, but the results generalise straightforward to the case of general $\lambda>0$.

For a general admissible function $f$ we cannot hope to have a representation of the form \eqref{eq:TypicalCellRepresentation} with independent components $R$ and $(X_1,\ldots,X_{d+1})$. On the other hand under some additional integrability conditions on the function $f$ we can still provide a representation of the form 
\[
    Z_{d,\gamma,\nu}(f)\overset{d}{=}\conv(Y_1,\ldots,Y_{d+1}),
\]
where $Y_1,\ldots,Y_{d+1}$ are random points in $\RR^d$, whose joint distribution is described explicitly in terms of fractional integrals and derivatives of $f$. This will be the main result of this section (see Theorem \ref{thm:typcell} below).

Before we can state our main result we need to introduce some additional notations. Let $E$ be some interval of the type (i), (ii) or (iii) and let $f\in L_{\rm loc}^{1,+}(E)$. For any $p\in \inter E$, $d\ge 1$ and $\alpha\in \RR$ we define
\begin{align*}
    &K_{d,f}^\alpha(p):=\int_{(\RR^d)^{d+1}} \Delta_d(x_1,\dots,x_{d+1})^\alpha \prod_{i=1}^{d+1} f(p-\|x_i\|^2)\dd x_1\dots \dd x_{d+1},\\
    &J_{d,f}^\alpha(p):=\int_{(\RR^d)^{d}} \nabla_d(x_1,\dots,x_d)^\alpha \prod_{i=1}^{d}f(p-\|x_i\|^2)\dd x_1\dots \dd x_d.
\end{align*}

\begin{theorem}\label{thm:typcell}
    Let $d\ge 2$ and let $f$ be admissible. Let $\nu\in \RR$ be such that
    \begin{equation}\label{eq:FiniteNormalizationConst}
        \int_{E}\exp\big(-\gamma\pi^{\frac{d}{2}}(\frI^{\frac{d}{2}+1}f)(p)\big) K^{\nu+1}_{d,f}(p)\dd p\in (0,\infty),
    \end{equation}
    where $\frI^{\frac{d}{2}+1}f$ is defined by \eqref{eq:Frac_Int}. Then
    \begin{align*}
        \PP_{f,\gamma,\nu}(\cdot)&=\frac{{(2\gamma)^{d+1}}}{(d+1)\alpha(f,\gamma,\nu)} \int_{(\RR^{d})^{d+1}}\dd y_1\dots \dd y_{d+1} \, \int_{E}\dd p\,{\bf 1}(\conv(y_1,\dots,y_{d+1})\in \cdot) \\
        &\qquad\times \exp\big(-\gamma\pi^\frac{d}{2} (\frI^{\frac{d}{2}+1}f)(p)\big) \Delta_{d}(y_1,\dots,y_{d+1})^{\nu+1}\prod_{i=1}^{d+1}f(p-\|y_i\|^2),
    \end{align*}
    where
   \begin{align}
        \alpha(f,\gamma,\nu)&=\frac{{(2\gamma)^{d+1}}}{d+1}\int_{E}\exp\big(-\gamma\pi^{\frac{d}{2}}(\frI^{\frac{d}{2}+1}f)(p)\big) K^{\nu+1}_{d,f}(p)\dd p.
        \label{eq:BoundNormalizationConst}
    \end{align}
    Moreover if we additionally assume $(\frI^{\frac{d}{2}+\lceil\frac{\nu+1}{2}\rceil}f)(p)<\infty$ for any $p\in E$ we have
    \begin{align*}
        \alpha(f,\gamma,\nu)&=\frac{{(2\gamma)^{d+1}}\pi^{{d(d+1)\over 2}}}{(d+1)(d!)^{\nu+1}} \prod_{k=1}^{d}\frac{\Gamma\left(\frac{\nu+1+k}{2}\right)}{\Gamma(\frac{k}{2})}\int_{E}e^{-\gamma\pi^{\frac{d}{2}}(\frI^{\frac{d}{2}+1}f)(p)} \Big(\frD^{\frac{\nu+1}{2}}\big[\big(\frI^{\frac{d+\nu+1}{2}}f\big)^{d+1}\big]\Big)(p) \dd p.
    \end{align*}
\end{theorem}

The next theorem provides a simple representation of the functions $K_{d,f}^{\alpha}$ and $J_{d,f}^{\alpha}$ in terms of fractional integrals and derivatives of the function $f$.

\begin{theorem}\label{thm:kd_jd}
    Let $E$ be an interval of type (i), (ii) or (iii) and let $f\in L_{\rm loc}^{1,+}(E)$. Then for any $\alpha> -1$ and $d\ge 1$ we get
    \begin{equation}\label{eq:JFormula}
        J_{d,f}^\alpha(p)= \left(\prod_{k=1}^{d}\frac{\Gamma\left(\frac{\alpha+k}{2}\right)}{\Gamma(\frac{k}{2})}\right) \left(\pi^{\frac{d}{2}}(\frI^{\frac{d+\alpha}{2}}f)(p)\right)^d.
    \end{equation}
    Further for any $\alpha\ge 0$, $d\ge 2$ and $p\in\inter E$ we get
    \begin{equation}\label{eq:KInequality}
        K_{d,f}^\alpha(p)\leq {(d+1)^{\max(1,\alpha)}\over (d!)^{\alpha}}{\pi^{d(d+1)\over 2}\Gamma({d+\alpha\over 2})^d\over \Gamma({d\over 2})^d}\Big((\frI^{{d+\alpha\over 2}}f)(p)\Big)^d(\frI^{{d\over 2}}f)(p),
    \end{equation}
    and for any $\alpha\ge 0$, $d\ge 2$ such that $(\frI^{{d\over 2}+\lceil{\alpha\over 2}\rceil}f)(p)<\infty$ for any $p\in\inter E$, we have
    \begin{equation}\label{eq:KFormula}
        K_{d,f}^\alpha(p)=\frac{\pi^{\frac{(d+1)d}{2}}}{(d!)^\alpha} \left(\prod_{k=1}^{d}\frac{\Gamma\left(\frac{\alpha+k}{2}\right)}{\Gamma(\frac{k}{2})}\right)\Big(\frD^{\frac{\alpha}{2}}\big[(\frI^{\frac{d+\alpha}{2}}f)^{d+1}\big]\Big)(p)
    \end{equation}
    for almost all $p\in\inter E$. In particular the equality in \eqref{eq:KFormula} holds if $K_{d,f}^{\alpha}$ is continuous in $p$.
\end{theorem}

The above theorem is of independent interest, since it in particular allows to compute moments for the volume in some models of random simplices and polytopes.

\begin{corollary}\label{cor:RandomSimplexVolume}
    Let $E$ be an interval of type (i), (ii) or (iii) and let $f\in L_{\rm loc}^{1,+}(E)$ be a function which is monotone and continuous almost everywhere on $\inter E$. Let $p\in \inter E$ be such that $(\frI^{d\over 2}f)(p)\in (0,\infty)$ {(see \eqref{eq:Frac_Int} for the definition of $\frI^{d\over 2}f$)} and let $X_1,\ldots,X_{d+1}$ be i.i.d. isotropic random points in $\RR^d$ whose distribution has density 
    $$
        g(x)=\pi^{-{d\over 2}}\big((\frI^{d\over 2}f)(p)\big)^{-1}f(p-\|x\|^2).
    $$
    Then for any $\alpha>-1$ and $d\ge 1$ we have
    \begin{equation}\label{eq:ExpectationParallelogram}
        \EE \big[\nabla_d(X_1,\ldots,X_{d})^{\alpha}\big]=\prod_{k=1}^{d}\frac{\Gamma\left(\frac{\alpha+k}{2}\right)}{\Gamma(\frac{k}{2})} \left({(\frI^{\frac{d+\alpha}{2}}f)(p)\over (\frI^{d\over 2}f)(p)}\right)^d,
    \end{equation}
    and for any $\alpha\ge 0$ and $d\ge 2$ such that $(\frI^{{d\over 2}+\lceil{\alpha\over 2}\rceil}f)(t)<\infty$ for some $\varepsilon>0$ and all $t\in (p-\varepsilon,p+\varepsilon)\subset\inter E$ we obtain
    \begin{equation}\label{eq:ExpectationSimplex}
        \EE \big[\Delta_d(X_1,\ldots,X_{d+1})^{\alpha}\big]=\frac{1}{(d!)^\alpha} \prod_{k=1}^{d}\frac{\Gamma\left(\frac{\alpha+k}{2}\right)}{\Gamma(\frac{k}{2})}{\big(\frD^{\frac{\alpha}{2}}\big[(\frI^{\frac{d+\alpha}{2}}f)^{d+1}\big]\big)(p)\over \big((\frI^{d\over 2}f)(p)\big)^{d+1}}.
    \end{equation}
\end{corollary}

\begin{remark}
    Let us point out that for integer values of $\alpha$ the formula \eqref{eq:ExpectationParallelogram} can also be deduced from a more general result of Miles \cite[Theorem 1]{Miles_IRS} (in particular formula \cite[Equation (23)]{Miles_IRS}).
\end{remark}

\begin{remark}
    Let us consider a few special cases. Let $f(t)=t^{\beta}{\bf 1}(t>0)$, $\beta>-1$. Note that $f\in L_{\rm loc}^{1,+}([0,\infty))$ and, hence, by Proposition \ref{prop:FICondition2} we have $(\frI^{\alpha}f)(p)<\infty$ for any $p\in(0,\infty)$ and $\alpha>0$. Then, since for any $\beta>-1$ we have
    \begin{equation}\label{eq:30.05.24_1}
        (\frI^{d+\alpha\over 2}f)(p)={\Gamma(\beta+1)\over \Gamma({d+\alpha\over 2}+\beta+1)}p^{{d+\alpha\over 2}+\beta},
    \end{equation}
    the density $g$ takes the form
    $$
        g(x)={\Gamma({d\over 2}+\beta+1)\over \pi^{d\over 2}p^{{d\over 2}+\beta}\Gamma(\beta+1)}(p-\|x\|^2)^{\beta}{\bf 1}\big(\|x\|<\sqrt{p}\big).
    $$
    Further since for any $\alpha>0$ we obtain
    \begin{equation}\label{eq:30.05.24_2}
        \big(\frD^{\frac{\alpha}{2}}\big[(\frI^{\frac{d+\alpha}{2}}f)^{d+1}\big]\big)(p)=\Big({\Gamma(\beta+1)\over \Gamma({d+\alpha\over 2}+\beta+1)}\Big)^{d+1}{\Gamma({(d+1)(d+\alpha)\over 2}+(d+1)\beta+1)\over \Gamma({d(d+1)+\alpha d\over 2}+(d+1)\beta+1)}p^{{d(d+1)+\alpha d\over 2}+(d+1)\beta}
    \end{equation}
    we conclude
    \begin{align*}
        \EE \big[\nabla_d(X_1,\ldots,X_{d})^{\alpha}\big]&=p^{\alpha d\over 2}\prod_{k=1}^{d}\frac{\Gamma\left(\frac{\alpha+k}{2}\right)}{\Gamma(\frac{k}{2})} \left({\Gamma({d\over 2}+\beta+1)\over \Gamma({d+\alpha\over 2}+\beta+1)}\right)^d,\\
        \EE \big[\Delta_d(X_1,\ldots,X_{d+1})^{\alpha}\big]&=\frac{p^{\alpha d\over 2}}{(d!)^\alpha} \prod_{k=1}^{d}\frac{\Gamma\left(\frac{\alpha+k}{2}\right)}{\Gamma(\frac{k}{2})}\Big({\Gamma({d\over 2}+\beta+1)\over \Gamma({d+\alpha\over 2}+\beta+1)}\Big)^{d+1}{\Gamma({(d+1)(d+\alpha)\over 2}+(d+1)\beta+1)\over \Gamma({d(d+1)+\alpha d\over 2}+(d+1)\beta+1)},
    \end{align*}
    where the first formula holds for $\alpha>-1$ and the second holds for $\alpha\ge 0$. The second formula with $p=1$ has been derived in \cite[Equation (74)]{Miles_IRS} for integer $\alpha\ge 0$ and can be found in \cite[Proposition 2.8]{KTT} for general $\alpha\ge 0$. The first formula with $p=1$ is due to Mathai and can be found in \cite[Theorem 19.2.5]{charalambides2000probability}. In the same way by choosing $f(t)=(-t)^{-\beta}{\bf 1}(t<0)$ with $\beta>{d\over 2}+1$ and $p=-1$ and $f(t)=e^{t/2}$ with $p=0$ we recover formulas (72) and (70) from \cite{Miles_IRS}.
\end{remark}

As a direct consequence of Theorem \ref{thm:typcell} we obtain the following formulas for the moments of the volume of a $\nu$-weighted typical cell.

\begin{theorem}\label{thm:moments}
    Let $d\ge 2$ and let $f$ be admissible. Let $\nu, s\in\RR$ be such that \eqref{eq:FiniteNormalizationConst} holds for $\nu$ and $\nu+s$. Then we have
    \begin{align*}
        \EE \Vol(Z_{d,\gamma,\nu}(f))^s = {\alpha(f,\gamma,\nu+s)\over \alpha (f,\gamma,\nu)},
    \end{align*}
    where $\alpha(f,\gamma,\nu)$ is given by \eqref{eq:BoundNormalizationConst}. If, moreover, $(\frI^{{d\over 2}+\lceil{\max\{\nu+s,\nu\}+1\over 2}\rceil})(p)<\infty$ for any $p\in E$, then 
    \begin{align*}
        \EE \Vol(Z_{d,\gamma,\nu}(f))^s ={1\over (d!)^s}\prod_{k=1}^d{\Gamma({\nu+s+1+k\over 2})\over \Gamma({\nu+1+k\over 2})} {\int_{E}e^{-\gamma\pi^{\frac{d}{2}}(\frI^{\frac{d}{2}+1}f)(p)}\big(\frD^{\frac{\nu+s+1}{2}}\big[\big(\frI^{\frac{d+\nu+s+1}{2}}f\big)^{d+1}\big]\big)(p) \dd p\over \int_{E}e^{-\gamma\pi^{\frac{d}{2}}(\frI^{\frac{d}{2}+1}f)(p)} \big(\frD^{\frac{\nu+1}{2}}\big[\big(\frI^{\frac{d+\nu+1}{2}}f\big)^{d+1}\big]\big)(p) \dd p},
    \end{align*}
\end{theorem}
\begin{remark}
    The moments of the volume of the $\nu$-weighted typical cell of the $\beta$-, $\beta'$- and Gaussian-Delaunay tessellation, where the function $f$ is defined by \eqref{eq:BetaModel}, \eqref{eq:BetaPrimeModel} and \eqref{eq:GaussianModel}, respectively, are explicitly given in \cite[Theorem~5.1]{GKT20} and \cite[Corollary~5.7]{GKT21}. Thus, for example, combining the formula from Theorem \ref{thm:moments} with \eqref{eq:30.05.24_1} and \eqref{eq:30.05.24_2} we recover the formula from \cite[Theorem~5.1]{GKT20} for the $\beta$-model.
\end{remark}

Before we move to the proof of the above results let us discuss the integrability condition \eqref{eq:FiniteNormalizationConst}. Let us point out that for any $f\in L_{\rm loc}^{1,+}(E)$ satisfying \hyperref[item:F1]{(F1)} (see Definition \ref{def:admissible}) we have 
$$
    \exp\big(-\pi^{\frac{d}{2}}(\frI^{\frac{d}{2}+1}f)(p)\big)>0
$$ 
on $E$. At the same time if $f$ is strictly positive on some set $I\subseteq E$ of positive Lebesgue measure, then 
\begin{align*}
    K^{\nu+1}_{d,f}(p)&=\int_{(\RR^d)^{d+1}} \Delta_d(x_1,\dots,x_{d+1})^{\nu+1} \prod_{i=1}^{d+1} f(p-\|x_i\|^2)\dd x_1\dots \dd x_{d+1}>0
\end{align*}
for any $p\in E$ and hence   
$$
    \int_{E}\exp\big(-\pi^{\frac{d}{2}}(\frI^{\frac{d}{2}+1}f)(p)\big) K^{\nu+1}_{d,f}(p) \dd p>0.
$$
On the other hand in general it is not easy to check that the above integral is finite. In the proposition below we present a few cases of functions and corresponding values $\nu$ such that \eqref{eq:FiniteNormalizationConst} holds.
In order to formulate the proposition we need some additional definitions. A measurable function $f\colon\RR_+\to \RR_+$ is called \textit{regularly varying at infinity with index $\rho$} if
\[
    \lim_{x\to\infty}{f(\lambda x)\over f(x)}=\lambda^{\rho},
\]
for all $\lambda>0$. The function $f(x)$ is called \textit{regularly varying at $0$ with index $\rho$} if the function $f(x^{-1})$ is regularly varying at infinity with index $\rho$. For more details we refer the reader to \cite{bingham_regular_variation} and \cite{resnick_ERVPP}.

\begin{proposition}\label{prop:FiniteNormalizationConst}
Let $d\ge 2$, $E$ be an interval of type (i), (ii) or (iii) and let $f\in L_{{\rm loc}}^{1,+}(E)$. If one of the conditions
\begin{itemize}
    \item[(i)]  $f$ satisfies \hyperref[item:F1]{(F1)} and $\nu=1$;
    \item[(ii)] $f\in L^1(\RR)$, $f$ satisfies \hyperref[item:F1]{(F1)} and $-1\leq \nu\leq 1$;
    \item[(iii)] $E=[0,
    \infty)$, $f$ is regularly varying at $+\infty$ with index $\beta>-1$ and $\nu\ge -1$;
    \item[(iv)] $E=(-\infty,0)$, $f$ is regularly varying at $0$ with index $\beta>{d\over 2}+1$, $f$ satisfies $(\frI^{\alpha}f)(p)<\infty$, $p<0$ for some $\alpha> {d\over 2}+1$ and $-1\leq \nu<2\min(\alpha,\beta)-d-1$;
\end{itemize}
    holds, then 
    $$
        \int_{E}\exp\big(-\gamma\pi^{\frac{d}{2}}(\frI^{\frac{d}{2}+1}f)(p)\big) K^{\nu+1}_{d,f}(p) \dd p<\infty,
    $$
    where \hyperref[item:F1]{(F1)} is given in Definition \ref{def:admissible} and $\frI^{\alpha}f$ is defined in \eqref{eq:Frac_Int}.
\end{proposition}

\begin{remark}
    Note that the function $f_{d,\beta}$ for $\beta>-1$ satisfies (iii), while the function $f'_{d,\beta}$ for $\beta>{d\over 2}+1$ satisfies (iv) with $\alpha=\beta$. For the function $\widetilde f_{\lambda}$ one can also easily ensure that \eqref{eq:FiniteNormalizationConst} holds. Hence, Theorem \ref{thm:typcell} is applicable in these cases and in particular we recover the results of \cite[Theorem 4.5]{GKT20} and \cite[Theorem 5.1]{GKT21} (see also Remark \ref{rmk:decomposition} below).
\end{remark}

\begin{remark}
     In Example \ref{ex:IndQMarking} we considered the case when $\xi$ is an independent $\QQ$-marking of a homogeneous Poisson point process on $\RR^d$, where $\QQ$ is absolutely continuous with respect to the Lebesgue measure. This setting corresponds to the situation $\eta_{q,\gamma}$, where $q:\RR\to\RR_+$, $q\in L^1(\RR)$ and $q$ is admissible. In this case according to Proposition \ref{prop:FiniteNormalizationConst} (ii) Theorem \ref{thm:typcell} applies with $\nu=0$ and the obtained formula coincides with the one from \cite[Theorem 3.3.1]{Ldoc}. In order to see this, one should substitute  $u_i=y_i/\|y_i\|$ and $r_i=(\|u_i\|^2-t)^{1/2}$ for $i=1,\dots,d+1$, taking into account $\FF(dr)=2f(-r^2)r\dd r$.
\end{remark}

\subsection{Canonical decomposition}

Let us turn back to the representation of the typical cell $Z_{d,\gamma,\nu}(f)$ in the form \eqref{eq:TypicalCellRepresentation}, namely
\[
    Z_{d,\gamma,\nu}(f)\overset{d}{=}\conv(RX_1,\ldots,RX_{d+1}),
\]
where $R$ and $(X_1,\ldots,X_{d+1})$ are independent components and the distribution of \\$(X_1,\ldots,X_{d+1})$ has a density of type \eqref{eq:JointDensityX}. As it was mentioned above for a general admissible function $f$ we will not have such a representation as clearly follows from the formula presented in Theorem \ref{thm:typcell}. On the other hand let us consider an admissible function $f$ satisfying the following assumption: $f\not\equiv 0$ and suppose there exist functions $\phi,\psi\colon \RR \to\RR_+$ with $\inter E\subset \supp(\phi):= \{x\in \RR: \phi(x)\neq 0\}$ and $\psi(0)=1$, such that
\begin{align}\label{decomp}
    f(p-\phi(p)^2s^2)=f(p)\psi(s^2),\qquad\qquad p\in \inter E,\, s\in \RR.
\end{align}
One might extend \eqref{decomp} to $p\in \RR$, but for our purpose of describing the distribution of the typical cell $Z_{d,\gamma,\nu}(f)$ the decomposition as stated above is sufficient. 

In this case we may obtain the following corollary.

\begin{corollary}\label{cor:decomposition}
    Let $d\ge 2$ and let $f$ be admissible satisfying \eqref{decomp}. Let $\nu\in \RR$ be such that \eqref{eq:FiniteNormalizationConst} holds. Then
    \[
        Z_{d,\gamma,\nu}(f)\overset{d}{=}\conv(\phi(Z)X_1,\ldots,\phi(Z)X_{d+1}),
    \]
    where $Z$ is a random variable, independent of $(X_1,\ldots,X_{d+1})$, whose density is proportional to
    \[
        \phi(z)^{d(d+\nu+2)}f(z)^{d+1}e^{-\gamma\pi^{d\over 2}(\frI^{{d\over 2}+1}f)(z)}
    \]
    and  $(X_1,\ldots,X_{d+1})$ are random points in $\RR^d$ whose joint density is proportional to
    \[
        \Delta_d(x_1,\ldots,x_{d+1})^{\nu+1}\prod_{i=1}^{d+1}\psi(\|x_i\|^2).
    \]
\end{corollary}

\begin{remark}\label{rmk:decomposition}
    Note that the function $f_{d,\beta}$ satisfies \eqref{decomp} with $\phi(x)=\sqrt{x}{\bf 1}(x\ge 0)$ and $\psi(x)=(1-x)^{\beta}{\bf 1}(x\leq 1)$. At the same time the function $f'_{d,\beta}$ satisfies \eqref{decomp} with $\phi(x)=\sqrt{-x}{\bf 1}(x<0)$ and $\psi(x)=(1+x)^{-\beta}$ and $\widetilde f_{\lambda}$ satisfies \eqref{decomp} with $\phi(x)\equiv 1$ and $\psi(x)=e^{-x}$. Taking this into account we directly recover the results of \cite[Theorem 4.5]{GKT20} and \cite[Theorem 5.1]{GKT21}.
\end{remark}

On the other hand the condition \eqref{decomp} appeared to be very restrictive. In particular under the additional assumption that $f$ is differentiable there are only three families of functions satisfying \eqref{decomp}.

\begin{proposition}\label{lm:decomp}
    Let $f$ be admissible satisfying \eqref{decomp}. Moreover let $f$ and $\phi$ be differentiable on $\inter E$. Then $f$ has the following form 
    \begin{align}
        &f(x)= c_1(x+c_2)^\beta, &&x> -c_2, \beta > -1,\, c_1 > 0,\, c_2 \in \RR; \label{eq:BetaCase}\\
        &f(x)=c_1(-(x+c_2))^{-\beta}, &&x < -c_2, \beta > \frac{d}{2}+1,\, c_1 < 0,\, c_2 \in \RR;\label{eq:BetaPrimeCase}\\
        &f(x)=c_1e^{\lambda x}, &&x \in \RR, c_1>0,\, \lambda > 0.\label{eq:GaussianCase}
    \end{align}
\end{proposition}

\begin{remark}
    It should be mentioned that due to Proposition \ref{prop:compwlin} we have $\cL_{d,\gamma}(f)\overset{d}{=}\sqrt{c_1}\cV_{d,\beta,\widetilde\gamma}$ if $f$ is of the form \eqref{eq:BetaCase}; $\cL_{d,\gamma}(f)\overset{d}{=}\sqrt{c_1}\cV'_{d,\beta,\widetilde\gamma}$ if $f$ is of the form \eqref{eq:BetaPrimeCase}; and $\cL_{d,\gamma}(f)\overset{d}{=}\sqrt{c_1}\cG_{d,\lambda,\widetilde\gamma}$ if $f$ is of the form \eqref{eq:GaussianCase}, where $\widetilde\gamma=c_1^{d/2+1}\gamma$. This fact may be used as a characterization of this three families of Poisson-Laguerre tessellations. A Similar characterization exists for $\beta$-, $\beta'$- and Gaussian-polytopes and is often referred to as canonical decomposition. For further details and relation to extreme-value distributions we refer the reader to \cite[Section 12]{Miles_IRS} and \cite[Section 3]{KZT20}.
\end{remark}

\subsection{Proofs}

\begin{proof}[Proof of Theorem \ref{thm:kd_jd}]
    We start by computing $J_{d,f}^{\alpha}(p)$ for any $\alpha>-1$, $d\ge 1$ and $p\in E$. Using the linear Blaschke-Petkanschin formula (see \cite[Theorem 7.2.1]{SW}) and the base-times-height formula 
    \[
        \nabla_d(x_1,\ldots,x_d)=\dist(x_d,\lin(x_1,\ldots,x_{d-1}))\nabla_{d-1}(x_1,\ldots,x_{d-1}),
    \]
    where $\dist(x, F)=\inf_{y\in F}\|x-y\|$ for any closed $F\subseteq \RR^d$ and $x\in\RR^d$ we get
    \begin{align*}
        J_{d,f}^{\alpha}(p)
        &= \int_{\RR^d} \int_{(\RR^d)^{d-1}} \nabla_{d-1}(x_1,\dots,x_{d-1})^\alpha \dist(x_{d},\lin(x_1,\dots,x_{d-1}))^{\alpha} \\ &\hspace{4cm} \times \prod_{i=1}^{d}f(p-\|x_i\|^2) \dd x_1 \dots \dd x_d\\
        &= {\omega_d\over 2}\int_{\RR^d} \int_{G(d,d-1)}\int_{L^{d-1}} \nabla_{d-1}(x_1,\dots,x_{d-1})^{\alpha+1} \dist(x_{d},L)^{\alpha} \\ &\hspace{4cm} \times \prod_{i=1}^{d}f(p-\|x_i\|^2)\lambda_L^{d-1}(\dd (x_1,\dots,x_{d-1}))\,\nu_{d-1}(\dd L)\, \dd x_d.
    \end{align*}
    Using Fubini's theorem and applying the change of variables $\RR^d\to L\times L^{\perp}$, $x_d\mapsto (y,t)$ we obtain 
    \begin{align}\label{eq:29.05.24_1}        
        J_{d,f}^{\alpha}(p) &= {\omega_d\over 2}\int_{G(d,d-1)}\int_{L^{d-1}}\nabla_{d-1}(x_1,\dots,x_{d-1})^{\alpha+1} \prod_{i=1}^{d-1}f(p-\|x_i\|^2) \notag\\ 
        &\qquad \times \int_{L}\int_{L^{\perp}} |t|^{\alpha}f(p-t^2-\|y\|^2)\lambda_{L}(\dd y) \lambda_{L^{\perp}}(\dd t)\lambda_L^{d-1}(\dd (x_1,\dots,x_{d-1}))\,\nu_{d-1}(\dd L),
    \end{align}
    where $\lambda_L$ is the Lebesgue measure restricted to $L$ (see \cite{SW} for more details). Now since $L$ is isometric to $\RR^{d-1}$ and $L^{\perp}$ is isometric to $\RR$ and applying Fubini's theorem we get
    \[
        I(p):=\int_{L}\int_{L^{\perp}} |t|^{\alpha}f(p-t^2-\|y\|^2)\lambda_{L}(\dd y) \lambda_{L^{\perp}}(\dd t)=2\int_{0}^{\infty}t^{\alpha}\int_{\RR^{d-1}} f(p-t^2-\|y\|^2)\dd y\,\dd t.
    \]
    Further, using polar coordinates $y=ru$, $r>0$, $u\in\SS^{d-2}$, change of variables $s=t^2$, $h=r^2$ and \eqref{eq:Frac_Int} we arrive at
    \begin{align*}
        I(p)&={\omega_{d-1}\over 2}\int_{0}^{\infty}s^{{\alpha+1\over 2}-1}\int_{0}^{\infty} f(p-s-h) h^{{d-1\over 2}-1} \dd h\, \dd s\\
        &=\pi^{d-1\over 2}\Gamma\Big({\alpha+1\over 2}\Big)(\frI^{\alpha+1\over 2}\frI^{d-1\over 2}f)(p)=\pi^{d-1\over 2}\Gamma\Big({\alpha+1\over 2}\Big)(\frI^{d+\alpha\over 2} f)(p),
    \end{align*}
    where in the last step we used the semigroup property \eqref{eq:SemigroupProp}. Substituting this into \eqref{eq:29.05.24_1} and noting that $\nu_{d-1}(G(d,d-1))=1$ leads to
    \begin{align*}
    J_{d,f}^{\alpha}(p)&= J_{d-1,f}^{\alpha+1}(p){\pi^{d-{1\over 2}}\Gamma({\alpha+1\over 2})\over \Gamma(\frac{d}{2})}(\frI^{d+\alpha\over 2} f)(p).
    \end{align*}
        Finally since for any $\alpha>-1$ we have
    \begin{align*}
        J_{1,f}^{\alpha}(p)=\int_{\RR}|x|^{\alpha}f(p-x^2)\dd x = \Gamma\Big(\frac{\alpha+1}{2}\Big)(\frI^{\frac{\alpha+1}{2}}f)(p),
    \end{align*}
    we get iteratively 
        \begin{align*}
        &J_{d,f}^{\alpha}(p) = \left(\prod_{k=1}^{d}\pi^{k-{1\over 2}}\frac{\Gamma\left(\frac{\alpha+k}{2}\right)}{\Gamma(\frac{k}{2})}\right) \left((\frI^{\frac{d+\alpha}{2}}f)(p)\right)^d=\left(\prod_{k=1}^{d}\frac{\Gamma\left(\frac{\alpha+k}{2}\right)}{\Gamma(\frac{k}{2})}\right) \left(\pi^{\frac{d}{2}}(\frI^{\frac{d+\alpha}{2}}f)(p)\right)^d.
    \end{align*}

    Next we prove the inequality \eqref{eq:KInequality}. We start by computing an upper bound for the volume of the simplex $\conv(x_1,\dots,x_{d+1})$. Let $X_i:=\{x_1,\dots,x_{i-1},x_{i+1},\dots,x_{d+1}\}$. First note that 
    \[
        \conv(x_1,\dots,x_{d+1})\subseteq \bigcup_{i=1}^{d+1} \conv(0,X_i),
    \]
    and, hence, $\Delta_d(x_1,\dots,x_{d+1}) \le \sum_{i=1}^{d+1} \Delta_d(0,X_i)$.
    For $\alpha> 1$ by the Hölder inequality with values $p=\alpha>1$ and $q=\frac{\alpha}{\alpha-1}$ we obtain
    \begin{align*}
        \Delta_d(x_1,\dots,x_{d+1})^{\alpha} &\le \left(\sum_{i=1}^{d+1} \Delta_d(0,X_i)\right)^{\alpha}\le (d+1)^{\alpha-1} \sum_{i=1}^{d+1} \Delta_d(0,X_i)^{\alpha},   
    \end{align*}
    while for $0\leq \alpha\leq 1$ we use the inequality $(x+y)^{\alpha}\leq x^{\alpha}+y^{\alpha}$, $x,y\ge 0$, to directly conclude 
    $$
        \Delta_d(x_1,\dots,x_{d+1})^{\alpha}\leq \sum_{i=1}^{d+1} \Delta_d(0,X_i)^{\alpha}.
    $$
    Further for $i=1,\dots,d+1$ we get
    \begin{align*}
        \Delta_d(0,X_i)^{\alpha} = \left(\frac{1}{d!}|\det(X_i)|\right)^{\alpha} \le \frac{1}{(d!)^{\alpha}}\prod_{j=1,j\neq i}^{d+1}\|x_j\|^{\alpha}.
    \end{align*}
    Thus, combining these estimates by \eqref{eq:Frac_Int} we get
    \begin{align*}
        K_{d,f}^{\alpha}(p)&\leq {(d+1)^{\max(0,\alpha-1)}\over (d!)^{\alpha}}\sum_{i=1}^{d+1}\int_{\RR^d}f(p-\|x_i\|^2)\dd x_i\prod_{j=1,j\neq i}^{d+1}\int_{\RR^d}f(p-\|x_j\|^2)\|x_j\|^{\alpha}\dd x_j\\
        &={(d+1)^{\max(1,\alpha)}\over (d!)^{\alpha}}{\pi^{d(d+1)\over 2}\Gamma({d+\alpha\over 2})^d\over \Gamma({d\over 2})^d}\Big((\frI^{{d+\alpha\over 2}}f)(p)\Big)^d(\frI^{{d\over 2}}f)(p).
    \end{align*}
    Let now $\alpha\ge 0$ be such that $(\frI^{{d\over 2}+\lceil{\alpha\over 2}\rceil}f)(p)<\infty$ for any $p\in \inter E$. In order to show \eqref{eq:KFormula} we write $K_{d,f}^\alpha(p)$ for $\alpha>0$ and $p\in \inter E$ in terms of $J_{d+1,f}^{\alpha-1}(p)$.
    First, we notice that using the base-times-height formula 
    $$
        \Delta_d(x_0,x_1,\ldots,x_d)={1\over d}\dist(x_0,\aff(x_1,\ldots,x_d))\Delta_{d-1}(x_1,\ldots,x_d)
    $$
    and the relation \cite[Equation (7.6)]{SW}
    $$
        \nabla_d(x_1,\dots,x_d)=d!\Delta_{d}(0,x_1,\dots,x_d),
    $$
    for any $x_1,\dots,x_{d+1}\in \RR^d$ we get 
    \begin{align*}
        \nabla_d(x_1,\dots,x_d)^\alpha&=(d!)^\alpha \Delta_{d}(0,x_1,\dots,x_d)^\alpha\\
        &=((d-1)!)^\alpha \dist(0,\aff(x_1,\dots,x_d))^\alpha \Delta_{d-1}(x_1,\dots,x_d)^\alpha.
    \end{align*}
    Using the affine Blaschke-Petkanschin formula \cite[Theorem 7.2.7]{SW} for any $\alpha>-1$ we get
    \begin{align*}
        J_{d,f}^\alpha(p) 
        &=((d-1)!)^\alpha\int_{(\RR^d)^d} \dist(0,\aff(x_1,\dots,x_d))^\alpha \Delta_{d-1} (x_1,\dots,x_d)^\alpha\\
        &\hspace{4cm}\times\prod_{i=1}^{d} f(p-\|x_i\|^2)\dd x_1\dots \dd x_d\\
        &=((d-1)!)^{\alpha+1}{\omega_d\over 2}\int_{A(d,d-1)}\int_{E^d} \dist(0,E)^\alpha\Delta_{d-1}(x_1,\dots,x_d)^{\alpha+1} \\
        &\hspace{4cm}\times \prod_{i=1}^{d} f(p-\|x_i\|^2)  \lambda_E^d(\dd(x_1,\dots,x_d))\,\mu_{d-1}(\dd E).
    \end{align*}
    Further, using that for any non-negative measurable function $f\colon A(d,d-1)\to \RR_+$ we have 
    $$
        \int_{A(d,d-1)}f(E)\mu_{d-1}(\dint E)={2\over \omega_{d}}\int_{\SS^{d-1}}\int_{0}^{\infty}f(H_t(u))\dint t\,\sigma_{d-1}(\dint u),
    $$
    where $H_t(u)$ denotes the hyperplane with normal vector $u$ and $\dist(0, H_t(u))=t$, we arrive at
    \begin{align*}
       J_{d,f}^\alpha(p)
        &=((d-1)!)^{\alpha+1}\int_{\SS^{d-1}}\int_{0}^{\infty}\int_{H_t(u)^d} t^\alpha \Delta_{d-1}(x_1,\dots,x_d)^{\alpha+1}\\
        &\hspace{3cm}\times\prod_{i=1}^{d} f(p-\|x_i\|^2) \lambda_{H_t(u)}^d(\dd(x_1,\dots,x_d))\,\dd t \,\sigma_{d-1}(\dd u).
    \end{align*}
    Further write $H_t(u) = H_0(u) + t u$ and for $i=1,\dots,d$ let $y_i=x_i-tu$ be the projection of $x_i$ from $H_t(u)$ to $H_0(u)$. Then $\|x_i\|^2 = \|y_i\|^2 + t^2$ and we have 
    \begin{align*}
        &\int_{H_t(u)^d} \Delta_{d-1}(x_1,\dots,x_d)^{\alpha+1} \prod_{i=1}^{d} f(p-\|x_i\|^2) \lambda_{H_t(u)}^d(\dd (x_1,\dots,x_d))\\
        &=\int_{H_0(u)^d} \Delta_{d-1}(y_1,\dots,y_d)^{\alpha+1} \prod_{i=1}^{d} f(p-t^2-\|y_i\|^2)\lambda_{H_0(u)}^d(\dd(y_i,\dots,y_d))=K_{d-1,f}^{\alpha+1}(p-t^2),
    \end{align*}
    where in the last equation we used the fact that $H_0(u)$ is isometric to $\RR^{d-1}$ and $\Delta_{d-1}(y_1,\dots,y_d)$ is rotationally invariant. By the equality above and noting that $\sigma_{d-1}(\SS^{d-1})=\omega_d$ we arrive at
    \begin{align*}
        J_{d,f}^\alpha(p) 
        &=\omega_d ((d-1)!)^{\alpha+1}\int_{0}^{\infty} t^\alpha K_{d-1,f}^{\alpha+1}(p-t^2) \dd t\\
        &={\omega_d\over 2} ((d-1)!)^{\alpha+1}\int_{-\infty}^{p} (p-s)^{\alpha-1\over 2} K_{d-1,f}^{\alpha+1}(s) \dd s\\
        &=\pi^{\frac{d}{2}}((d-1)!)^{\alpha+1}\frac{\Gamma\left(\frac{\alpha+1}{2}\right)}{\Gamma(\frac{d}{2})}(\frI^{\frac{\alpha+1}{2}} K_{d-1,f}^{\alpha+1})(p),
    \end{align*}
    which holds for any $p\in \inter E$. Hence, on $\inter E$ and for any $\alpha>0$ we have
    \begin{align*}
        J_{d+1,f}^{\alpha-1}=(d!)^\alpha\frac{\pi^{\frac{d+1}{2}}\Gamma\left(\frac{\alpha}{2}\right)}{\Gamma(\frac{d+1}{2})}(\frI^{\frac{\alpha}{2}} K_{d,f}^{\alpha}).
    \end{align*}
    Further note that since $K_{d,f}^{\alpha}(p)$ is non-negative we have by \eqref{eq:KInequality} for any $p\in\inter E$ that
    \begin{align*}
      (\frI^{\lceil{\alpha\over 2}\rceil}K_{d,f}^{\alpha})(p)&\leq  {(d+1)^{\max(1,\alpha)}\over (d!)^{\alpha}}{\pi^{d(d+1)\over 2}\Gamma({d+\alpha\over 2})^d\over \Gamma({d\over 2})^d}\Big(\frI^{\lceil{\alpha\over 2}\rceil}\Big[\big(\frI^{{d+\alpha\over 2}}f\big)^d(\frI^{{d\over 2}}f)\Big]\Big)(p).
    \end{align*}
    Since $f$ is non-negative we also have that $\frI^{\frac{d+\alpha}{2}}f$ and $\frI^{\frac{d}{2}}f$ are non-negative and monotone increasing for any $d\ge 2$ and $\alpha>0$. This implies
    \begin{align*}
        \Big(\frI^{\lceil{\alpha\over 2}\rceil}\Big[\big(\frI^{{d+\alpha\over 2}}f\big)^d(\frI^{{d\over 2}}f)\Big]\Big)(p)&=\int_{-\infty}^p\big((\frI^{{d+\alpha\over 2}}f)(t)\big)^d(\frI^{{d\over 2}}f)(t)(p-t)^{\lceil{\alpha\over 2}\rceil-1}\dint t\\
        &\le\big((\frI^{{d+\alpha\over 2}}f)(p)\big)^d \int_{-\infty}^p(\frI^{{d\over 2}}f)(t)(p-t)^{\lceil{\alpha\over 2}\rceil-1}\dint t\\
        &=\big((\frI^{{d+\alpha\over 2}}f)(p)\big)^d \big(\frI^{\lceil{\alpha\over 2}\rceil}\frI^{{d\over 2}}f\big)(p).
    \end{align*}
    Thus, by the semigroup property \eqref{eq:SemigroupProp} we get 
    \begin{align*}
       (\frI^{\lceil{\alpha\over 2}\rceil}K_{d,f}^{\alpha})(p)& \leq C_2\big((\frI^{{d+\alpha\over 2}}f)(p)\big)^d \big(\frI^{\lceil{\alpha\over 2}\rceil}\frI^{{d\over 2}}f\big)(p)=C_2\big((\frI^{{d+\alpha\over 2}}f)(p)\big)^d \big(\frI^{{d\over 2}+\lceil{\alpha\over 2}\rceil}f\big)(p)<\infty
    \end{align*}
    for any $p\in \inter E$ due to Lemma \ref{lm:technical} and since that $(\frI^{{d\over 2}+\lceil{\alpha\over 2}\rceil}f)(p)<\infty$ for any $p\in \inter E$. Hence, by \eqref{eq:DiffIntInverse} we obtain 
    \begin{align*}
        K_{d,f}^{\alpha} = \frac{\Gamma(\frac{d+1}{2})}{\pi^{\frac{d+1}{2}}(d!)^\alpha\Gamma(\frac{\alpha}{2})} \big(\frD^{\frac{\alpha}{2}}J_{d+1,f}^{\alpha-1}\big).
    \end{align*} 
    Now combining this with \eqref{eq:JFormula} we get for almost all $p\in\inter E$ that
    \begin{align*}
        K_{d,f}^{\alpha}(p)&=\frac{\Gamma(\frac{d+1}{2})}{\pi^{\frac{d+1}{2}}(d!)^\alpha\Gamma(\frac{\alpha}{2})} \left(\prod_{k=1}^{d+1}\pi^{k-{1\over 2}}\frac{\Gamma\left(\frac{\alpha-1+k}{2}\right)}{\Gamma(\frac{k}{2})}\right)\Big(\frD^{\frac{\alpha}{2}}\big[(\frI^{\frac{d+\alpha}{2}}f)^{d+1}\big]\Big)(p)\\
        &=\frac{\pi^{\frac{d}{2}}}{(d!)^\alpha} \left(\prod_{k=1}^{d}\pi^{k-{1\over 2}}\frac{\Gamma\left(\frac{\alpha+k}{2}\right)}{\Gamma(\frac{k}{2})}\right)\Big(\frD^{\frac{\alpha}{2}}\big[(\frI^{\frac{d+\alpha}{2}}f)^{d+1}\big]\Big)(p),
    \end{align*}
    with equality if $K_{d,f}^{\alpha}$ is continuous in $p$.
    The case $\alpha=0$ follows from direct computations, namely by using polar coordinates in $\RR^d$ we get
    \begin{align*}
        K_{d,f}^0(p)=\prod_{i=1}^{d+1}\int_{\RR^d} f(p-\|x_i\|^2)\dd x_i&=\prod_{i=1}^{d+1}{\pi^{d\over 2}\over \Gamma({d\over 2}+1)}\int_{0}^{\infty}f(p-r_i^2)r_i^{d-1}\dd r_i\\&=\pi^{d(d+1)\over 2}\Big((\frI^{d\over 2}f)(p)\Big)^{d+1},
    \end{align*}
    which coincides with \eqref{eq:KFormula} with $\alpha=0$ by recalling that $\frD^0f=f$.
\end{proof}

\begin{proof}[Proof of Corollary \ref{cor:RandomSimplexVolume}]
    First we note that 
    \[
        c(p):=\int_{\RR^d}f(p-\|x\|^2)\dd x
        =\pi^{d\over 2}(\frI^{d\over 2}f)(p)\in (0,\infty),
    \]
    and, hence, $\int_{\RR^d}g(x)\dint x=1$. Further we note that by the dominated convergence theorem $K_{d,f}^{\alpha}$ is continuous in $p\in\inter E$ if $f$ is continuous almost everywhere on $\inter E$. Indeed, let $p\in\inter E$ and $(p_n)_{n\in\NN}$ be a sequence, such that $p_n\to p$. Without loss of generality assume $p_n\in (p-\varepsilon,p+\varepsilon)$ for all $n\in \NN$. Defining functions
    \[
        g_n\colon(\RR^d)^{d+1}\to \RR_+, \quad (x_1,\ldots, x_{d+1})\mapsto\Delta_d(x_1,\ldots,x_{d+1})^\alpha \prod_{i=1}^{d+1}f(p_n-\|x_i\|^2),
    \]
    we have
     \[
        K_{d,f}^\alpha(p_n)=\int_{(\RR^d)^{d+1}} g_n(x_1,\ldots,x_{d+1})\dd x_1,\ldots,\dd x_{d+1}.
    \]
    If $f$ is monotonically increasing we have $f(p_n-\|x\|^2)\leq f(p+\varepsilon-\|x\|^2)$ for any $x\in\RR^d$ and $n\in\NN$. Then for any $(x_1,\ldots, x_{d+1})\in\RR^{d+1}$ we obtain 
    \[
        g_n(x_1,\ldots,x_{d+1})\leq \Delta_d(x_1,\ldots,x_{d+1})^\alpha \prod_{i=1}^{d+1}f(p+\varepsilon-\|x_i\|^2)=:G(x_1,\ldots,x_{d+1}),
    \]
    which is integrable majorant, since $p+\varepsilon\in\inter E$ and by \eqref{eq:KInequality} in combination with Lemma \ref{lm:technical} we have
    \begin{align*}
    \int_{(\RR^d)^{d+1}} G(x_1,\ldots,x_{d+1})\dd x_1\dots \dd x_{d+1}&=K_{d,f}^{\alpha}(p+\varepsilon)\leq C_1\big((\frI^{{d+\alpha\over 2}}f)(p+\varepsilon)\big)^{d}(\frI^{{d\over 2}}f)(p+\varepsilon)<\infty,
    \end{align*}
    for some constant $C_1\in (0,\infty)$ depending on $d$ and $\alpha$ only. The claim now follows since for almost every $x\in\RR^d$ it holds that $\lim_{n\to\infty}f(p_n-\|x\|^2)=f(p-\|x\|^2)$ due to the continuity of $f$ on $\inter E$ and by the dominated convergence theorem we get
    \[
        \lim_{n\to\infty}K^{\alpha}_{d,f}(p_n)=K^{\alpha}_{d,f}(p).
    \]
    The proof for monotonically decreasing function is the same with $p+\varepsilon$ replaced by $p-\varepsilon$.
    Then \eqref{eq:ExpectationParallelogram} and \eqref{eq:ExpectationSimplex} follow directly from
    \begin{align*}
        &\EE \big[\nabla_d(X_1,\ldots,X_{d})^{\alpha}\big]=c(p)^{-d}J_{d,f}^{\alpha}(p),\\
        &\EE \big[\Delta_d(X_1,\ldots,X_{d+1})^{\alpha}\big]=c(p)^{-d-1}K_{d,f}^{\alpha}(p),
    \end{align*}
    in combination with Theorem \ref{thm:kd_jd}.
    \end{proof}

    \begin{proof}[Proof of Theorem \ref{thm:typcell}]
    Given distinct points $x_1=(v_1,h_1),\dots,x_{d+1}=(v_{d+1},h_{d+1})$ of $\eta_{f,\gamma}$ let $w=w(x_1,\ldots,x_{d+1})$ denote the spatial coordinate of the apex of $\Pi(x_1,\dots,x_{d+1})$. Fix a Borel set $A\subseteq \cC'$ and consider
    \begin{align*}
        S_{f,\gamma,\nu}(A)&:=\EE \sum_{(v,h)\in \eta^*_{f,\gamma}} {\bf 1}(v\in[0,1]^d){\bf 1}(C((v,h),\eta^*_{f,\gamma})-v\in A)\Vol(C((v,h),\eta^*_{f,\gamma}))^\nu\\
        &={1\over (d+1)!}\EE \sum_{(x_1,\ldots,x_{d+1})\in (\eta_{f,\gamma})_{\neq}^{d+1}}{\bf 1}(w(x_1,\dots,x_{d+1})\in [0,1]^d)\\
        &\qquad\qquad\times {\bf 1}(\conv(v_1,\dots,v_{d+1})-w(x_1,\dots,x_{d+1})\in A)\\
        &\qquad\qquad\times {\bf 1}(\eta_{f,\gamma}\cap \inter\Pi^{\downarrow}(x_1,\ldots,x_{d+1})=\emptyset)\Delta_{d}(v_1,\dots,v_{d+1})^{\nu},
    \end{align*}
    where we used the construction of $\eta^*_{f,\gamma}$ described in \eqref{eq:DualProcess}.
    Note that $S_{f,\gamma,\nu}=\alpha(f,\gamma,\nu) \PP_{f,\gamma,\nu}$ if $\alpha(f,\gamma,\nu) \in(0,\infty)$. Applying the multivariate Mecke formula \cite[Corollary 3.2.3]{SW} we get
    \begin{equation}
        \begin{aligned}\label{eq:S_f}
        S_{f,\gamma,\nu}(A)&=\frac{\gamma^{d+1}}{(d+1)!}\int_{{(\RR^{d})}^{d+1}} \int_{E^{d+1}} {\bf 1}(\conv(v_1,\dots,v_{d+1})-w(x_1,\dots,x_{d+1})\in A)\\
        &\qquad\times {\bf 1}(w(x_1,\dots,x_{d+1})\in [0,1]^{d})\PP\big(\eta_{f,\gamma}\cap\inter (\Pi^{\downarrow}(x_1,\dots,x_{d+1}))=\emptyset\big)\\
        &\qquad\times \prod_{i=1}^{d+1} f(h_i) \Delta_{d}(v_1,\dots,v_{d+1})^{\nu}\dd(h_1,\dots,h_{d+1})\,\dd(v_1,\dots,v_{d+1}).      
        \end{aligned}
    \end{equation}
    We will now identify the vector $(v_1,\dots,v_{d+1},h_1,\dots,h_{d+1}) \in (\RR^{d})^{d+1} \times E^{d+1}$ with \\$(w,p,y_1,\dots,y_{d+1}) \in \RR^{d} \times \RR_+ \times (\RR^{d})^{d+1}$ as follows. Let $I_{d}$ be the set of $(d+1)$-tuples $(v_1,\ldots, v_{d+1})\in (\RR^{d})^{d+1}$ such that $v_1,\dots,v_{d+1}$ are affinely dependent. It is clear, that $I_d$ has Lebesgue measure $0$. Let $(v_1,\dots,v_{d+1},h_1,\dots,h_{d+1}) \in (\RR^{d}\setminus I_d)^{d+1} \times E^{d+1}$, $x_i=(v_i,h_i)$ and denote the apex of $\Pi(x_1,\dots,x_{d+1})$ by $(w,p) \in \RR^{d} \times \RR$. For $i=1,\dots,d+1$ we can represent $v_i=w+y_i$ with uniquely defined pairwise distinct $y_1,\dots,y_{d+1} \in \RR^{d}$. Hence we consider the transformation $\Psi$ defined as
    \begin{align*}
        \Psi\colon  &\RR^{d} \times \RR \times (\RR^{d})^{d+1} \to (\RR^{d} \times \RR)^{d+1}\\
        &(w,p,y_1,\dots,y_{d+1}) \mapsto (w+y_1,p-\|y_1\|^2,\dots,w+y_{d+1},p-\|y_{d+1}\|^2).
    \end{align*} 
    Note, that $\Psi(w,p,y_1,\dots,y_{d+1})=(v_1,h_1,\ldots,v_{d+1},h_{d+1})$ almost everywhere. Applying the transformation $\Psi$ in \eqref{eq:S_f} and using Lemma \ref{lm:jacobian} gives us 
    \begin{align*}
        S_{f,\gamma,\nu}(A)&=\frac{(2\gamma)^{d+1}}{d+1}\int_{{(\RR^{d})}^{d+1}} \int_{\RR^{d}} \int_{\RR} {\bf 1}(\conv(y_1,\dots,y_{d+1})\in A){\bf 1}(w\in [0,1]^{d})\\
        &\times \PP\big(\eta_{f,\gamma}(\inter\Pi_{(w,p)}^{\downarrow})=0\big)\Delta_{d}(y_1,\dots,y_{d+1})^{\nu+1}\prod_{i=1}^{d+1}f(p-\|y_i\|^2)\dd p\, \dd w\, \dd(y_1,\dots,y_{d+1}).
    \end{align*}
    By Lemma \ref{lm:IntensityParabola} we have for any $(w,p)\in\RR^d\times \RR$ that
    \[
        \PP(\eta_{f,\gamma}(\inter\Pi^\downarrow_{(w,p)})=0)=\exp\big(-\gamma\pi^{d\over 2}(\frI^{{d\over 2}+1}f)(p)\big),
    \]
    which implies
    \begin{align*}
        S_{f,\gamma,\nu}(A)&=\frac{(2\gamma)^{d+1}}{d+1}\int_{{(\RR^{d})}^{d+1}} \int_{\RR} {\bf 1}(\conv(y_1,\dots,y_{d+1})\in A)\exp\big(-\gamma\pi^{\frac{d}{2}}(\frI^{\frac{d}{2}+1}f)(p)\big)\\
        &\qquad\qquad\times\Delta_{d}(y_1,\dots,y_{d+1})^{\nu+1}\prod_{i=1}^{d+1}f(p-\|y_i\|^2)\dd p\, \dd y_1\dots \dd y_{d+1}.
    \end{align*}
     Note that in the case of an interval of type (ii) we have $\exp(-\gamma\pi^{\frac{d}{2}}(\frI^{\frac{d}{2}+1}f)(p))=0$ for any $p\ge b$. On the other hand if $p<a$, then $f(p-\|y\|^2)=0$ for any $y\in\RR^d$. Thus, we may restrict to $p\in E$.
     
     It remains to show that $\alpha(f,\gamma,\nu)=S_{f,\gamma,\nu}(\cC')$ is finite. By applying Fubini's theorem we have
    \begin{align*}
        \alpha(f,\gamma,\nu)=S_{f,\gamma,\nu}(\cC')&=\frac{(2\gamma)^{d+1}}{d+1}\int_{{(\RR^{d})}^{d+1}} \int_{E} \exp\big(-\gamma\pi^{\frac{d}{2}}(\frI^{\frac{d}{2}+1}f)(p)\big)\Delta_{d}(y_1,\dots,y_{d+1})^{\nu+1}\\
        &\hspace{5cm}\times\prod_{i=1}^{d+1}f(p-\|y_i\|^2)\dd p\,\dd y_1\dots\dd y_{d+1}\\
        &=\frac{(2\gamma)^{d+1}}{d+1}\int_{E} \exp\big(-\gamma\pi^{\frac{d}{2}}(\frI^{\frac{d}{2}+1}f)(p)\big)K_{d,f}^{\nu+1}(p)\dd p,
    \end{align*}
    which is finite under our assumptions. If additionally we have $(\frI^{\frac{d}{2}+\lceil\frac{\nu+1}{2}\rceil}f)(p)<\infty$ for any $p\in E$ then by \eqref{eq:KFormula} we obtain
    \begin{align*}
        \alpha(f,\gamma,\nu)&=\frac{(2\gamma)^{d+1}}\pi^{d(d+1)\over 2}{(d+1)(d!)^{\nu+1}} \prod_{k=1}^{d}\frac{\Gamma\left(\frac{\nu+1+k}{2}\right)}{\Gamma(\frac{k}{2})}\\
        &\hspace{2cm}\times\int_{E}\exp\big(-\gamma\pi^{\frac{d}{2}}(\frI^{\frac{d}{2}+1}f)(p)\big) \Big(\frD^{\frac{\nu+1}{2}}\big[(\frI^{\frac{d+\nu+1}{2}}f)^{d+1}\big]\Big)(p) \dd p.
    \end{align*}
    This finishes the proof.
\end{proof}

\begin{proof}[Proof of Theorem \ref{thm:moments}]
By the definition of $Z_{d,\gamma,\nu}(f)$ we have
    \begin{align*}
        \EE \Vol(Z_{d,\gamma,\nu}(f))^s &= \frac{1}{\alpha(f,\gamma,\nu)} \EE \sum_{(v,h)\in \eta^*_{f,\gamma}}{{\bf 1}}(v\in[0,1]^d)\Vol(C((v,h),\eta^*_{f,\gamma}))^{\nu+s}= \frac{\alpha(f,\gamma,\nu+s)}{\alpha(f,\gamma,\nu)},
    \end{align*}
    which together with \eqref{eq:BoundNormalizationConst} finishes the proof.
\end{proof}

\begin{proof}[Proof of Proposition \ref{prop:FiniteNormalizationConst}]
    First we note that for $\nu\ge-1$ according to Theorem \ref{thm:kd_jd}, in particular \eqref{eq:KInequality}, we have
    \begin{align*}
     \int_{E}&e^{-\gamma\pi^{\frac{d}{2}}(\frI^{\frac{d}{2}+1}f)(p)} K^{\nu+1}_{d,f}(p) \dd p\leq C(d,\nu)\int_{E}e^{-\gamma\pi^{\frac{d}{2}}(\frI^{\frac{d}{2}+1}f)(p)} \Big((\frI^{{d+\nu+1\over 2}}f)(p)\Big)^d(\frI^{{d\over 2}}f)(p) \dd p,
    \end{align*}
    for some explicit constant $C(d,\nu)\in(0,\infty)$ depending on $\nu$ and $d$ only. Thus, it is sufficient to show
    \begin{equation}\label{eq:19.03.24_1}
        Q_{\nu}(f):=\int_{E}e^{-\gamma\pi^{\frac{d}{2}}(\frI^{\frac{d}{2}+1}f)(p)} \Big((\frI^{{d+\nu+1\over 2}}f)(p)\Big)^d(\frI^{{d\over 2}}f)(p) \dd p<\infty.
    \end{equation}

    \textit{Case (i):} For $\nu=1$ equation \eqref{eq:19.03.24_1} becomes 
    \begin{align*}
        Q_{\nu}(f)= \int_{E}e^{-\gamma\pi^{\frac{d}{2}}(\frI^{\frac{d}{2}+1}f)(p)} \Big((\frI^{{d\over 2}+1}f)(p)\Big)^d(\frI^{{d\over 2}}f)(p) \dd p.
    \end{align*}
    Then noting that under \hyperref[item:F1]{(F1)} the function $\frI^{{d\over 2}+1}f$ is differentiable almost everywhere on $E$ with $(\frI^{{d\over 2}+1}f)'=\frI^{{d\over 2}}f$ and applying the change of variables $z=\gamma\pi^{\frac{d}{2}}(\frI^{\frac{d}{2}+1}f)(p)$ we obtain
    \begin{align*}
        Q_{\nu}(f)=(\gamma\pi^{d\over 2})^{-d-1}\int_{0}^\infty e^{-z}z^d \dd z = (\gamma\pi^{d\over 2})^{-d-1}\Gamma(d+1) < \infty.
    \end{align*}

    \textit{Case (ii):}
    Let $\int_{\RR} f(x) \dd x =: c < \infty$. For $-1\leq \nu <1$ and $d>2$ and for $\nu\in(-1,1)$ and $d=2$ by the Hölder inequality with $q_1={d\over 1-\nu}$, $q_2=\frac{d}{d+\nu-1}\in (1,\infty)$ for any $p\in\RR$ we get
    \begin{align*}
        (\frI^{d+\nu+1\over 2}f)(p)&={1\over \Gamma({d+\nu+1\over 2})}\int_{a}^{p}f(t)(p-t)^\frac{d+\nu-1}{2} \dd p \\
        &\le {1\over \Gamma({d+\nu+1\over 2})}\Big(\int_{a}^{p}f(t)\dd t\Big)^{1-\nu\over d}\Big(\int_{a}^p f(t) (p-t)^{\frac{d}{2}} \dd t\Big)^{d+\nu-1\over d} \\
        &= {c^{1-\nu\over d}\over \Gamma({d+\nu+1\over 2})} \Big(\Gamma\Big(\frac{d}{2}+1\Big)(\frI^{\frac{d}{2}+1} f)(p)\Big)^{d+\nu-1\over d}.
    \end{align*}
    Hence,
    \begin{align*}
        Q_{\nu}(f)&\leq c^{1-\nu}{\Gamma(\frac{d}{2}+1)^{d+\nu-1}\over \Gamma({d+\nu+1\over 2})^d}\int_{\RR}e^{-\gamma\pi^{\frac{d}{2}}(\frI^{\frac{d}{2}+1}f)(p)} \Big((\frI^{{d\over 2}+1}f)(p)\Big)^{d+\nu-1}(\frI^{{d\over 2}}f)(p) \dd p\\
    &= c^{1-\nu}{\Gamma(\frac{d}{2}+1)^{d+\nu-1}\over \Gamma({d+\nu+1\over 2})^d}(\gamma\pi^{d\over 2})^{-d-\nu}\Gamma(d+\nu)< \infty,
    \end{align*}
    where the last step follows from the same argument as in case (i). For $\nu=-1$ and $d=2$ we have 
    $$
        (\frI^{d+\nu+1\over 2}f)(p)=\int_{a}^p f(t)\dd t\leq c,
    $$
    and 
    \begin{align*}
        Q_{\nu}(f)&\leq c^d\int_{\RR}e^{-\gamma\pi^{\frac{d}{2}}(\frI^{\frac{d}{2}+1}f)(p)}(\frI^{{d\over 2}}f)(p) \dd p= c^{d}(\gamma\pi^{d\over 2})^{-1}< \infty,
    \end{align*}
    where in the last step we again used the same argument as in case (i).

    \textit{Case (iii):} Let $f\colon(0,\infty) \to \RR_+$ be a locally integrable function on $[0,\infty)$, which is regularly varying at $+\infty$ with index $\beta>-1$. Note that by \cite[Lemma 2.1]{MARTINEZ1992111} we have $(\frI^{\alpha}f)(p)<\infty$ for any $p>0$ and $\alpha>0$. At the same time we have $(\frI^{\alpha}f)(p)=0$ for any $p\leq 0$ and $\alpha>0$. Set $\delta={\beta+1\over 2}$, so that $\beta>-1+\delta$ and let $\nu \ge -1$. As a corollary of the Karamata representation (see \cite[p.17]{resnick_ERVPP}) there exists an $A\ge 1$ such that for all $x>A$ we have 
    \[
        x^{\beta-\delta}\le f(x) \le x^{\beta+\delta}.
    \]
    Then for any $\alpha\ge 1$ and $p> A$ we have
    \begin{align*}
        (\frI^{\alpha}f)(p)&=\frac{1}{\Gamma(\alpha)} \int_{0}^{p} f(x)(p-x)^{\alpha-1} \dd x \\
        &=\frac{1}{\Gamma(\alpha)} \left(\int_{0}^{A} f(x)(p-x)^{\alpha-1} \dd x+\int_{A}^{p} f(x)(p-x)^{\alpha-1} \dd x\right)\\
        & \le \frac{1}{\Gamma(\alpha)} \left(p^{\alpha-1}\int_{0}^{A}f(x)\dd x+\int_{A}^{p}x^{\beta+\delta}(p-x)^{\alpha-1}\dd x\right) \\
        & \le \frac{1}{\Gamma(\alpha)} \left(p^{\alpha-1}\int_{0}^{A}f(x)\dd x+p^{\alpha+\beta+\delta}B(\beta+\delta+1,\alpha)\right)\\
        &\le C_1 p^{\alpha+\beta+\delta}, 
    \end{align*} 
    for some constant $C_1$ independent of $p$, where $B(z_1,z_2)$ for $z_1,z_2>0$ denotes the beta function. Analogously we have for $p>A$ that $(\frI^{\alpha}f)(p) \ge C_2 p^{\alpha+\beta-\delta}$ for some constant $C_2$ independent of $p$. Hence, for $p>A$ we have 
    \begin{equation}\label{eq:19.03.24_2}
        (\frI^{\frac{d}{2}+1}f)(p) \ge c_1 p^{\frac{d}{2}+1+\beta-\delta},\qquad (\frI^{\frac{d+\nu+1}{2}}f)(p) \le c_2 p^{\frac{d+\nu+1}{2}+\beta+\delta},\qquad (\frI^{\frac{d}{2}}f)(p) \le c_3 p^{\frac{d}{2}+\beta+\delta},
    \end{equation} 
    for some constants $c_1,c_2,c_3>0$, independent of $p$. Since $\frI^{\frac{d+\nu+1}{2}}f$ and $\frI^{\frac{d}{2}}f$ are non-negative and monotonically increasing in $p$ we have
    \begin{align}
        \int_{0}^{A} e^{-\gamma\pi^{\frac{d}{2}}(\frI^{\frac{d}{2}+1}f)(p)} &\Big((\frI^{{d+\nu+1\over 2}}f)(p)\Big)^d(\frI^{{d\over 2}}f)(p) \dd p\le A\Big((\frI^{{d+\nu+1\over 2}}f)(A)\Big)^d(\frI^{{d\over 2}}f)(A)  < \infty.\label{eq:19.03.24_3}
    \end{align}
    On the other hand by \eqref{eq:19.03.24_2} we get
    \begin{align*}
        \int_{A}^\infty &e^{-\gamma\pi^{\frac{d}{2}}(\frI^{\frac{d}{2}+1}f)(p)} \Big((\frI^{{d+\nu+1\over 2}}f)(p)\Big)^d(\frI^{{d\over 2}}f)(p) \dd p \\
        &\hspace{2cm}\le c_2^dc_3\int_{0}^{\infty} e^{-\gamma\pi^{\frac{d}{2}}c_1 p^{\frac{d}{2}+1+\beta-\delta}}p^{\frac{d(d+\nu+2)}{2}+(d+1)(\beta+\delta)}\dd p\\
        &\hspace{2cm}=c\Gamma\Big(\frac{d(d+\nu)+2(d+1)(\beta+1+\delta)}{d+2(\beta+1-\delta)}\Big)<\infty,
    \end{align*}
    where $c\in(0,\infty)$ is some constant depending on $A$, $d$, $\nu$, $\beta$ and $\delta$. This together with \eqref{eq:19.03.24_3} and \eqref{eq:19.03.24_1}, where in this case $E=(0,\infty)$, finishes the proof.

    \textit{Case (iv):} Let $f\colon(-\infty,0)\to \RR_+$ be a locally integrable function on $(-\infty,0)$, which is regularly varying at $0$ with index $\beta>{d\over 2}+1$ and for some $\alpha> {d\over 2}+1$ it holds that $(\frI^{\alpha}f)(p)<\infty$ for all $p<0$. Then by Lemma \ref{lm:technical} we have $(\frI^{\mu}f)(p)<\infty$ for all $p<0$, $1\leq\mu\leq \alpha$ and it implies that $f$ satisfies \hyperref[item:F1]{(F1)}.  By \eqref{eq:frac_f_g} for any $q>0$ and $1\leq \mu\leq \alpha$ we have
    \[
        (\frI^{\mu}f)(-q)=q^{\mu-1}(\frI^\mu f_\mu)\left(q^{-1}\right)<\infty,   
    \]
    where $f_{\mu}(u)=u^{-\mu-1}f(-1/u)$ is a non-negative function, defined on the interval $(0,\infty)$. Further note that for any $p>0$ we also have
    \begin{align*}
        \int_{0}^pf_{\mu}(u)\dd u&=\int_{-\infty}^{-1/p}\big((-1/p-t)+1/p\big)^{\mu-1}f(t)\dd t\\
        &\leq 2^{\max(0,\mu-2)}\Big(\int_{-\infty}^{-1/p}(-1/p-t)^{\mu-1}f(t)\dd t+p^{-\mu+1}\int_{-\infty}^{-1/p}f(t)\dd t\Big)\\
        &=2^{\max(0,\mu-2)}\big((\frI^{\mu}f)(-1/p)+p^{-\mu+1}(\frI^1 f)(-1/p)\big)<\infty,
    \end{align*}
    where in the third step we used the inequality $(x+y)^{\mu-1}\leq 2^{\max(0,\mu-2)}(x^{\mu-1}+y^{\mu-1})$ which holds for all $\mu\ge 1$ and $x,y\ge 0$. Thus, for any $1\leq \mu< \min(\alpha,\beta)$ we conclude that $f_\mu\in L_{\rm loc}^{1,+}([0,\infty))$, which is regularly varying at $+\infty$ with index $\beta-\mu-1>-1$. Let $0<\delta\leq  {\beta-\mu\over 2}$. By using the same arguments as in the proof of case (iii) we get, that there exists $A:=A(\mu)>0$, such that for all $x>A$ we have 
    \[
        x^{\beta-\mu-1-\delta}\le f_\mu(x)\le x^{\beta-\mu-1+\delta}
    \]
    and, hence, there are constants $C_1,C_2$ independent of $x$ such that
    \begin{equation}\label{eq:30.05.24_10}
        C_1 x^{\beta-\mu-\delta}\leq x^{-\mu+1}(\frI^{\mu}f_{\mu})(x)=(\frI^{\mu}f)(-1/x)\leq C_2 x^{\beta-\mu+\delta}
    \end{equation}
    for any $x>A$. Further, since $1\leq {d\over 2}<{d\over 2}+1< \min(\alpha,\beta)$ and $1\leq {d+\nu+1\over 2}<\min(\alpha,\beta)$ for $-1\leq \nu<2\min(\alpha,\beta)-d-1$ 
    we get for any $p\leq B:=-1/A'<0$, where $A'=\min\{A(d/2), A(d/2+1),A((d+\nu+1)/2)\}$, by \eqref{eq:30.05.24_10} that
    \begin{equation}\label{eq:12.04.24_1}
        \begin{aligned}
            &(\frI^{\frac{d}{2}+1}f)(p) \ge c_1(-p)^{{d\over 2}+1-\beta+\delta},\,  (\frI^{\frac{d+\nu+1}{2}}f)(p) \le c_2 (-p)^{{d+\nu+1\over 2}-\beta-\delta}, \, (\frI^{\frac{d}{2}}f)(p) \le c_3 (-p)^{{d\over 2}-\beta-\delta},
        \end{aligned}
    \end{equation}
    for some constants $c_1$, $c_2$, $c_3>0$, independent of $p$, and $\delta=(\beta-d/2-1)/2$. Further consider
    \begin{align}
        Q_{\nu}(f)&=\int_{-\infty}^{B}e^{-\gamma\pi^{\frac{d}{2}}(\frI^{\frac{d}{2}+1}f)(p)}\Big((\frI^{\frac{d+\nu+1}{2}}f)(p)\Big)^d (\frI^{\frac{d}{2}}f)(p)\dd p \notag\\
        &\qquad\qquad + \int_{B}^0 e^{-\gamma\pi^{\frac{d}{2}}(\frI^{\frac{d}{2}+1}f)(p)}\Big((\frI^{\frac{d+\nu+1}{2}}f)(p)\Big)^d (\frI^{\frac{d}{2}}f)(p)\dd p.\label{eq:20.03.24_2}
    \end{align}
    Note that since $f$ is positive and ${d+\nu+1\over 2}\ge 1$ we have, that $\frI^{\frac{d+\nu+1}{2}}f$ is monotonically increasing in $p$ and, hence, we obtain
    \begin{align}
        &\int_{-\infty}^B e^{-\gamma\pi^{\frac{d}{2}}(\frI^{\frac{d}{2}+1}f)(p)}\Big((\frI^{\frac{d+\nu+1}{2}}f)(p)\Big)^d (\frI^{\frac{d}{2}}f)(p) \dd p\notag\\
        &\hspace{4cm}\le \Big((\frI^{\frac{d+\nu+1}{2}}f)(B)\Big)^d\int_{-\infty}^0 e^{-\gamma\pi^{\frac{d}{2}}(\frI^{\frac{d}{2}+1}f)(p)} (\frI^{\frac{d}{2}}f)(p) \dd p\notag\\
        &\hspace{4cm}= \Big((\frI^{\frac{d+\nu+1}{2}}f)(B)\Big)^d\gamma^{-1}\pi^{-{d\over 2}} \int_{0}^\infty e^{-z} \dd z < \infty,\label{eq:20.03.24_3}
    \end{align}
    where we again used that the function $\frI^{\frac{d}{2}+1}f$ is differentiable almost everywhere with derivative $\frI^{\frac{d}{2}}f$ and applied the change of variables $z=\gamma\pi^{\frac{d}{2}}(\frI^{\frac{d}{2}+1}f)(p)$. Finally, using \eqref{eq:12.04.24_1} we obtain
    \begin{align*}
        \int_{B}^0 e^{-\gamma\pi^{\frac{d}{2}}(\frI^{\frac{d}{2}+1}f)(p)}&\Big((\frI^{\frac{d+\nu+1}{2}}f)(p)\Big)^d (\frI^{\frac{d}{2}}f)(p)\dd p \\
        &\quad\leq c_2^d c_3\int_{B}^{0} e^{-\gamma\pi^{d\over 2}c_1(-p)^{{d\over 2}+1-\beta+\delta}}((-p)^{\frac{d+\nu+1}{2}-\beta-\delta})^d(-p)^{\frac{d}{2}-\beta-\delta}\dd p\\
        &\quad\leq c_2^d c_3\int_{0}^{\infty} e^{-\gamma\pi^{d\over 2}c_1s^{\beta-{d\over 2}-1-\delta}}s^{(d+1)(\beta+\delta)-{d(d+\nu+2)\over 2}-2}\dd s\\
        &\quad=c \Gamma\left({2(d+1)(\beta+\delta)-d(d+\nu+2)-2\over 2(\beta-\delta)-d-2}\right)<\infty, 
    \end{align*}
    where $c\in(0,\infty)$ is a constant depending on $d$, $\nu$, $\beta$, $\delta$ and $B$. Together with \eqref{eq:20.03.24_2} and \eqref{eq:20.03.24_3}
    this finishes the proof.
    \end{proof}

    \begin{proof}[Proof of Corollary \ref{cor:decomposition}] 
    First note, that by Theorem \ref{thm:typcell} we have
    \begin{align*}
        \PP_{f,\gamma,\nu}(\cdot)&=\frac{(2\gamma)^{d+1}}{(d+1)\alpha(f,\gamma,\nu)} \int_{(\RR^{d})^{d+1}}\dd y_1\dots \dd y_{d+1} \, \int_{E}\dd p\,{\bf 1}(\conv(y_1,\dots,y_{d+1})\in \cdot) \\
        &\qquad\times e^{-\gamma\pi^\frac{d}{2} (\frI^{\frac{d}{2}+1}f)(p)} \Delta_{d}(y_1,\dots,y_{d+1})^{\nu+1}\prod_{i=1}^{d+1}f(p-\|y_i\|^2).
    \end{align*}  
    Applying the change of variables $y_i=\phi(p)x_i$, where $1\leq i\leq d+1$, and using \eqref{decomp} we get
    \begin{align*}
        \PP_{f,\gamma,\nu}(\cdot)&=\frac{(2\gamma)^{d+1}}{(d+1)\alpha(f,\gamma,\nu)} \int_{(\RR^{d})^{d+1}}\dd x_1\dots \dd x_{d+1} \, \int_{E}\dd p\,{\bf 1}(\conv(\phi(p)x_1,\dots,\phi(p)x_{d+1})\in \cdot) \\
        &\qquad\times \phi(p)^{d(d+\nu+2)}e^{-\gamma\pi^\frac{d}{2} (\frI^{\frac{d}{2}+1}f)(p)} \Delta_{d}(x_1,\dots,x_{d+1})^{\nu+1}\prod_{i=1}^{d+1}f(p-\phi(p)^2\|x_i\|^2)\\
        &=\frac{(2\gamma)^{d+1}}{(d+1)\alpha(f,\gamma,\nu)} \int_{(\RR^{d})^{d+1}}\dd x_1\dots \dd x_{d+1} \, \int_{E}\dd p\,{\bf 1}(\conv(\phi(p)x_1,\dots,\phi(p)x_{d+1})\in \cdot) \\
        &\qquad\times \phi(p)^{d(d+\nu+2)}f(p)^{d+1}e^{-\gamma\pi^\frac{d}{2} (\frI^{\frac{d}{2}+1}f)(p)} \Delta_{d}(x_1,\dots,x_{d+1})^{\nu+1}\prod_{i=1}^{d+1}\psi(\|x_i\|^2),
    \end{align*}  
    which finishes the proof.
    \end{proof}

    \begin{proof}[Proof of Proposition \ref{lm:decomp}]    
    The proof of this lemma follows similar arguments given by Miles (see \cite[pp. 376-377]{Miles_IRS}). We start by noting that \eqref{decomp} with $p=l$ for some $l \in \supp(f)\cap \supp(\phi)=\supp(f)\neq \emptyset$ leads to  
    \[
        \psi(s^2)=f(l)^{-1}f(l-\phi(l)^2s^2),
    \]
    which holds for any $s\ge 0$. Substituting this into \eqref{decomp} yields
    \begin{equation}\label{eq:U3}
        f(p-\phi(p)^2s^2)f(l)=f(p)f(l-\phi(l)^2s^2), \qquad\qquad p\in \inter E,\, s\ge 0,\, l\in \supp(f).
    \end{equation}
    Alternatively, defining $x:=p$, $y:=l-\phi(l)^2s^2$, $V_l(x):=\phi(x)^2/\phi(l)^2$ and 
    $U_l(x):= f(x)f(l)^{-1}$ we get
    \begin{align}\label{eq:U}
        U_l(x-V_l(x)(l-y))=U_l(x)U_l(y),\qquad x\in\inter E, y\leq l,
    \end{align}
    which holds for any $l\in\supp(f)$.

    Further we note that for any $p\ge l$, $p\in\inter E$, by substituting $s=\phi(p)^{-1}\sqrt{p-l}\in [0,\infty)$ (since $\phi(p)>0$ for any $p\in\inter E$) into \eqref{eq:U3} we have
    \[
        f(p)f\Big(l-{\phi(l)^2\over \phi(p)^2}(p-l)\Big)=f(l)^2>0.
    \]
    This in particular implies that $f(p)>0$ for any $p\ge l$, $p\in\inter E$ and, hence, $\supp (f) = (q_1,q_2)$ for some $-\infty\leq q_1<q_2\leq \infty$. Note that in the case of an interval of type (i) and (iii) we have $q_2=\infty$ and in the case of an interval of type (ii) we have $q_2=b\in\RR$. In the case of an interval of type (i) we also have $q_1\ge a\in \RR$ and, thus, without loss of generality we may assume $q_1=a$, leading to $\supp(f)=(a,\infty)=\inter E$ in this case. Finally in the case of an interval of type (iii) we may assume without loss of generality $q_1=-\infty$, since otherwise we would be in the situation suitable for case (i). Thus, $\supp(f)=\RR$ in the case of the interval of type (iii).

    Our next aim is to determine the type of the function $\phi$. Assume first that $q_1\in\RR$, which is possible for intervals of type (i) or (ii). Let $l_1,l_2\in\supp(f)$ be arbitrary, then from \eqref{eq:U3} we have
    \[
        f(l_1-\phi(l_1)^2s^2)f(l_2)=f(l_2-\phi(l_2)^2s^2)f(l_1),
    \]
    which holds for any $s\ge 0$. In particular the latter means that $l_1-\phi(l_1)^2s^2\in \supp (f) = (q_1,q_2)$ if and only if $l_2-\phi(l_2)^2s^2\in\supp (f) = (q_1,q_2)$, which implies
    \[
        s^2\leq {l_1-q_1\over \phi(l_1)^2}={l_2-q_1\over \phi(l_2)^2}
    \]
    for any $l_1,l_2\in\supp(f)$. Fixing $l_1=\min\{(q_1+q_2)/2,q_1+1\}\in(q_1,q_2)$ and setting $l_2=x$ we get
    \[
        \phi(x)^2={\phi(l_1)^2\over l_1-q_1}(x-q_1),\qquad x\in (q_1,q_2).
    \]
    In particular
    \begin{equation}\label{eq:V1}
        V_l(x)={\phi(x)^2\over \phi(l)^2}={x-q_1\over l-q_1}=1+c(l)(x-l),\qquad x\in (q_1,q_2),\, c(l)=(l-q_1)^{-1}>0.
    \end{equation}  
    
    Now consider the case when $q_1=-\infty$, which may correspond to the intervals of type (ii) and (iii). Differentiating the equation \eqref{eq:U} with respect to $x$ and with respect to $y$ assuming that $x,y\in\inter E$, $y\leq l$ and $x-V_l(x)(l-y)\in\inter E$ gives us
    \begin{align}
        &\frac{\partial}{\partial x} U_l(x-V_l(x)(l-y)) = (1-V_l'(x)(l-y))U_l'(x-V_l(x)(l-y))=U_l'(x)U_l(y), \label{eq:D1}\\
        &\frac{\partial}{\partial y} U_l(x-V_l(x)(l-y)) = V_l(x)U_l'(x-V_l(x)(l-y))=U_l(x)U_l'(y). \label{eq:D2}
    \end{align}   
    Further note, that in the case of an admissible function $f$ we have $U_l'(x-V_l(x)(l-x))\neq 0$ for almost all $x\leq l$. Indeed, assume $U_l'(x-V_l(x)(l-x))=0$ on some interval $I:=[x_1,x_2]\subset(-\infty,l]$. Then $U_l(x-V_l(x)(l-x))=C$ for some $C>0$, implying 
    \[
        f(x-\phi(l)^{-2}\phi(x)^2(l-x))= C f(l)
    \]
    for all $x\in I$. On the other hand by \eqref{eq:U} we have
    \[
        Cf(l)=f(x-\phi(l)^{-2}\phi(x)^2(l-x)) =f(x)^2f(l)^{-1},\qquad x\in I.
    \] 
    Hence $f(x)=\sqrt{C}f(l)$ for all $x\in I$. Since by \eqref{eq:U3} for any $s\ge 0$ and $x\in I$ it holds that
    \[
        f(x-\phi(x)^2s^2)f(l)= f(x)f(l-\phi(l)^2s^2),
    \]
    we have $f(x-\phi(x)^2s^2)= \sqrt{C}f(l-\phi(l)^2s^2)$. The latter means that for any fixed $s\in\RR$ we have that the function $f$ is a constant (i.e. $\sqrt{C}f(l-\phi(l)^2s^2)$) on the set $I(s)=\{x-\phi(x)^2s^2\colon x\in I\}$. Now note that for any $s\ge 0$ there is $s'\ge 0$, $s'\neq s$, such that $I(s)\cap I(s')\neq \emptyset$.
    
    Thus, we conclude that $f(x)=\sqrt{C}f(l)$ for all $x\leq x_2$. It is easy to ensure that $(\frI^{{d\over 2}+1}f)(x)=\infty$ for any $x\in\RR$ and the function is not admissible. Now setting $x=y\leq l$ in \eqref{eq:D1} and \eqref{eq:D2}, eliminating $U_l'(x-V_l(x)(l-x))\neq 0$ and solving the resolving differential equation
    \begin{align*}
        V_l(x)=1-V_l'(x)l+V_l'(x)x,\qquad x\leq l,
    \end{align*}
    we obtain 
    \begin{equation}\label{eq:V2}
        V_l(x)=1+c(x-l),\qquad x\leq l,
    \end{equation}
    for some constant $c\in \RR$. Note that for $c>0$ we additionally have $x\ge l-c^{-1}$ since $V_l(x)\ge 0$ by definition. On the other hand since $(-\infty,q_2)=\supp(f)\subset\supp(V_l)=\supp(\phi)$ we have that $c>0$ is not possible. Let us point out that for any $l\in(-\infty,q_2)$ we have that $c\leq 0$ is a constant, but its value may depend on $l$. We stress it by writing $c=c(l)$ to specify this dependence if needed.

    Further, combining \eqref{eq:V1} and \eqref{eq:V2} with \eqref{eq:U} we get
    \begin{align}\label{eq:U1}
        U_l(x+y-l+c\cdot (l-x)(l-y))=U_l(x)U_l(y),\qquad q_1< x,y\leq l.    
    \end{align}
    We distinguish between the cases $c=0$ and $c \neq 0$. 
    
    For $c=0$ we have $q_1=-\infty$ and by \eqref{eq:U1} substituting $\tilde x=l-x\ge 0$, $\tilde y=l-y\ge 0$ and $\tilde U_l(x)=U_l(l-x)$ we arrive at the well-known Cauchy-Hamel equation
    \begin{align*}
        \tilde U_l(\tilde x+\tilde y)=\tilde U_l(\tilde x)\tilde U_l(\tilde y),\qquad \tilde x,\tilde y\ge 0.
    \end{align*}
    whose only well-behaved non-zero solutions are $\tilde U_l(\tilde{x})=f(l-\tilde x)/f(l)=e^{-\lambda \tilde{x}}$, $\tilde x\ge 0$ for $\lambda \in \RR$ and, hence, $f(x)=f(l)e^{-\lambda l}e^{\lambda x}=c_1e^{\lambda x}$, $x\leq l$. If $\supp (f) \cap (0,\infty) \neq \emptyset$ we assume $l>0$. Let $\widetilde{l}:=lt$ for some $t\ge 1$ be such that $f(\widetilde{l})>0$. By the same arguments as above we get
    \[
        f(x)=f(\widetilde{l})e^{-\lambda \widetilde{l}}e^{\lambda x}=f(lt)e^{-\lambda lt}e^{\lambda x},\qquad t\ge 1, x\le lt.
    \]
    Since $l\le lt$ we have $f(l)=f(lt)e^{-\lambda lt } e^{\lambda l}$ for all $t\ge 1$ and substituting $x=lt$ yields
    \[
        f(x)=f(l)e^{-\lambda l} e^{\lambda x }=c_1e^{\lambda x}
    \]
    for all $x\ge l$ and, hence, for all $x\in(-\infty,q_2)$. If on the other hand $\supp(f) \subset (-\infty,0)$, we arrive at $f(x)=c_1e^{\lambda x}$, $x\in (-\infty,q_2)$, by the same arguments as above with $0\le t \le 1$. Now note that for $\lambda \le 0$ we have by direct computations that $(\frI^{\frac{d}{2}+1}f)(x)=\infty$ for all $x\in \RR$ and, thus, $f$ is not admissible. If $\lambda>0$ and $q_2=\infty$ (meaning that $E=\RR$ is an interval of type (iii)) it was shown in Example \ref{ex:BetaModels} that $f$ is admissible. If $\lambda>0$ and $q_2<\infty$ (meaning that $E=(-\infty,q_2)$ is an interval of type (ii)) $f$ is not admissible since $(\frI^{\frac{d}{2}+1}f)(q_2)<\infty$.
    
    Further consider the case $c(l) > 0$, which is only possible if $q_1\in\RR$. Defining $W_l(x):=U_l(x-c(l)^{-1}+l)$, $\theta :=x+c(l)^{-1}-l=x-q_1$ and $\xi :=y+c(l)^{-1}-l=y-q_1$, where we recall that $c(l)=(l-q_1)^{-1}$ in this case {(see \eqref{eq:V1})}, equation \eqref{eq:U1} becomes the Cauchy-Hamel equation
    \begin{align*}
        W_l(c(l)\theta\xi)&=U_l(x-(1-c(l)l+c(l)x)(l-y))= U_l(x)U_l(y)=W_l(\theta)W_l(\xi), 
    \end{align*}
    where $0<\xi, \theta \leq c(l)^{-1}.$ The unique non-trivial solution to this equation is $W_l(\theta)=(c(l)\theta)^\beta$, $0<\theta \leq c^{-1}$, for some $\beta \in \RR$ (see \cite[Section 1.1]{bingham_regular_variation} for more details) and therefore we get $f(x)=f(l)(l-q_1)^{-\beta}(x-q_1)^\beta$, $q_1<x\leq l$, where $f(l) \in \RR_+\bsl\{0\}$ and $\beta \in \RR$. As in the previous case consider $\widetilde l=lt$ with $t\ge 1$ if $l>0$ and $0\leq t\leq 1$ if $l<0$. If $f(\widetilde l)\neq 0$ by the same arguments as above we obtain
    \[
        f(x)=f(lt)(lt-q_1)^{-\beta}(x-q_1)^\beta, \qquad q_1<x\leq lt,
    \]
    and since $l\leq lt$ by substituting first $x=l$ and then writing $y:=tl$ we get
    \[
        f(y)=f(l)(l-q_1)^{-\beta}(y-q_1)^\beta=c_1(y+c_2)^\beta,\qquad -c_2=q_1<y<q_2.
    \]
    If $q_2=\infty$, which corresponds to the interval of type (i) we have that $f$ if admissible, i.e. locally integrable, if and only if $\beta>-1$. If $q_2<\infty$, implying that $E=(-\infty,q_2)$ is an interval of type (ii), we conclude that the function $f$ is not admissible since $(\frI^{{d\over 2}+1}f)(q_2)=c_3(q_2-q_1)^{\beta+d/2+1}<\infty$.
    
    Finally, for $c(l)<0$ we have that $q_1=-\infty$. Defining $W_l$, $\theta$ and $\xi$ as in the case $c(l)>0$ the equation \eqref{eq:U1} becomes the Cauchy-Hamel equation
    \begin{align*}
        W_l(c(l)\theta\xi)&=W_l(\theta)W_l(\xi), \qquad \xi, \theta \leq  c(l)^{-1}.
    \end{align*}
    The unique non-trivial solution to this equation is $W_l(\theta)=(c(l)\theta)^{-\beta}$, $\theta \leq c(l)^{-1}$. Hence, in this case $f(x)=f(l)(c(l)x+1-c(l)l)^{-\beta}$, $x\le l$, where $f(l) \in \RR_+\bsl\{0\}$, $c(l)<0$ and $\beta \in \RR$. Our next aim is to show that in the case of a finite value $q_2\in\RR$ we have $c(l)=-(q_2-l)^{-1}$ and if $q_2=\infty$ we get $c(l)\ge 0$ for any $l\in\RR$, which is a contradiction to $c(l)<0$. Let $l_1,l_2\in (-\infty, q_2)$, $l_1\le l_2$. Then $\phi(x)^2=\phi(l_1)^2(1+c(l_1)(x-l_1))=\phi(l_2)^2(1+c(l_2)(x-l_2))$ for all $x\le l_1\le l_2$ and hence
    \[
        0 = \phi(l_1)^2-\phi(l_2)^2+\phi(l_2)^2c(l_2)l_2-\phi(l_1)^2c(l_1)l_1+x(\phi(l_1)^2c(l_1)-\phi(l_2)^2c(l_2))
    \]
    for all $x\le l_1\le l_2$, which yields 
    \begin{align}\label{eq:zero1}
        \phi(l_2)^2c(l_2)=\phi(l_1)^2c(l_1)
    \end{align}
    and 
    \begin{align}\label{eq:zero2}
        \phi(l_1)^2-\phi(l_2)^2+\phi(l_2)^2c(l_2)l_2-\phi(l_1)^2c(l_1)l_1 = 0.    
    \end{align}
    Combining \eqref{eq:zero1} and \eqref{eq:zero2} we get
    \begin{align}\label{eq:cprime}
        c(l_1)=\frac{\phi(l_2)^2-\phi(l_1)^2}{\phi(l_1)^2(l_2-l_1)}=\frac{\phi(l_2)^2}{\phi(l_1)^2(l_2-l_1)}-\frac{1}{l_2-l_1}=\frac{\phi(l_2)^2}{l_2}\frac{1}{\phi(l_1)^2(1-\frac{l_1}{l_2})}-\frac{1}{l_2-l_1}    
    \end{align}
    for all $l_2\ge l_1 > -\infty$.
    Let first $q_2=\infty$, then we have $\supp (f) = \RR$. Taking $l:=l_1$ and letting $l_2\to q_2=\infty$ we get
    \begin{align*}
        c(l)=\lim_{l_2\to \infty} \frac{\phi(l_2)^2}{l_2}\frac{1}{\phi(l_1)^2(1-\frac{l_1}{l_2})}-\frac{1}{l_2-l_1} = \begin{cases}
            0, \qquad &\text{ if } \frac{\phi(l_2)^2}{l_2} \to 0,\\
            \frac{c}{\phi(l_1)^2}>0, \qquad &\text{ if } \frac{\phi(l_2)^2}{l_2} \to c\in(0,\infty),\\
            \infty, \qquad &\text{ if } \frac{\phi(l_2)^2}{l_2} \to \infty,
        \end{cases}
    \end{align*}
    for all $l\in \supp(f)$, which is a contradiction to $c(l)<0$ for all $l\in \supp(f)$. Hence, $q_2<\infty$. Note that for $f$ to be admissible we assume $\lim_{p\to q_2} f(p)=\infty$ and $f(p)<\infty$ for all $p\in (-\infty, q_2)$. Taking into account that from \eqref{eq:U3} we have
    \[
        f(p-\phi(p)^2s^2)f(l)=f(p)f(l-\phi(l)^2s^2), \qquad\qquad p\in \inter E,\, s\ge 0,\, l\in \supp(f)
    \]
    and taking the limit as $p\to q_2$ we get $f(p-\phi(p)^2s^2)\to \infty$ and, hence, $\phi(p)\to 0$ as $p\to q_2$. Then setting $l=l_1$ and letting $l_2\to q_2$ equation \eqref{eq:cprime} becomes
    \[
        c(l)=-\frac{1}{q_2-l}.
    \]
    By the same arguments as in the case $c(l)>0$ and writing $c_1:=f(l)(q_2-l)^{\beta}$ and $c_2:=q_2$ we get 
    \[
        f(x)=f(l)(q_2-l)^{\beta}(-x+q_2)^{-\beta}=c_1(-x+c_2)^{-\beta}
    \]
    for all $x<c_2$. This finishes the proof.
    
\end{proof}

\section*{Acknowledgements}

The authors were supported by the DFG under Germany's Excellence Strategy  EXC 2044 -- 390685587, \textit{Mathematics M\"unster: Dynamics - Geometry - Structure}. AG was supported by the DFG priority program SPP 2265 \textit{Random Geometric Systems}.

\end{document}